\documentclass[11pt]{amsart}
\usepackage[usenames,dvipsnames]{pstricks}
\usepackage[mathscr]{eucal}
\usepackage{amsfonts,amsmath,amssymb,amsthm,amscd,amsxtra,cite}
\usepackage{enumerate,verbatim}
\usepackage[all,2cell,ps]{xy}
\usepackage[notcite, notref]{}
\usepackage[pagebackref]{hyperref}
\usepackage{todonotes}
\usepackage{mathptmx}
\usepackage{calc,color}
\usepackage{wasysym}
\theoremstyle{plain} %default

\newtheorem{thm}{Theorem}[section]
\newtheorem{prop}[thm]{Proposition}
\newtheorem{cor}[thm]{Corollary}
\newtheorem{lem}[thm]{Lemma}

\theoremstyle{definition}
\newtheorem{dfn}[thm]{Definition}

\newtheorem{eg}[thm]{Example}

\newtheorem{rmk}[thm]{Remark}
\newtheorem{ex}[thm]{Example}

\newtheorem{defnot}[thm]{Definition and Notation}
\numberwithin{equation}{section}
\newcommand{\fm}{\mathfrak{m}}

\newcommand{\fp}{\mathfrak{p}}
\newcommand{\fq}{\mathfrak{q}}
\newcommand{\fa}{\mathfrak{a}}
\newcommand{\fb}{\mathfrak{b}}
\newcommand{\fc}{\mathfrak{c}}

\newcommand{\C}{\mathrm{c}}

\newtheorem{chunk}[thm]{\hspace*{-1.065ex}\bf}

\DeclareMathOperator{\ann}{Ann} \DeclareMathOperator{\Ass}{Ass}
 
\DeclareMathOperator{\V}{V} \DeclareMathOperator{\hh}{H}
\DeclareMathOperator{\E}{E} \DeclareMathOperator{\Att}{Att}
\DeclareMathOperator{\cm}{CM} \DeclareMathOperator{\X}{X}
\DeclareMathOperator{\Y}{Y}

\DeclareMathOperator{\Sn}{S}
\def\gd{\operatorname{\mathcal{G}-\mathsf{dim}}}
\def\gid{\operatorname{\mathsf{Gid}}}
\def\cid{\operatorname{\mathsf{CI-dim}}}

\def\gcpd{\operatorname{\mathcal{G}_{\it C_\fp}\mathsf{-dim}}}

\def\gkkd{\operatorname{\mathcal{G}_{\it K}\mathsf{-dim}}}
\def\gkbd{\operatorname{\mathcal{G}_{\it\overline{K}}\mathsf{-dim}}}

\def\gkkpd{\operatorname{\mathcal{G}_{\it{K_{\fp}}}\mathsf{-dim}}}
\def\gkkp0d{\operatorname{\mathcal{G}_{\it{K_{\fp_0}}}\mathsf{-dim}}}
\def\gkkP0d{\operatorname{\mathcal{G}_{\it{\widehat{K}_{P_0}}}\mathsf{-dim}}}

\def\gkppd{\operatorname{\mathcal{G}_{\it{K^\prime}}}}
\def\gkkkd{\operatorname{\mathcal{G}_{\it K}}}

\def\pd{\operatorname{\mathsf{pd}}}
\def\gpkd{\operatorname{\mathsf{G}(\mathcal{P}_K)\mathsf{-dim}}}
\def\pkd{\operatorname{\mathcal{P}_K\mathsf{-dim}}}

\def\refnk{\operatorname{\mathsf{ref}_n(K)}}
\def\corefnk{\operatorname{\mathsf{coref}_n(K)}}

\def\gr{\operatorname{\mathsf{grade}}}

\def\Tr{\mathsf{Tr}}
\def\Epi{\mathsf{Epi}}

\def\L{\mathbb{L}}
\def\ak{\mathcal{A}_K}
\def\bk{\mathcal{B}_K}
\def\pk{\mathcal{P}_K}
\def\mp{\mathcal{P}}

\def\ml{\mathcal{L}}
\def\trk{\mathsf{Tr}_{K}}
\def\gkdp{\operatorname{\mathsf{G}_{\it {K}_\fp}\mathsf{-dim}}}

\def\Min{\mathsf{Min}}

\def\mx{\mathcal{X}}
\def\my{\mathcal{Y}}
\def\G{\mathsf{G}}
\def\C{\mathsf{C}}
\def\Dn{\mathcal{D}^n}
\def\DnK{\mathcal{D}^n_K}
\def\P{\mathcal{P}}

\DeclareMathOperator{\coker}{coker}
\DeclareMathOperator{\md}{mod}
\def\depth{\operatorname{\mathsf{depth}}}
 
\def\Ext{\operatorname{\mathsf{Ext}}}

\DeclareMathOperator{\hei}{height}

\def\Hom{\operatorname{\mathsf{Hom}}}

\DeclareMathOperator{\id}{injdim} \DeclareMathOperator{\im}{im}

\DeclareMathOperator{\Supp}{Supp} \DeclareMathOperator{\Spec}{Spec}

\def\Tor{\operatorname{\mathsf{Tor}}}

\DeclareMathOperator{\cod}{codim}

\def\urltilda{\kern -.15em\lower .7ex\hbox{\~{}}\kern .04em}
\def\urldot{\kern -.10em.\kern -.10em}\def\urlhttp{http\kern -.10em\lower -.1ex
	\hbox{:}\kern -.12em\lower 0ex\hbox{/}\kern -.18em\lower
	0ex\hbox{/}}
\begin{document}
\baselineskip=15pt
%\listoftodos{}

\title[Linkage of modules by reflexive morphisms]
 {Linkage of modules by reflexive morphisms}

%%% ----------------------------------------------------------------------
%%% ----------------------------------------------------------------------
%%% ----------------------------------------------------------------------
\bibliographystyle{amsplain}
%%% ----------------------------------------------------------------------
%%% ----------------------------------------------------------------------
%%% ----------------------------------------------------------------------
\author[Dehghani-Zadeh]{Fatemeh Dehghani-Zadeh$^1$}
\author[Dibaei]{Mohammad-T. Dibaei$^{2,3}$}
\author[Sadeghi]{Arash Sadeghi$^3$}
\address{$^{1}$ Department of Mathematics, Islamic Azad University, Yazd Branch. Yazd, Iran.}
\address{$^{2}$ Faculty of Mathematical Sciences and Computer,
	Kharazmi University, Tehran, Iran.}
\address{$^{3}$ School of Mathematics, Institute for Research in Fundamental Sciences (IPM), P.O. Box: 19395-5746, Tehran, Iran }
\email{fdzadeh@gmail.com}
\email{dibaeimt@ipm.ir}
\email{sadeghiarash61@gmail.com}

\keywords{linkage of modules, homological dimensions, local cohomology}
\subjclass[2010]{13C40, 13D05, 13D45, 13C14}
\thanks{Sadeghi's research was supported by a grant from IPM, Iran.}
\thanks{Dibaei thanks IPM, Iran, for providing him office and facilities during this research.}
\thanks{Dehghani-Zadeh thanks the Mosaheb Mathematical Institute, Tehran, where she visited during this research for one year starting October 2017. }
\begin{abstract}	
In this paper, we introduce and study the notion of linkage of modules by reflexive homomorphisms. This notion unifies and generalizes several known concepts of linkage of modules and enables us to study the theory of linkage of modules over Cohen-Macaulay rings rather than the more restrictive Gorenstein rings. It is shown that several known results for Gorenstein linkage are still true in the more general case of module linkage over Cohen-Macaulay rings. We also introduce the notion of colinkage of modules and establish an adjoint equivalence between the linked and colinked modules.
\end{abstract}
\maketitle
%%%%%%%%%%%%%%%%%%%%%%%%%%%%%%%%%%%%
%%%%%%%%%%%%%%%%%%%%%%%%%%%%%%%%%
\section{Introduction}
%%%%%%%%%%%%%%%%%%%%%%%%%%%%%%%%
%%%%%%%%%%%%%%%%%%%%%%%%%%%%%%%%%
The theory of linkage is a classical subject in both commutative algebra and algebraic geometry. Its roots go back to the late 19th and early 20th century, when M. Noether, Halphen and Severi used it to study algebraic curves in $\mathbb{P}^3$. The modern theory of linkage for subschemes of projective space was introduced by Peskine and Szpiro \cite{PS} in 1974.
Since then, it has been studied extensively by several authors
including  Huneke, Rao, Ulrich and many others
\cite{Hu,Hu1,HuUl,HuU,HuU1,HuU2,KMMNP,KM,KM1,KMU,Mi,MiNa,MiNa1,Na1,Ra1,Ra,Ra2}. Recall that 
two ideals $I$ and $J$ in a Gorenstein local ring $R$ are said to be linked if there is a regular sequence ${\bf x}$ in their intersection such that $I=(({\bf x}):J)$ and $J=(({\bf x}):I)$. The notion of linkage by a complete intersection has been extended by some authors to include linkage by more general ideals. For instance, Gorenstein linkage has been studied by Golod \cite{G1}, Schenzel \cite{Sc}, Kustin and Miller \cite{KM}, Nagel \cite{Na1} and Kleppe et al.\cite{KMMNP}; linkage by generically Gorenstein, Cohen-Macaulay ideals has been investigated by Martin \cite{Mar1}.

The classical linkage theory has been extended to modules in different ways by several authors in recent years, for instance Martin \cite{Mar}, Yoshino and Isogawa \cite{YI}, Martsinkovsky and Strooker \cite{MS}, Nagel \cite{Na}, and Iima and Takahashi \cite{IT}. Based on these generalizations, several works have been done on studying the linkage of modules; see for example \cite{BY,CDGST,DGHS,DS,DS1,DS2,N,Sa,Sa1,TJP,TP}. 

Inspired by Nagel's work, we introduce and study the notion of linkage of modules by reflexive homomorphisms. This is a new notion of linkage for modules which includes the concepts of linkage due to Yoshino and Isogawa \cite{YI}, Martsinkovsky and Strooker \cite{MS}, Nagel \cite{Na}, and Iima and Takahashi \cite{IT} (see Remark \ref{r}). Moreover, this notion enables us to study the theory of linkage of modules over Cohen-Macaulay rings rather than the more restrictive Gorenstein rings.

Throughout $R$ denotes a commutative Noetherian ring and $\md R$ denotes the category of all finitely generated $R$-modules. For a semidualizing $R$-module $K$ and an integer $n\geq0$, fix $\mx$ as an $n$-reflexive subcategory of $\md R$ with respect to $K$ (see Definition \ref{def1}).
The subcategory of perfect modules of grade $n$, denoted by $\mp^n(R)$, is an example of $n$-reflexive subcategory with respect to $R$. Also, over a Cohen-Macaulay local ring $R$ with dualizing module $\omega$, the subcategory of Cohen-Macaulay modules of codimension $n$, denoted by $\cm^n(R)$, is an $n$-reflexive subcategory with respect to $\omega$ (see also Example \ref{ex} and Proposition \ref{ccc} for more examples of such subcategories).

We denote by $\Epi(\mx)$ the set of $R$-epimorphisms $\phi:X\twoheadrightarrow M$, where $X\in\mx$ and $M\in\md R$ with $\gr_R(M)=\gr_R(X)=n$. Such homomorphisms are called \emph{reflexive homomorphisms} with respect to $\mx$. Given a reflexive homomorphism $\phi\in\Epi(\mx)$, we construct a new reflexive homomorphism $\L_K^n(\phi)\in\Epi(\mx)$ (see Definition \ref{d2} and Theorem \ref{t1}). Two reflexive homomorphisms $\phi:X\twoheadrightarrow M$ and $\psi:Y\twoheadrightarrow N$ in $\Epi(\mx)$ are said to be \emph{equivalent} and denoted by $\phi\equiv\psi$, provided that there exist isomorphisms $\alpha:M\overset{\cong}{\to} N$ and $\beta:X\overset{\cong}{\to} Y$ such that $\alpha\circ\phi=\psi\circ\beta$ (see Definition \ref{def2}). 

Let $\phi$ and $\psi$ be reflexive homomorphisms in $\Epi(\mx)$. We say that $R$-modules $M$, $N$ are \emph{linked with respect to} $\mx$ in one step by the pair
$(\phi,\psi)$, provided that $M=\im\phi$, $N=\im\psi$, $\phi\equiv\L_K^n(\psi)$ and $\psi\equiv\L_K^n(\phi)$ (see Definition \ref{def4}). 
The $R$-modules $M$ and $N$ are said to be in the same \emph{liaison class with respect to $\mx$} provided that $M$ and $N$ are linked in $m$ steps for some integer $m>0$, i.e. there exist $R$-modules $N_0 = M, N_1, \cdots, N_{m-1},N_m=N$ such that $N_i$ and $N_{i+1}$ are directly linked with respect to $\mx$ for all $i=0,\cdots,m-1$. If $m$ is even, then $M$ and $N$ are said to be \emph{evenly linked}. 

It should be noted that, an ideal $I$ of grade $n$ over a Gorenstein local ring $R$ is linked by a Gorenstein ideal if and only if $R/I$ is linked with respect to the subcategory $\cm^n(R)$, consisting of all Cohen–Macaulay $R$-modules
of codimension $n$, in the sense of the new definition (see Corollary \ref{c5}). More generally, over a Cohen-Macaulay local ring $R$ which is a homomorphic image of a Gorenstein local ring, every unmixed $R$-module of codimension $n$ is linked with respect to $\cm^n(R)$ (see Corollary \ref{c1} and Example \ref{ee} for more examples of linked modules).

%%%%%%%%%%%%%%%%%%%%%%%%%%%%%%%%%%%%%%%%%%%%%%%%%%%%%%%%%%%%%%%%%%%%
One of the first main results in the classical theory of linkage, due to Peskine and Szpiro, indicates that the Cohen-Macaulay property is preserved under linkage over Gorenstein rings. 
Note that this result is no longer true if the base ring is Cohen-Macaulay but not Gorenstein. Attempts to extend this theorem lead to several developments in the theory of linkage (see for example \cite{Hu,HVV}). In our first main result, it is shown that
the Cohen-Macaulay property is preserved under linkage over a Cohen-Macaulay local ring which is a homomorphic image of a Gorenstein local ring. In fact, we show that several results known for Gorenstein linkage are still true
in the more general case of module linkage over Cohen-Macaulay rings. For example,
properties such as being perfect, Cohen-Macaulay, generalized Cohen-Macaulay are preserved in module liaison classes. Also, we show that the homological dimensions and local cohomology modules are preserved in even module liaison classes. 
There is an interesting generalization of the Peskine-Szpiro Theorem to the vanishing of certain local cohomology modules,
due to Schenzel \cite{Sc}. Over a Gorenstein local ring, he studied the relationship between ideals $I$ satisfying Serre's condition $(S_t)$ and the vanishing condition of local cohomology modules of ideals $J$ that are linked to $I$ \cite[Theorem 4.1]{Sc}. Recall that for an $R$-module $M$ the $(S_t)$ locus of $M$, denoted by $S_t(M)$, is the subset of $\Spec R$ consisting of all prime ideals $\fp$ of $R$ such that $M_\fp$ satisfies the Serre's condition $(S_t)$. Inspired by the Schenzel's result, we study the connection of $(S_t)$ locus of a linked module and the set of attached primes of certain local cohomology modules of its linked module. More precisely, we prove the following result (see Corollary \ref{c3}).
%%%%%%%%%%%%%%%%%%%%%%%%%%%%%%%%%%%%%%%%%%%%%%%%%%%%%%%%%%%%%%%%%%
\begin{thm}\label{it}
Let $(R,\fm,k)$ be a Cohen-Macaulay local ring of dimension $d$ which is a homomorphic image of a Gorenstein local ring. Assume that $M$ and $N$ are $R$-modules which are in the same liaison class with respect to $\cm^{n}(R)$. The following statements hold true.
	\begin{enumerate}[\rm(i)]
		\item  $M$ is Cohen-Macaulay if and only if $N$ is so.
        \item If $\X$ is an open subset of $\Spec R$ and $M$, $N$ are linked in an odd number of steps, then
        $$\X\subseteq\Sn_t(M) \text{ if and only if } \Att_R(\hh^j_\fm(N))\subseteq\Spec R\setminus\X \text{ for all } j,\ \dim_R(N)-t<j<\dim_R(N).$$
        \item If $M$ and $N$ are in the same even liaison class, then $\hh^i_{\fm}(M)\cong\hh^i_{\fm}(N)$ for all $i$, $0<i< d-n$.
		\item If $M$ and $N$ are linked in an odd number of steps and $M$ is generalized Cohen-Macaulay, then $$\hh^i_{\fm}(M)\cong\Hom_R(\hh^{d-n-i}_{\fm}(N),\E_R(k)) \text{ for all } i, \  0<i<d-n.$$
		In particular, $N$ is generalized Cohen-Macaulay.
	\end{enumerate}
\end{thm}
%%%%%%%%%%%%%%%%%%%%%%%%%%%%%%%%%%%%%%%%%%%%%%%%%%%%%%%%%%%%%%
Recall that the Cohen-Macaulay locus of an $R$-module $M$, denoted by $\cm_R(M)$, is the subset of $\Spec R$ consisting of all prime ideals $\fp$ of $R$ such that $M_\fp$ is a Cohen-Macaulay $R_\fp$-module.

As a consequence of Theorem \ref{it}, we obtain the following interesting result for self-linked modules.
\begin{cor}\label{self}
	Let $(R,\fm)$ be a Cohen-Macaulay local ring which is a homomorphic image of a Gorenstein local ring and let $\X$ be an open subset of $\Spec R$. Assume that $M$ is an $R$-module of dimension $d$ which is self-linked with respect to $\cm^n(R)$. The following are equivalent.
	\begin{enumerate}[\rm(i)]
		\item $\Att_R(\hh^i_{\fm}(M))\subseteq\Spec R\setminus\X$ for all $i$, $\lfloor d/2\rfloor\leq i<d$.
		\item $\Att_R(\hh^i_{\fm}(M))\subseteq\Spec R\setminus\X$ for all $i$, $0<i<d$.
		\item $\X\subseteq\cm_R(M)$. 
	\end{enumerate} 
\end{cor}
It should be noted that, even if we are restricted to classical linkage theory, i.e. linkage theory for ideals, the above result is new and can be viewed as a generalization of \cite[Proposition 4.3]{Sc}.

%%%%%%%%%%%%%%%%%%%%%%%%%%%%%%%%%%%%%%%%%%%%%%%%%%%%%%%%%%
For a semidualizing $R$-module $K$, we denote by $\mp_K^n(R)$ the subcategory of $\mp_K$-perfect modules of grade $n$ (see Section 5 for more details). Note that, in the trivial case, i.e. $K=R$, we omit the subscript and recover the subcategory of perfect modules of grade $n$, $\mp^n(R)$.
Moreover, we denote by $\ak(R)$ (resp. $\bk(R)$) the Auslander class (resp. Bass class) with respect to $K$ (see Definition \ref{def}).

In the second part of the paper, we introduce and study the notion of colinkage of modules by coreflexive homomorphisms. This can be viewed as a generalization of the notion of the linkage with respect to a semidualizing module \cite{DS2} (see Proposition \ref{ccp}). This notion enables us to study the theory of linkage for modules in the Bass class with respect to $K$.
Note that the notions of linkage and colinkage are the same over a Gorenstein local ring. It is shown that every grade-unmixed module in the Bass class with respect to $K$ can be colinked with respect to the category of $\mp_K$-perfect modules (see Example \ref{cee} for more examples of colinked modules). Several properties such as being $\mp_K$-perfect, Cohen-Macaulay, generalized Cohen-Macaulay are preserved in module coliaison classes (see Theorem \ref{ct} and Corollaries \ref{ccc2} and \ref{ccc3}). We also establish an adjoint equivalence between the linked modules with respect to the category of perfect modules and the colinked modules with respect to the category of $\mp_K$-perfect modules. More precisely, we prove the following (see also Theorem \ref{ctt} for a more general case).  

\begin{thm}\label{icctt}
	Let $R$ be a commutative Noetherian ring and let $K$ be a semidualizing $R$-module. There is an adjoint equivalence
		\begin{center}		
			$\left\{\begin{array}{llll}
			M\in\ak(R){\bigg |} \begin{array}{lll} M\mathrm{\ is\ linked\ with } \\
			\mathrm{\ \ respect\ to } \ \mp^n(R) \\
			\end{array} \end{array}
			\right\}
			\begin{array}{ll}\overset{-\otimes_{R}K}{\xrightarrow{\hspace*{2cm}}}\\ \underset{\Hom_{R}(K,-)}{\xleftarrow{\hspace*{2cm}}}\\ \end{array}
			\left\{\begin{array}{lll}
			N\in\bk(R){\bigg |} \begin{array}{lll} N\mathrm{\ is\ colinked\ with }\\
			\mathrm{\ \ respect\ to }	\ \mp_K^n(R)\\ \end{array}\end{array}
			\right\}.$
		\end{center}	
\end{thm}
%%%%%%%%%%%%%%%%%%%%%%%%%%%%%%%%%%%%%%%%%%%%%%%%%%%%%%%%%%%%%%%%%
The organization of the paper is as follows. In section 2, we collect some  definitions and results which will be used in this paper. In section 3, we introduce the notion of linkage by reflexive homomorphisms and establish its basic properties. We also obtain a criterion for linkage of modules by reflexive homomorphisms. In section 4, we study the $\mathcal{G}_K$-perfect linkage which is a generalization of the notion of Gorenstein linkage. We prove our first main result, Theorem \ref{it},
in this section. Finally, in the last section, we introduce and study the notion of colinkage by coreflexive homomorphisms and give a proof for Theorem \ref{icctt}. 
%%%%%%%%%%%%%%%%%%%%%%%%%%%%%%%%%%%%%%%%%%%%%%%%%%%%%%%%%%%%%%%%%%%
%%%%%%%%%%%%%%%%%%%%%%%%%%%%%%%%%%%%%%%
\section*{Convention}
Throughout the paper $R$ is a commutative Noetherian ring, all modules over $R$ are assumed to be finitely generated, and the category of all such modules is denoted by $\md R$. All subcategories are assumed to be full and closed under isomorphism. If $R$ is local, then the Matlis dual functor is denoted by  $(-)^\vee:=\Hom_R(-,E_R(k))$, where $E_R(k)$ is the injective envelope of the residue field $k$. For an $R$-module $M$, the number $\cod(M):=\dim R-\dim M$ is called the \emph{codimension} of $M$. Over a Cohen-Macaulay local ring $R$, we denote by $\cm^n(R)$ the subcategory of $\md R$ consisting of all Cohen-Macaulay $R$-modules of codimension $n$. The letter $K$ is always used for a semidualizing $R$-module. The $K$-dual functor is denoted by $(-)^{\triangledown}:=\Hom_R(-,K)$. For an integer $n$, we denote by $\mp^n(R)$ (resp. $\G^n_K(\mp)$) the category of perfect ( resp. $\gkkkd$-perfect ) $R$-modules of grade $n$.
%%%%%%%%%%%%%%%%%%%%%%%%%%%%%%%%%%%%
\section{Preliminaries}
%%%%%%%%%%%%%%%%%%%%%%%%%%%%%%%%%%%%%
%%%%%%%%%%%%%%%%%%%%%%%%%%%%%%%%%%%%%
In this section we collect some definitions and results which will be used throughout the paper.

Let $\cdots\to P_i\overset{\partial_{i}}{\to}\cdots\to P_0\overset{\partial_0}{\to}M\to0$ be a projective resolution of an $R$-module $M$. For an integer $i\geq0$, the image of $\partial_{i}$, is called the $i$-\emph{syzygy} of $M$ and denoted by $\Omega^i_RM$ (or simply $\Omega^i M$) which is unique up to projective equivalence. The \emph{transpose} of $M$, denoted by $\Tr M$, is defined to be $\coker \partial_1^*$, where
$(-)^* := \Hom_R(-,R)$, which satisfies the exact
sequence $$0\rightarrow M^*\rightarrow
P_0^{*}\overset{\partial_1^*}{\rightarrow} P_1^{*}\rightarrow \Tr M\rightarrow 0.$$
Note that the transpose of $M$ is unique up to projective equivalence. However, if $M$ has a minimal presentation (e.g. $R$ is semiperfect), then $\Tr M$ is defined uniquely up to isomorphism.
%%%%%%%%%%%%%%%%%%%%%%%%%%%%%%%%%%%%%%%%%%%%%%%%%%%%%%%%%%%%%%%%%%%
%%%%%%%%%%%%%%%%%%%%%%%%%%%%%%%%%%%%%%%%%%%%%%%%%%%%%%%%%%%%%%%%%%%%%%%%%%%%%%%%%%%%%%
\begin{chunk}\textbf{Gorenstein dimension with respect to a semidualizing module.}\label{d3}
The notion of Gorenstein dimension (or $\mathcal{G}$-dimension) for finitely generated modules was introduced by Auslander \cite{A}
and deeply developed by Auslander and Bridger \cite{AB}. It is a refinement of the classical projective dimension
in the sense that they are equal when the latter is finite, but the Gorenstein dimension may be finite without the projective one being so.
A commutative Noetherian local ring is Gorenstein precisely when the Gorenstein dimension of any finitely generated module is finite.

Recall that a finitely generated $R$-module $K$ is \emph{semidualizing} if the homothety morphism $R\to\Hom_R(K,K)$ is an isomorphism and $\Ext^i_R(K,K)=0$ for all $i>0$. Semidualizing modules are initially studied by Foxby \cite{F}, Vasconcelos \cite{V} and Golod \cite{G}. Examples of such modules
include all finitely generated projective modules of rank one and a dualizing module, when one exists. The notion of Gorenstein dimension has been extended to Gorenstein dimension with respect to a semidualizing module by Foxby \cite{F} and Golod \cite{G}.

From now on, we fix $K$ as a semidualizing $R$-module. An $R$-module $M$ is called {\it totally $K$-reflexive} if $M$ is
$K$-reflexive, i.e. the natural evaluation map $M\rightarrow\Hom_R(\Hom_R(M,K),K)$ is bijective and
$$\Ext^i_R(M,K)=0=\Ext^i_R(\Hom_R(M,K),K), \text{  for all }\ i>0.$$
A $\gkkkd$-resolution of an $R$-module $M$ is a right acyclic complex
of totally $K$-reflexive modules whose $0$th homology is $M$. The minimum lengths of all $\gkkkd$-resolutions of $M$ is denoted by
$\gkkd_R(M)$. Note that, over a local ring $R$, a semidualizing $R$-module $K$ is
a dualizing module if and only if $\gkkd_R(M)<\infty$ for all
finitely generated $R$-modules $M$ (see \cite[Proposition
1.3]{Ge}).

The Auslander's transpose has been generalized by Foxby \cite{F} as follows.
Let $\pi: P_1\overset{\partial}{\rightarrow}P_0\rightarrow M\rightarrow 0$ be a
projective presentation of an $R$-module $M$. The \emph{transpose of $M$ with respect to} $K$, denoted by $\trk^{\pi} M$, is defined to be $\coker \partial^{\triangledown}$, where
$(-)^{\triangledown} := \Hom_R(-,K)$, which satisfies the following exact
sequence
\begin{equation}\tag{\ref{d3}.1}
0\rightarrow M^{\triangledown}\rightarrow
P_0^{\triangledown}\overset{\partial^{\triangledown}}{\rightarrow} P_1^{\triangledown}\rightarrow \trk^{\pi}
M\rightarrow 0.
\end{equation}
Denote  $\lambda_K^{\pi}M:=\im(\partial^{\triangledown})$.
Hence one has the exact sequences
\begin{equation}\tag{\ref{d3}.2}
0\to M^{\triangledown}\to P_0^{\triangledown}\to\lambda_K^{\pi}M\to0 \text{ and } 
0\to\lambda_K^{\pi}M\to P_1^{\triangledown}\to\Tr^{\pi}_KM\to0.
\end{equation}
Note that the transpose of $M$ with respect to $K$ depends on the choice of the projective presentation
of $M$, but it is unique up to $K$-projective equivalence. In other words, if $T$ and $T'$ are transposes of $M$ with respect to $K$, then there exist projective modules $P$ and $Q$ such that $T\oplus(P\otimes_RK)\cong T'\oplus(Q\otimes_RK)$.
Hence, unless otherwise specified, we simply denote the transpose of $M$ with respect to $K$ by $\Tr_{K}M$. 

For an $R$-module $M$, there exists the following exact
sequence
\begin{equation}\tag{\ref{d3}.3}
0\rightarrow\Ext^1_R(\trk M,K)\rightarrow M\rightarrow
M^{\triangledown\triangledown}\rightarrow\Ext^2_R(\trk
M,K)\rightarrow0,
\end{equation}
where $M\to M^{\triangledown\triangledown}$ is the natural evaluation map (see \cite[Proposition 3.1]{F}).
\end{chunk}
In the following, we collect some basic properties about $\gkkkd$-dimension which will be used throughout the paper (see \cite{AB}, \cite{G} for more details).
\begin{thm}\label{G3}
Let $M$ be an $R$-module. The following statements hold true.
\begin{enumerate}[\rm(i)]
	\item If $\gkkd_R(M)<\infty$, then $\gkkd_R(M)=\sup\{i\mid\Ext^i_R(M,K)\neq0\}$.
	\item $\gkkd_R(M)=0$ if and only if $\gkkd_R(\trk M)=0$.
	\item If $R$ is local and $\gkkd_R(M)<\infty$, then $\gkkd_R(M)=\depth R-\depth_R(M)$.
\end{enumerate}
\end{thm}

%%%%%%%%%%%%%%%%%%%%%%%%%%%%%%%%%%%%%%%%%%%%%%%%%%%%%%%%%%%%%%%%%%%%%%%%%%%%%%%
\begin{defnot}\label{perfect}
An $R$-module $M$ is called $\gkkkd$-{\it perfect} if $\gr(M)=\gkkd(M)$. Note that if $K$ is a dualizing module, then the category of $\gkkkd$-perfect modules coincides with the category of Cohen-Macaulay modules.
An $R$-module
$C$ is called $\gkkkd$-{\it Gorenstein} provided that $C$ is $\gkkkd$-perfect and
$\Ext^{n}_{R}(C, K) \cong C$, where $n=\gkkd(M)$. 
An ideal $I$ is called $\gkkkd$-perfect (resp. $\gkkkd$-Gorenstein) if $R/I$ is 
$\gkkkd$-perfect (resp. $\gkkkd$-Gorenstein) as an $R$-module.
For an integer $n$, we denote by $\G^n_K(\mathcal{P})$ (resp. $\G^n_K(\mathcal{G})$) the subcategory of $\md R$ consisting of all $\gkkkd$-perfect (resp. $\gkkkd$-Gorenstein) $R$-modules of grade $n$.	
\end{defnot}

 Here are some basic properties of $\gkkkd$-perfect modules.
	\begin{lem}\emph{\cite{G}}\label{ll1} Let $M$ be a $\gkkkd$-perfect $R$-module of grade $n$. Then the following statements hold true.
	\begin{enumerate}[\rm(i)]
		\item $\Ext^n_R(M,K)\in\G^n_K(\mathcal{P})$.
		\item  $M\cong\Ext^n_R(\Ext^n_R(M,K),K)$.
		\item $\ann_R(M)=\ann_R(\Ext^n_R(M,K))$.
	\end{enumerate}
\end{lem}
%%%%%%%%%%%%%%%%%%%%%%%%%%%%%%%%%%%%%%%%%%%%%%%%%%%%%%%%%%%%%%%%%%%%%
%%%%%%%%%%%%%%%%%%%%%%%%%%%%%%%%%%%%%%%%%%%%%%%%%%%%%%%%%%%%%%%%%%
%%%%%%%%%%%%%%%%%%%%%%%%%%%%%%%%%%%%%%%%%%%%%%%%%%%%%%%%%%%%%%%%%%
We recall the following results of Golod.
%%%%%%%%%%%%%%%%%%%%%%%%%%%%%%%%%%%%%%%%%%%%%%%%%%%%%%%%%%%%
\begin{lem}\label{G1} \emph{\cite[Corollary]{G}}
	Let $I$ be an ideal of $R$ and let $n$ be an integer. Assume that $M$ is an $R$-module such that $\Ext^j_R(R/I,M)=0$ for
	all $j\neq n$. Then there is an isomorphism of functors $\Ext^i_{R/I}(-,\Ext^n_R(R/I,M))\cong\Ext^{n+i}_R(-,M)$ on the category of $R/I$-modules for all $i\geq0$.
\end{lem}
%%%%%%%%%%%%%%%%%%%%%%%%%%%%%%%%%%%%%%%%%%%%%%%%%%%%%%%%%%%%%%%%%%%%%%%5
\begin{thm}\label{G2}
	\emph{\cite[Proposition 5]{G}} Let $I$ be a $\gkkkd$-perfect ideal of grade $n$. Set $\overline{K}=\Ext^{n}_R(R/I,K)$. Then the following statements hold true.
	\begin{enumerate}[\rm(i)]
		\item $\overline{K}$ is a semidualizing $R/I$-module.
		\item For an $R/I$-module $M$, one has $\gkkd_R(M)<\infty$ if and only if $\gkbd_{R/I}(M)<\infty$, and
			if these dimensions are finite, then
			$\gkkd_R(M)=\gr(I)+\gkbd_{R/I}(M)$.
	\end{enumerate}
\end{thm}
%%%%%%%%%%%%%%%%%%%%%%%%%%%%%%%%%%%%%%%%%%%%%%%%%%%%%%%%%%%%%%%%%%%%%%%
An $R$-module $M$ is said to be \emph{grade-unmixed} if $\gr_R(M)=\depth R_\fp$ for all $\fp\in\Ass_R(M)$. Note that every $\gkkkd$-perfect module is grade-unmixed (see for example \cite[Proposition 1.4.16]{BH}). 
An $R$-module $M$ is said to be \emph{unmixed} if all its associated prime ideals have the same height. Hence, an $R$-module over a Cohen-Macaulay local ring $R$ is grade-unmixed if and only if it is unmixed.
%%%%%%%%%%%%%%%%%%%%%%%%%%%%%%%%%%%%%%%%%%%%%%%%%%%%%%%%%%%%%%%%%%%%%%%%%%%%
\begin{lem}\label{l}
	Let $M$ be an $R$-module of grade $n$. Then $\Ext^n_R(M,K)$ is grade-unmixed. In particular, $\gr_R(\Ext^n_R(M,K))=n$.
\end{lem}
\begin{proof}
	Let ${\bf x}$ be an ideal of $R$, generated by a regular sequence of length $n$, contained in $\ann_R(M)$. Set $S=R/{\bf x}$ and $\overline{K}=\Ext^n_R(S,K)$. By Theorem \ref{G2}, $\overline{K}$ is a semidualizing $S$-module. Note that, by Lemma \ref{G1}, $\Ext^n_R(M,K)\cong\Hom_S(M,\overline{K})$. Let $\fp\in\Ass_R(\Ext^n_R(M,K))$ and set $\overline{\fp}=\fp/{\bf x}$. It follows from \cite[Proposition 9.A]{Ma1} that
	$\overline{\fp}\in\Ass_S(\Hom_S(M,\overline{K}))\subseteq\Ass_S(\overline{K})\subseteq\Ass S$. Therefore
	$\depth R_\fp=\depth S_{\overline{\fp}}+n=n$ for all 
	$\fp\in\Ass_R(\Ext^n_R(M,K))$.	Hence $\Ext^n_R(M,K)$ is grade-unmixed and $\gr_R(\Ext^n_R(M,K))=n$ by \cite[Corollary 4.6]{AB}.
\end{proof}

%%%%%%%%%%%%%%%%%%%%%%%%%%%%%%%%%%%%%%%%%%%%%%%%%%%%%%%%%%%%%%%%%%%%%%%%%%%%
For an $R$-module $M$ of grade $n$, there is a natural homomorphism $\eta_K^R(M):M\rightarrow\Ext^n_R(\Ext^n_R(M,K),K),$
which is a generalization of the natural evaluation map $M\to\Hom_R(\Hom_R(M,K),K)$. This homomorphism was studied by 
several authors, including Roos \cite{Roo}, Fossum \cite{Fos} and Foxby \cite{F}.
\begin{lem}\label{l5}
Let
$M$ be an $R$-module of grade $n$. The following statements hold true.
\begin{enumerate}[\rm(i)]
	\item There exists the exact sequence
	 $$0\to\Ext^{n+1}_R(\Tr_{K}\Omega^nM,K)\to M\overset{\eta_K^R(M)}{\longrightarrow}\Ext^n_R(\Ext^n_R(M,K),K)\to\Ext^{n+2}_R(\Tr_{K}\Omega^{n}M,K)\to0.$$
	\item Let $I\subseteq\ann_R(M)$ be a $\gkkkd$-perfect ideal of grade $n$. Set $S=R/I$ and $\overline{K}=\Ext^n_R(R/I,K)$. Then
    $$\Ext^{i}_S(\Tr_{\overline{K}}M,\overline{K})\cong\Ext^{i+n-2}_R(\Ext^n_R(M,K),K)$$
    for all $i>2$. Moreover,	$\ker\eta_K^R(M)\cong\Ext^1_S(\Tr_{\overline{K}}M,\overline{K})$ and $\coker\eta_K^R(M)\cong\Ext^2_S(\Tr_{\overline{K}}M,\overline{K})$.
\end{enumerate}	
\end{lem}
\begin{proof}
Part (i) is a consequence of \cite[Proposition 3.4]{F}.
	
(ii). First note that, by Theorem \ref{G2}(i), $\overline{K}$ is a semidualizing $S$-module. By Lemma \ref{G1} and \cite[Proposition 3.1]{F}, one has the commutative diagram
$$\begin{array}{llllllll}
0\longrightarrow\ \ \ \ \ \ker\eta_K^R(M) &\longrightarrow &M &\overset{\eta_K^R(M)}{\longrightarrow}&\Ext^n_R(\Ext^n_R(M,K),K)&\longrightarrow \ \ \ \coker\eta_K^R(M)&\longrightarrow 0  \\
&&\downarrow\shortparallel  &&\downarrow {\cong}\\
 0\longrightarrow\Ext^1_S(\Tr_{\overline{K}}M,\overline{K})&\longrightarrow &M &\longrightarrow&\Hom_S(\Hom_S(M,\overline{K}),\overline{K})&\longrightarrow \Ext^2_S(\Tr_{\overline{K}}M,\overline{K})&\longrightarrow 0 \\
\end{array}$$\\
with exact rows. Hence we obtain the following isomorphisms
\begin{equation}\tag{\ref{l5}.1}
\ker\eta_{K}^R(M)\cong\Ext^1_S(\Tr_{\overline{K}}M,\overline{K}) \text{ and } \coker\eta_{K}^R(M)\cong\Ext^{2}_R(\Tr_{\overline{K}}M,\overline{K}).
\end{equation}
Let $P_1\to P_0\to M\to0$ be an $S$-projective presentation of $M$. Dualizing with respect to $\overline{K}$, we get the exact sequence
$0\to\Hom_S(M,\overline{K})\to\Hom_S(P_0,\overline{K})\to\Hom_S(P_1,\overline{K})\to\Tr_{\overline{K}}M\to0.$
Breaking into short exact sequences and using the fact that $\Ext_S^{i>0}(\overline{K}, \overline{K})=0$ and Lemma \ref{G1}, one obtains
\[\begin{array}{rl}\tag{\ref{l5}.2}
\Ext^{i+2}_S(\Tr_{\overline{K}}M,\overline{K})&\cong\Ext^{i}_S(\Hom_S(M,\overline{K}),\overline{K})\\ &\cong{\Ext^{n+i}_R(\Hom_S(M, \overline{K}),K)}\\ &{\cong\Ext^{n+i}_R(\Ext^n_R(M, K), K)},
\end{array}\]
for all $i>0$. 
\end{proof}
For an integer $n$, we set $\X^n(R)=\{\fp\in\Spec R\mid \depth R_\fp\leq n\}$. Let $M$ be an $R$-module. We say that $M$ has finite $\gkkkd$-dimension on $\X^{n}(R)$, provided that $\gkdp_{R_\fp}(M_\fp)<\infty$ for all $\fp\in\X^{n}(R)$.
For a positive integer $n$, an $R$-module $M$ is called an $n$-$K$-{\it syzygy} if there exists an exact sequence of $R$-modules
$0\longrightarrow M\longrightarrow P_{0}\otimes_RK\longrightarrow\cdots\longrightarrow P_{n-1}\otimes_RK,$
where all $P_i$'s are projective $R$-modules. This property was investigated by Bass \cite{B} for modules over Gorenstein rings and was extensively studied by Auslander and Bridger \cite{AB} for modules of finite Gorenstein dimension. The following is a generalization of \cite[Theorem 4.25]{AB} and \cite[Theorem 43]{M1}.

\begin{thm}\label{th}
Let $n$ be a
positive integer, $M$ an $R$-module. Consider the statements
\begin{enumerate}[\rm(i)]
	\item $\Ext^i_R(\trk M,K)=0$ for all $i$, $1\leq i\leq n$;
	\item $M$ is an $n$-$K$-syzygy module;
	\item $\depth_{R_\fp}(M_\fp)\geq\min\{n,\depth R_\fp\}$ for all $\fp\in\Spec R$.
\end{enumerate}
Then the following implications hold true.
\begin{itemize}
	\item[(a)] (i)$\Rightarrow$(ii)$\Rightarrow$(iii).
	\item[(b)] If $M$ has finite $\gkkkd$-dimension on $\X^{n-2}(R)$, then (ii)$\Rightarrow$(i).
	\item[(c)]If $M$ has finite $\gkkkd$-dimension on   $\X^{n-1}(R)$, then (iii)$\Rightarrow$(i).
\end{itemize}
\end{thm}
\begin{proof}
As parts (a) and (c) follow from \cite[Proposition 2.4]{DS1}, we only need to prove part (b). We argue by induction on $n$. Assume that $n=1$ and so $M$ is a first $K$-syzygy module. We obtain the commutative diagram
$$\begin{CD}
\ \ &&&& 0@>>> M @>>> C&  \\
&&&&&& @VV{\eta_K^R(M)}V @VV{\eta_K^R(C)}V \\
\ \  &&&&&&M^{\triangledown\triangledown}@>>>C^{\triangledown\triangledown}, &\\
\end{CD}$$\\
where $C=P_0\otimes_RK$ for some projective module $P_0$. Note that $C$ is totally $K$-reflexive. Hence, by definition, the map $\eta_K^R(C)$ is an isomorphism which implies that $\eta_K^R(M)$ is injective. Now the assertion follows from Lemma \ref{l5}(i). Assume that $n>1$.
Hence, by the induction hypothesis, $\Ext^i_R(\trk M,K)=0$ for $1\leq i\leq n-1$.
There exists an exact sequence $0\to M\overset{f}{\to} C\to N\to0$ such that  $C=P\otimes_R K$ for some projective $R$-module $P$ and $N$ is an $(n-1)$-$K$-syzygy module. Hence we get the exact sequence
\begin{equation}\tag{\ref{th}.1}
0\to N^\triangledown\to C^\triangledown\overset{f^{\triangledown}}{\to} M^\triangledown\to\trk N\to\trk C\to\trk M\to0,
\end{equation}
(see for example \cite[Lemma 2.2]{DS1}).
As $C$ is totally reflexive, $\Ext^1_R(C,K)=0$ and so $\coker(f^{\triangledown})\cong\Ext^1_R(N,K)$. Therefore we obtain the following exact sequences, induced by (\ref{th}.1)
\begin{equation}\tag{\ref{th}.2}
0\to\Ext^1_R(N,K)\to\trk N\to Z\to0,
\end{equation}
\begin{equation}\tag{\ref{th}.3}
0\to Z\to\trk C\to \trk M\to0.
\end{equation}

Next we show that $\gr_R(\Ext^1_R(N,K))\geq n-1$. Assume contrarily that
$\gr_R(\Ext^1_R(N,K))=\depth R_\fp$ for some $\fp\in\X^{n-2}(R)$. As $N$ is an $(n-1)$-$K$-syzygy, $\depth_{R_\fp}(N_\fp)\geq\min\{n-1,\depth R_\fp\}=\depth R_\fp$. Note that $N$ has finite $\gkkkd$-dimension on $\X^{n-2}(R)$. Hence by Theorem \ref{G3}(iii) $\gkkpd_{R_\fp}(N_\fp)=0$ and so $\Ext^1_R(N,K)_\fp=0$ which is a contradiction.
Therefore $\gr_R(\Ext^1_R(N,K))\geq n-1$. In other words, $\Ext^i_R(\Ext^1_R(N,K),K)=0$ for $i<n-1$. Applying the functor $(-)^{\triangledown}$ to the exact sequence (\ref{th}.2), we get the following long exact sequence
\begin{equation}\tag{\ref{th}.4}
\cdots\to\Ext^{n-2}_R(\Ext^1_R(N,K),K)\to\Ext^{n-1}_R(Z,K)\to\Ext^{n-1}_R(\trk N,K)\to\cdots.
\end{equation}
By the induction hypothesis, $\Ext^i_R(\trk N,K)=0$ for $i= n-1$ so that, by (\ref{th}.4), $\Ext^{n-1}_R(Z,K)=0$. 
As $C$ is totally $K$-reflexive, by Theorem \ref{G3}, $\Ext^i_R(\trk C,K)=0$ for $i>0$. Hence,
applying the functor $(-)^{\triangledown}$ to the exact sequence (\ref{th}.3), we get $\Ext^{n}_R(\trk M,K)=0$ as desired.
\end{proof}
%%%%%%%%%%%%%%%%%%%%%%%%%%%%%%%%%%%%%%%%%%%%%%%%%%%%%%%%%%%%%%%%%%%%%%%%%%%%%%%%%%
In the following, we give a necessary and sufficient condition for $\eta_K^R(M)$ to be an isomorphism.
\begin{cor}\label{l4}
For an $R$-module $M$ of grade $n$ and an integer $t>1$, the following statements hold true.
\begin{enumerate}[\rm(i)]
	\item  If $M$ has finite $\gkkkd$-dimension on $X^n(R)$, then  $\depth_{R_\fp}(M_\fp)\geq\min\{1,\depth R_\fp-n\}$ for all $\fp\in\Spec R$ if and only if $\eta_K^R(M)$ is a monomorphism.
	
	\item If $M$ has finite $\gkkkd$-dimension on $\X^{n+t-1}(R)$, then $\depth_{R_\fp}(M_\fp)\geq\min\{t,\depth R_\fp-n\}$ for all $\fp\in\Spec R$ if and only if $\eta_K^R(M)$ is an isomorphism and $\Ext^i_R(\Ext^n_R(M,K),K)=0$ for all $i$, $n+1\leq i\leq n+t-2$.
\end{enumerate}
\end{cor}
\begin{proof}
Let ${\bf x}\subseteq\ann_R(M)$ be an ideal generated by an $R$-sequence of length $n$. Hence ${\bf x}$ is a $\gkkkd$-perfect ideal of grade $n$.
Set $S=R/{\bf x}$ and $K'=\Ext^n_R(S,K)$. Note that, for an integer $m>0$, $\depth_{R_\fp}(M_\fp)\geq\min\{m,\depth R_\fp-n\}$	for all $\fp\in\Spec R$ if and only if $\depth_{S_\fq}(M_\fq)\geq\min\{m,\depth S_\fq\}$ for all $\fq\in\Spec S$.
Also, by Theorem \ref{G2}, $M$ has finite $\gkkkd$-dimension on $\X^{m}(R)$ for some $m\geq n$ if and only if it has finite $\gkppd$-dimension on $\X^{m-n}(S)$.
Now the assertion follows from Lemma \ref{l5} and Theroem \ref{th}.
\end{proof}
%%%%%%%%%%%%%%%%%%%%%%%%%%%%%%%%%%%%%%%%%%%%%%%%%%%%%%%%%%%%%%%%%%%%%%%%%%%%%
%%%%%%%%%%%%%%%%%%%%%%%%%%%%%%%%%%%%%%%%%%%%%%%%%%%%%%%%%%%%%%%%%%%%%%%%%%%%%%
%%%%%%%%%%%%%%%%%%%%%%%%%%%%%%%%%%%%%%%%%%%%%%%%%%%%%%%%%%%%%%
%%%%%%%%%%%%%%%%%%%%%%%%%%%%%%%%%%%%%%%%%%%%%%%%%%%%%%%%%%%%%
\section{Definition and basic properties of linkage}
%%%%%%%%%%%%%%%%%%%%%%%%%%%%%%%%%%%%%%%%%%%%%%%%%%%%%%%%%%%%%
%%%%%%%%%%%%%%%%%%%%%%%%%%%%%%%%%%%%%%%%%%%%%%%%%%%%%%%%%%%%%%%
In this section, we define the notion of linkage by a reflexive homomorphism. Throughout $K$ is a semidualizing $R$-module and $n$ is a non-negative integer. We start by introducing some notations.

\begin{defnot}\label{def1}
Set $\refnk$ as the subcategory of $\md R$ consisting of all $R$-modules $M$ with grade $n$ such that the natural map $\eta_K^R(M):M\to\Ext^n_R(\Ext^n_R(M,K),K)$ is an isomorphism.

A subcategory $\mx\subseteq\refnk$ is called $n$-{\em reflexive subcategory with respect to $K$} if it is closed under $\Ext^n_R(-,K)$ (i.e. for all $M\in\mx$, we have $\Ext^n_R(M,K)\in\mx$).	
\end{defnot}

Note that $\mp^n(R)$, the subcategory of $\md R$ consisting of all perfect $R$-modules of grade $n$, is a classical example of an $n$-reflexive subcategory with respect to $R$. Clearly, $\refnk$ itself is an $n$-reflexive subcategory with respect to $K$. The following provides more examples of $n$-reflexive subcategories with respect to $K$ (see Definition \ref{perfect} and Lemma \ref{ll1}).
%%%%%%%%%%%%%%%%%%%%%%%%%%%%%%%%%%%%%%%%%%%%%%%%%%%%%%%%%%%%%%%
\begin{eg}\label{ex}
The following subcategories of $\md R$ are $n$-reflexive with respect to $K$.
\begin{enumerate}[(i)]
   \item The subcategory $\G^n_K(\mathcal{P})$ of $\md R$, consisting of all $\gkkkd$-perfect modules of grade $n$.
	\item The subcategory $\G^n_K(\mathcal{G})$ of $\md R$, consisting of all $\gkkkd$-Gorenstein modules of grade $n$.
	\item The subcategory $\cm^n(R)$ of $\md R$, consisting of all Cohen-Macaulay modules of codimension $n$, where $R$ is a Cohen-Macaulay local ring and $K$ is a dualizing module.
\end{enumerate}	
\end{eg}
%%%%%%%%%%%%%%%%%%%%%%%%%%%%%%%%%%%%%%%%%%%%%%%%%%%%%%%
%%%%%%%%%%%%%%%%%%%%%%%%%%%%%%%%%%%%%%%%%%%%%%%%%%%%%%%
Here is another example of an $n$-reflexive subcategory with respect to $K$.
\begin{prop}\label{ccc}
    The subcategory $\mx$ of $\md R$ consisting of all $R$-modules $M$ of grade $n$ such that
	\begin{enumerate}[\rm(i)]
		\item $\gkdp_{R_\fp}(M_\fp)<\infty$ for all $\fp\in\X^{n+1}(R)$ (e.g. $\id_{R_\fp}(K_\fp)<\infty$ for all $\fp\in\X^{n+1}(R)$),
		\item $\depth_{R_\fp}(M_\fp)\geq\min\{2,\depth R_\fp-n\}$, for all $\fp\in\Spec R$.
	\end{enumerate}
	is an $n$-reflexive subcategory with respect to $K$.
\end{prop}
\begin{proof}
	By Corollary \ref{l4}(ii), $\mx\subseteq\refnk$. Let $M\in\mx$. We show that $\Ext^n_R(M,K)\in\mx$. By Lemma \ref{l}, $\gr_R(\Ext^n_R(M,K))=n$. Now we prove that $\Ext^n_R(M,K)$ satisfies condition (i).
	It follows from our assumptions and Theorem \ref{G3}(iii) that
	\begin{equation}\tag{\ref{ccc}.1}
	n\leq\gr_{R_\fp}(M_\fp)\leq\gkdp_{R_\fp}(M_\fp)=\depth R_\fp-\depth_{R_\fp}(M_\fp)\leq n,
	\end{equation}
	for all $\fp\in\X^{n+1}(R)$. In other words, for all $\fp\in\X^{n+1}(R)$, $M_\fp$ is a $\mathcal{G}_{K_\fp}$-perfect $R_\fp$-module of grade $n$ and so, by Lemma \ref{ll1}(i), $\Ext^n_R(M,K)_\fp$ is a $\mathcal{G}_{K_\fp}$-perfect $R_\fp$-module of grade $n$. In particular,  $\gkdp_{R_\fp}(\Ext^n_R(M,K)_\fp)<\infty$.
  Now, we show that $\Ext^n_R(M,K)$ satisfies in the condition (ii). Choose an ideal ${\bf x}\subseteq\ann_R(M)$, generated by an $R$-sequence of length $n$. Set $S=R/{\bf x}$ and $K'=\Ext^n_R(S,K)$. Note that, by Theorem \ref{G2}, $K'$ is a semidualizing $S$-module.
  Let $P_1\to P_0\to M\to0$ be a $S$-projective presentation of $M$. Dualizing with respect to $K'$ one gets the following exact sequence
  \begin{equation}\tag{\ref{ccc}.2}
  0\to\Hom_S(M,K')\to\Hom_S(P_0,K')\to\Hom_S(P_1,K').
  \end{equation}
  By Lemma \ref{G1}, $\Ext^n_R(M,K)\cong\Hom_S(M,K')$. Hence, it follows from (\ref{ccc}.2) that $\Ext^n_R(M,K)$ is a second $K'$-syzygy $S$-module. Therefore,
   $\depth_{S_\fp}(\Ext^n_R(M,K)_\fp)\geq\min\{2,\depth S_\fp\}$, for all $\fp\in\Spec S$. This is equivalent to what we want to achieve.
\end{proof}
%%%%%%%%%%%%%%%%%%%%%%%%%%%%%%%%%%%%%%%%%%%%%%%%%%%%%%%%
For a subcategory $\mx$ of $\md R$, let $\Epi(\mx)$ to be the set of $R$-epimorphisms $\phi:X\twoheadrightarrow M$, where  $X\in\mx$ and $M\in\md R$ with $\gr_R(M)=\gr_R(X)$. Whenever $\mx$ is an $n$-reflexive subcategory with respect to $K$, we call $\phi$ as a \emph{reflexive homomorphism} (with respect to $\mx$).

Given a homomorphism $\phi\in\Epi(\mx)$ we want to construct a new reflexive homomorphism $\L_K^n(\phi)$.

\begin{dfn}\label{d2}
	Let $\mx\subseteq\md (R)$ be an $n$-reflexive subcategory with respect to $K$ and let $\phi\in\Epi(\mx)$. Consider the exact sequence $0\rightarrow\ker\phi\overset{j}{\rightarrow} X\rightarrow\im\phi \rightarrow 0$, where $X\in\mx$. Applying the functor $(-)^{\triangledown}=\Hom_R(-,K)$ and using the fact that $\gr_R(\ker\phi)=n$, imply the exact sequence
	$$0\rightarrow\Ext^{n}_R(\im\phi, K)\rightarrow\Ext^{n}_{R}(X, K)\rightarrow \Ext^{n}_R(\ker\phi, K)\rightarrow\Ext^{n+1}_R(\im\phi, K)\to\Ext^{n+1}_R(X, K).$$ 	
	Define $\L_K^n(\phi):\Ext^{n}_{R}(X, K)\twoheadrightarrow\im(\Ext^n_R(j, K))$ as the epimorphism induced by $\Ext^n_R(j, K)$. Therefore one has exact sequences
	\begin{equation}\tag{\ref{d2}.1}
	0\rightarrow\Ext^{n}_R(\im\phi, K)\rightarrow\Ext^n_R(X,K)  \rightarrow\im \L_K^n(\phi)\rightarrow 0,
	\end{equation}
	\begin{equation}\tag{\ref{d2}.2}
	0\rightarrow\im\L_K^n(\phi)\to\Ext^{n}_R(\ker\phi, K)\rightarrow\Ext^{n+1}_R(\im\phi,K)  \rightarrow\Ext^{n+1}_R(X,K).
	\end{equation}	
\end{dfn}
%%%%%%%%%%%%%%%%%%%%%%%%%%%%%%%%%%%%%%%%%%%%%%%%%%%%%%%%%%%%%%%%%%%%%%%%%%%%%%%%%%%%
We summarize some basic properties of $\L_K^n(\phi)$ in the following theorem which is a generalization of \cite[Proposition 3.4]{Na}.
\begin{thm}\label{t1}
	Let $\mx$ be an $n$-reflexive subcategory with respect to $K$ and let $\phi\in\Epi(\mx)$ be a homomorphism
	which is not injective. Then the following statements hold true.
	\begin{enumerate}[\rm(i)]
		\item $\L_K^n(\phi)\in\Epi(\mx)$.
		\item $\im\L_K^n(\phi)$ is a grade-unmixed $R$-module.
		\item The image of the canonical map $\eta_K^R(\im\phi)$ is isomorphic to
		$\im\L_K^n(\L_K^n(\phi))$. In particular, there exists the exact sequence
		$0\to\Ext^{n+1}_R(\Tr_K\Omega^n\im\phi,K)\to \im\phi\to\im\L_K^n(\L_K^n(\phi))\to0.$
		\item If $\eta_K^R({\im\phi})$ is injective, then
		$\ker\phi\cong\Ext^n_R(\im\L_K^n(\phi),K)$.
	\end{enumerate}	
\end{thm}
\begin{proof}
	(i) and (ii). First note that $\ker\phi$ is a non-zero module of grade $n$.
	Hence, by Lemma \ref{l}, $\Ext^n_R(\ker\phi,K)$ is grade-unmixed of grade $n$. It follows from the exact sequence (\ref{d2}.2) that $\depth R_\fp=n$ for all
	$\fp\in\Ass_R(\im\L_K^n(\phi))$. Hence $\im\L_K^n(\phi)$ is also a grade-unmixed module of grade $n$. As $\mx$ is $n$-reflexive with respect to $K$, $\Ext^n_R(X,K)\in\mx$ for all $X\in\mx$ and so $\L_K^n(\phi)\in\Epi(\mx)$.	
	
	(iii). Consider the exact sequence $0\to\ker\phi\to X\overset{\phi}{\to} \im\phi\to0$, where  $X\in\mx$. Applying the functor $(-)^{\triangledown}:=\Hom_R(-,K)$, implies the exact sequence
	$0\to\Ext^n_R(\im\phi,K)\to\Ext^n_R(X,K)\to\im\L_K^n(\phi)\to0$, which induces the following exact sequence
	$$0\to\Ext^n_R(\im\L_K^n(\phi),K)\to \Ext^n_R(\Ext^n_R(X,K),K)\overset{\phi''}{\to}\Ext^n_R(\Ext^n_R(\im\phi,K),K),$$
	where $\phi''=\Ext^n_R(\Ext^n_R(\phi,K),K)$. Moreover, one has the commutative diagram
	$$\begin{CD}
	&&&&&&&&\\
	\ \ &&&& X @>\phi>> \im\phi &  \\
	&&&& @VV{\eta_K^R(X)}V @VV{\eta_K^R(\im\phi)}V \\
	\ \  &&&&\Ext^n_R(\Ext^n_R(X,K),K)@>\phi''>>\Ext^n_R(\Ext^n_R(\im\phi,K),K).&\\
	\end{CD}$$\\
	As $X\in\mx$, we have $\eta_K^R(X)$ is an isomorphism. It follows from the above commutative diagram that $\im(\eta_K^R(\im\phi))=\im(\phi'')=\im(\L_K^n(\L_K^n(\phi)))$. Now the assertion follows from Lemma \ref{l5}(i).
	
	(iv). By part (iii), $\im\phi\cong\im(\L_K^n(\L_K^n(\phi)))$ and we have the commutative diagram
	$$\begin{CD}
	&&&&&&&&\\
	\ \ &&&& 0@>>>\ker\phi @>>>X @>\phi>> \im\phi@>>>0 &  \\
	&&&&&&&&@VV{\cong}V @VV{\cong}V \\
	\ \  &&&&0@>>>\Ext^n_R(\im\L_K^n(\phi),K)@>>>\Ext^n_R(\Ext^n_R(X,K),K)@>>>\im(\L_K^n(\L_K^n(\phi)))@>>>0,&\\
	\end{CD}$$\\
with exact rows. Now, by the commutativity of the above diagram, the assertion follows.
\end{proof}
%%%%%%%%%%%%%%%%%%%%%%%%%%%%%%%%%%%%%%%%%%%%%%%%%%%%%%%%
\begin{dfn}\label{def2}
	Let $\mx$ be a subcategory of $\md R$. Two epimorphisms $\phi:X\twoheadrightarrow M, \psi:Y\twoheadrightarrow N$ in $\Epi(\mx)$ are said to be \em{equivalent} and denoted by $\phi\equiv\psi$, provided that there exist isomorphisms $\alpha:M\overset{\cong}{\to} N$ and $\beta:X\overset{\cong}{\to} Y$ such that the diagram
	$$\begin{CD}
	%&&&&&&&&\\
	\ \ &&&& X @>\phi>> M &  \\
	&&&& @VV{\beta}V @VV{\alpha}V \\
	\ \  &&&& Y@>\psi>>N &\\
	\end{CD}$$\\
	is commutative. It is easy to see that $\equiv$ is an equivalence relation in $\Epi(\mx)$.
\end{dfn}
%%%%%%%%%%%%%%%%%%%%%%%%%%%%%%%%%%%%%%%%%%%%%%%%%%%%%%%
\begin{lem}\label{lll}
	Let $\mx$ be an $n$-reflexive subcategory with respect to $K$ and let $\phi, \psi\in\Epi(\mx)$. Then the following statements hold true.
	\begin{enumerate}[\rm(i)]
		\item{If $\phi\equiv\psi$, then
			$\L_K^n(\phi)\equiv\L_K^n(\psi)$. In particular, $\im\L_K^n(\phi)\cong\im\L_K^n(\psi)$.}
		\item{$\phi\equiv\L_K^n(\L_K^n(\phi))$ if and only if
			there exists $\mu\in\Epi(\mx)$ such that
			 $\phi\equiv\L_K^n(\mu)$ and $\mu\equiv\L_K^n(\phi)$.}
	\end{enumerate}
\end{lem}
\begin{proof}
	(i). Let $\phi:X\to M$ and $\psi:Y\to N$, where $X,Y\in\mx$. There exist isomorphisms $\alpha:M\overset{\cong}{\to} N$ and $\beta:X\overset{\cong}{\to} Y$ such that $\psi\circ\beta=\alpha\circ\phi$. Hence, there is a map $\gamma:\ker\phi\to\ker\psi$ making the diagram 
	$$\begin{CD}
	&&&&&&&&\\
	\ \ &&&&0@>>>\ker\phi@>>> X @>\phi>> M@>>>0 &  \\
	&&&&&&@VV{\gamma}V @VV{\beta}V @VV{\alpha}V \\
	\ \  &&&&0@>>>\ker\psi@>>> Y@>\psi>>N@>>>0.&\\
	\end{CD}$$\\
	commutative. Now, by the snake lemma, $\gamma$ is an isomorphism. Applying the functor $\Ext^n_R(-,K)$ to the above commutative diagram, gives the commutative diagram (see (\ref{d2}.1))
	$$\begin{CD}
	&&&&&&&&\\
	\ \  &&&&0@>>>\Ext^n_R(N,K)@>>>\Ext^n_R(Y,K)@>\L_K^n(\psi)>>\im\L_K^n(\psi)@>>>0.&\\
	&&&&&&@VV{\alpha'}V @VV{\beta'}V @VV{\gamma'}V \\
	\ \ &&&&0@>>>\Ext^n_R(M,K)@>>>\Ext^n_R(X,K) @>\L_K^n(\phi)>>\im\L_K^n(\phi)@>>>0 &  \\
	\end{CD}$$\\
	As $\alpha'=\Ext^n_R(\alpha,K)$, $\beta'=\Ext^n_R(\beta,K)$ are isomorphisms, so is $\gamma'$. Hence $\L_K^n(\phi)\equiv\L_K^n(\psi)$.
	
	(ii). follows from part (i).
\end{proof}
%%%%%%%%%%%%%%%%%%%%%%%%%%%%%%%%%%%%%%%%%%%%%%%%%%%%%%%
We are now in the position to introduce the notion of linkage of modules by reflexive morphisms.
\begin{dfn}\label{def4}
	Let $\mx$ be an $n$-reflexive subcategory with respect to $K$ and let $\phi,\psi\in\Epi(\mx)$. Two $R$-modules
	$M$, $N$ are said to be \emph{linked with respect to }$\mx$, in one step (or directly), by the pair
	$(\phi,\psi)$ provided that
	the following conditions hold.
	\begin{enumerate}[\rm(i)]
		\item $M=\im\phi$ and $N=\im\psi$.
		\item $\phi\equiv\L_K^n(\psi)$ and $\psi\equiv\L_K^n(\phi)$.
	\end{enumerate}
  In this situation, we write $M\underset{(\phi,\psi)}{\sim}N$ (or simply $M\sim N$).
	Equivalently, an $R$-module $M$ is said to be \emph{linked} by a reflexive homomorphism $\phi$, if $M=\im\phi$ and $M\cong\im\L_K^n(\L_K^n(\phi))$ (see Lemma \ref{lll}). 
\end{dfn}
Notice that the above definition is symmetric in the following sense:
$M\underset{(\phi,\psi)}{\sim} N$ if and only if $N\underset{(\psi,\phi)}{\sim} M$.

%%%%%%%%%%%%%%%%%%%%%%%%%%%%%%%%%%%%%%%%%%%%%%%%%%%%%%%%%%%%%%%%%
Here is a characterization of a linked module in terms of the homomorphism $\eta_K^R(-)$. The following should be compared with \cite[Theorem 2]{MS}.
\begin{cor}\label{c}
	Let $\mx$ be an $n$-reflexive subcategory with respect to $K$ and let $M$ be an $R$-module. Assume that $\phi\in\Epi(\mx)$
	is a non-injective homomorphism such that $\im(\phi)=M$. The following statements are equivalent.
	\begin{enumerate}[\rm(i)]
		\item $M$ is linked by $\phi$.
		\item $\Ext^{n+1}_R(\trk\Omega^nM,K)=0$.
		\item $\eta_K^R(M)$ is injective.
	\end{enumerate}
\end{cor}
\begin{proof}
This is an immediate consequence of Theorem \ref{t1}.(iii) and \cite[Theorem 2.4]{Ma}.
\end{proof}
%%%%%%%%%%%%%%%%%%%%%%%%%%%%%%%%%%%%%%%%%%%%%%%%%%%%%%%%%%%%%%%%%%%
\begin{cor}\label{c1}
	Let $\mx$ be an $n$-reflexive subcategory with respect to $K$ and let $M$ be an $R$-module. Assume that $\phi$ is a non-injective homomorphism in $\Epi(\mx)$ such that $\im(\phi)=M$.
	If $M$ has finite $\gkkkd$-dimension on $\X^{n}(R)$ (e.g. $\id_{R_\fp}(K_\fp)<\infty$ for all $\fp\in\X^{n}(R)$), then the following statements are equivalent.
	\begin{enumerate}[\rm(i)]
		\item $M$ is linked by $\phi$.
		\item $M$ is grade-unmixed.
		\item $\depth_{R_\fp}(M_\fp)\geq\min\{1,\depth R_\fp-n\}$ for all $\fp\in\Spec R$.
	\end{enumerate}
\end{cor}
\begin{proof}
	(i)$\Rightarrow$(ii). As $M$ is linked by $\phi$,  $M\cong\im\L_K^n(\L_K^n(\phi)))$. Hence $M$ is grade-unmixed by Theorem \ref{t1}(i), (ii).
	
	(ii)$\Rightarrow$(iii). Let $\depth_{R_\fp}(M_\fp)=0$ for some $\fp\in\Spec R$. Hence $\fp\in\Ass_R(M)$ and so $\depth R_\fp=n$, because $M$ is grade-unmixed.
	
	(iii)$\Rightarrow$(i). The assertion follows from Corollary \ref{l4}(i), Lemma \ref{l5}(i) and Corollary \ref{c}.
\end{proof}	
%%%%%%%%%%%%%%%%%%%%%%%%%%%%%%%%%%%%%%%%%%%%%%%%%%%%%%%%
%%%%%%%%%%%%%%%%%%%%%%%%%%%%%%%%%%%%%%%%%%%%%%%%%%%%%%
The following is an immediate consequence of Corollary \ref{c1} and
\cite[Remarks 0.1, 0.2]{Hu}.
\begin{cor}
	Let $R$	be a Gorenstein local ring and let $I$ be an ideal of grade $n$. Assume that ${\bf x}\subsetneq I$ is an ideal generated by a regular sequence of length $n$ and that $\phi:R/{\bf x}\twoheadrightarrow R/I$ is the natural epimorphism.
	Then  $R/I$ is linked by $\phi$ with respect to $\mp^n(R)$ if and only if $I=({\bf x}:_R({\bf x}:_RI))$.
\end{cor}
%%%%%%%%%%%%%%%%%%%%%%%%%%%%%%%%%%%%%%%%%%%%%%%%%%%%%%%%%%%
In the following we collect some examples of linked modules (see also Example \ref{e}).
\begin{eg}\label{ee}
	\begin{enumerate}[\rm(i)]
		\item Let $\mx$ be an $n$-reflexive subcategory with respect to $K$ which is closed under direct sum (i.e. if $M, N\in\mx$, then $M\oplus N\in\mx$). Every two modules in $\mx$ are directly linked with respect to $\mx$. This is an immediate consequence of the exact sequence
		$$0\rightarrow \Ext^n_R(N,K)\rightarrow M \oplus\Ext^n_R(N,K) \rightarrow M\rightarrow0$$ and Corollary \ref{c}. In particular, $M$ is directly linked to $\overset{t}{\oplus}M$ for every $M\in\mx$ and an integer $t>0$. Also, every $M\in\mx$ is self-linked with respect to $\mx$.
		For example, every $M\in\refnk$ is self-linked with respect to $\refnk$.
		\item Let $M$ be a grade-unmixed $R$-module of finite $\gkkkd$-dimension with grade $n$. Assume that $I\subseteq\ann_R(M)$ is an ideal of $R$ such that $R/I\in\refnk$.
		If $\phi:F\twoheadrightarrow M$ is a non-injective epimorphism where $F$ is a free $R/I$-module, then $M$ is linked by $\phi$ with respect to $\refnk$. This is an immediate consequence of Corollary \ref{c1}.
		\item
		Let $R_\fp$ be Gorenstein for all $\fp\in\X^{n}(R)$ and let $M$ be a grade-unmixed $R$-module. Assume that $\gr_R(M)=n$ and that $\fc\subseteq\ann_R(M)$ is an ideal generated by a regular sequence of length $n$. If $\phi:F\twoheadrightarrow M$ is a non-injective epimorphism where $F$ is a free $R/\fc$-module, then $M$ is linked by $\phi$ with respect to $\mathcal{P}^n(R)$. This is an immediate consequence of Corollary \ref{c1}.
	\end{enumerate}	
\end{eg}
%%%%%%%%%%%%%%%%%%%%%%%%%%%%%%%%%%%%%%%%%%%%%%%%%%%%%%%%%%%%%%
%%%%%%%%%%%%%%%%%%%%%%%%%%%%%%%%%%%%%%%%%%%%%%%%%%%%%%%%%%%%%%%%%%55
\begin{rmk}\label{r} 
There are several notions of module linkage in the literature that have been
developed independently. Yoshino and Isogawa \cite{YI} introduced the notion of linkage for Cohen-Macaulay modules over a Gorenstein local ring. Martsinkovsky and Strooker \cite{MS} generalized the notion of linkage for modules over semiperfect rings by using the operator $\lambda_R:=\Omega_R\Tr_R$. This notion includes the concept of linkage due to Yoshino and Isogawa (see \cite[Remark on page 594]{MS}). Here we observe that the notion of the linkage of modules by reflexive homomorphisms includes the notions of linkage due to Martsinkovsky and Strooker \cite{MS}, Nagel \cite{Na}, and Iima and Takahashi \cite{IT}.
\begin{enumerate}[\rm(I)]
	\item\textbf{Horizontal linkage} (Martsinkovsky-Strooker).\label{d1}
	Let $R$ be a semiperfect ring. Two $R$-modules $M$ and $N$ are said to be \emph{horizontally linked} if $M\cong\lambda_R N$ and $N\cong\lambda_R M$.  An $R$-module $M$ is horizontally linked (to $\lambda_R M$) if and only if it has no projective summands and $\Ext^1_R(\Tr M,R)=0$ \cite[Theorem 2]{MS}. 	
	Two $R$-modules $M$ and $N$ are said to be \emph{linked by an ideal} $\fc$, if $\fc\subseteq\ann_R(M)\cap\ann_R(N)$ and $M$, $N$ are horizontally linked as $R/\fc$-modules. 
		
	Let $R$ be a semiperfect ring and let $M$ be an $R$-module. Assume that $P\overset{\phi}{\to} M\to0$ is a projective cover of $M$. It is clear that $M$ is horizontally linked, in the sense of Martsinkovsky-Strooker definition, if and only if $M$ is linked by $\phi$ with respect to $\mp^0(R)$.
	
	%%%%%%%%%%%%%%%%%%%%%%%%%%%%%%%%%
	\item \textbf{Linkage of modules by a quasi-Gorenstein module} (Nagel).\label{d}
	Let $R$ be a Gorenstein local ring. An $R$-module $M$ is said to be \emph{quasi-Gorenstein}
	if it is perfect and
	there is an isomorphism $M\overset{\simeq}{\to}\Ext^n_R(M,R)$, where $n=\pd_R(M)$ (see also \cite[Definition 2.1]{Na} for the graded version). 
	
	Let $C$ be a quasi-Gorenstein $R$-module of codimension $n$. Following \cite{Na}, the set of $R$-homomorphisms
	$\phi:C\to M$ where $M$ is an $R$-module and $\im\phi$ has the same dimension as $C$ is denoted by $\Epi(C)$. For $\phi\in\Epi(C)$, the exact sequence $0\to\ker\phi\to C\to\im\phi\to0$ induces the long exact sequence  $0\to\Ext^n_R(\im\phi,R)\to\Ext^n_R(C,R)\overset{f}{\to}\Ext^n_R(\ker\phi,R)\to\cdots.$
	By assumption there is an isomorphism $\alpha:C\overset{\simeq}{\to}\Ext^n_R(C,R)$. Hence we obtain the homomorphism
	$L_C(\phi):= f\circ\alpha:C\to\Ext^n_R(\ker\phi,R)$. Now we are ready to recall the definition of linkage of modules by a quasi-Gorenstein module (see \cite[Definition 3.7]{Na}).
	
	Let $R$ be a Gorenstein local ring.	
	Two $R$-modules $M$, $N$ are said to be \emph{linked in one step} by the quasi-Gorenstein module $C$ if there are homomorphisms $\phi, \psi\in\Epi(C)$ such that
	\begin{enumerate}[\rm(i)]
		\item $M=\im\phi$ and $N=\im\psi$.
		\item $M\cong\im L_C(\psi)$ and $N\cong\im L_C(\phi)$.
	\end{enumerate}
	Note that, over a Gorenstein local ring, this notion includes the concept of linkage due to Martsinkovsky and Strooker (see \cite[Remark 3.20]{Na} for more details).
	
	Let $R$ be a Gorenstein local ring and let $C$ be a quasi-Gorenstein $R$-module. It follows from the definition that $\mx=\{C\}$ is an $n$-reflexive subcategory with respect to $R$, where $n=\gr_R(C)$. Assume that $M$ is an $R$-module and that $\phi\in\Epi(C)$ is an $R$-homomorphism with $\im\phi=M$. Now it is easy to see that $M$ is linked by $C$, in the sense of Nagel's definition, if and only if it is linked by $\phi$ with respect to $\mx$.
	\item \textbf{Perfect linkage} (Iima-Takahashi).\label{d4}
	Let $R$ be a Cohen-Macaulay local ring with dualizing module $\omega$ and let $M$ be a Cohen-Macaulay $R$-module of codimension $n$. A surjective homomorphism $f:P\twoheadrightarrow M$ is called a \emph{perfect morphism} of $M$ if $P$ is a perfect R-module of codimension $n$. For a perfect
	morphism $f$ of $M$, we define the perfect link $\L_fM$ of $M$ with respect to $f$ by $\L_fM =\Ext^n_R(\ker f,\omega)$.
	Note that $\L_fM$ is again a Cohen-Macaulay $R$-module of codimension $n$, and there exists a perfect morphism
	$g$ of $\L_fM$ such that $\L_g\L_fM\cong M$. Two Cohen-Macaulay $R$-modules $M$ and $N$ of codimension $n$ are called \emph{perfectly linked} if there exists a perfect morphism $f$ of $M$ such that $\L_fM\cong N$.	
	This notion includes the concept of linkage due to Yoshino and Isogawa.
	
	Let $R$ be a Cohen-Macaulay local ring with dualizing module $\omega$ and let $M$ be a Cohen-Macaulay $R$-module of codimension $n$. Assume that $P\overset{\phi}{\to} M\to0$ is a perfect morphism of $M$. It is clear that $M$ is perfectly linked, in the sense of Iima-Takahashi definition, if and only if it is linked by $\phi$ with respect to $\cm^n(R)$.
\end{enumerate}	
\end{rmk}
%%%%%%%%%%%%%%%%%%%%%%%%%%%%%%%%%%%%%%%%%%%%%%%%%%%%%%%%%%%%%%%%%%%%%%%%%
\begin{dfn}
	Let $m>0$ be an integer. Two $R$-modules $M$ and $N$ are said to be linked in $m$ steps if there are modules $N_0 = M, N_1, \cdots, N_{m-1},N_m=N$ such that $N_i$ and $N_{i+1}$ are directly linked for all $i=0,\cdots,m-1$. If $m$ is even, then $M$ and $N$ are said to be {\it evenly linked}.
	{\it Module liaison} is the equivalence relation generated by directly linkage. Its equivalence classes are called liaison classes. More precisely,
	we say that $R$-modules $M$ and $N$ belong to the same \emph{liaison class}  if they are linked in some positive number of steps.
	Even linkage also generates an equivalence relation. Its equivalence classes are called \emph{even liaison classes}.
\end{dfn}
%%%%%%%%%%%%%%%%%%%%%%%%%%%%%%%%%%%%%%%%%%%%%%%%%%%%%%%%
In the following, we recall the definitions of two important classes of modules related to a fixed semidualizing $R$-module, due to Foxby \cite{F} (see also \cite{AvF}, \cite{Ch}).
\begin{dfn}(Foxby Classes)\label{def}	
	\begin{enumerate}[(a)]
		\item The {\em Auslander class
			with respect to} $K$, denoted by $\mathcal{A}_{K}(R)$, consists of all
		$R$-modules $M$ satisfying the conditions.
		\begin{enumerate}[\rm(i)]
			\item  The natural map $\mu_K^R(M):M\longrightarrow\Hom_R(K,M\otimes_RK)$ is an isomorphism.
			\item $\Tor_i^R(M,K)=0=\Ext^i_R(K,M\otimes_RK)$ for all $i>0$.
		\end{enumerate}		
		\item The {\em Bass class with respect to} $K$, denoted by $\bk(R)$, consists of all $R$-modules
		$M$ satisfying the conditions.
		\begin{enumerate}[\rm(i)]
			\item   The evaluation map $\nu_K^R(M):K\otimes_R\Hom_R(K,M)\longrightarrow M$ is an isomorphism.
			\item  $\Tor_i^R(\Hom_R(K,M),K)=0=\Ext^i_R(K,M)$ for all $i>0$.
		\end{enumerate}
	\end{enumerate}
\end{dfn}
Note that the Auslander class $\ak(R)$ (resp. Bass class $\bk(R)$) contains every $R$-module of finite projective (resp. injective) dimension.
%%%%%%%%%%%%%%%%%%%%%%%%%%%%%%%%%%%%%%%%%%%%%%%%%%%%%%%%%%%%%
Recall that a subcategory $\mx$ of $\md R$ is called {\em thick} 
if two terms of any exact sequence $0\to X\to  Y\to Z\to 0$ in $\md R$ are in $\mx$, then so is the third one.
Here are some examples of thick subcategories.
\begin{eg}\label{exaa}
	Let $N$ be an $R$-module. The following subcategories of $\md R$ are thick.
	\begin{enumerate}[(i)]
		\item The subcategory of $\md R$ consisting of all modules of finite $\gkkkd$-dimension.
		\item The subcategory of $\md R$ consisting of all modules of finite projective (injective) dimension.
		\item $\{M\in\md R\mid\Ext^i_R(M,N)=0 \text{ for } i\gg0\}$.
		\item $\{M\in\md R\mid\Ext^i_R(N,M)=0 \text{ for } i\gg0\}$.
		\item $\{M\in\md R\mid\Tor_i^R(M,N)=0 \text{ for } i\gg0\}$.
		\item $\{M\in\md R\mid M\in\mathcal{A}_K(R)\}.$
		\item $\{M\in\md R\mid M\in\mathcal{B}_K(R)\}.$
	\end{enumerate}
\end{eg}
\begin{prop}\label{p}
	Let $\mx$ be an $n$-reflexive subcategory. The following statements hold true.
	\begin{enumerate}[\rm(i)]
		\item Let $M$ and $N$ be $R$-modules in the same liaison class with respect to $\mx$. If $\mx$ is a thick subcategory, then	$ M\in \mathcal{X} \text{ if and only if } N\in\mathcal{X}.$
		\item Assume that $\mathcal{Y}$ is a thick subcategory of $\md R$ containing $\mx$. Let $M$ and $N$ be $R$-modules in the same even liaison class with respect to $\mx$. Then	$M\in \mathcal{Y}$ if and only if $N\in\mathcal{Y}$.
	\end{enumerate}
\end{prop}
\begin{proof}
	(i). Without loss of generality one may assume that $M$ is directly linked to $N$. Hence there exists $\phi\in\Epi(\mx)$ such that $N=\im(\phi)$ and $M\cong\im\L_K^n(\phi)$. By Theorem \ref{t1}(iv) and Corollary \ref{c}, we obtain the exact sequence $0\to\Ext^n_R(M,K)\to X\to N\to0,$
	where $X\in\mx$. Suppose that $M\in\mathcal{X}$. Hence $\Ext^n_R(M,K)\in\mx$. It follows from the above exact sequence that $N\in\mathcal{X}$, because $\mathcal{X}$ is thick. The other side follows from the symmetry.	
	
	(ii). Without loss of generality we may assume that $M$ is linked to $N$ in two steps. Hence there exists an $R$-module $C$ such that $M\sim C\sim N$. In other words, there exist $\phi, \psi\in\Epi(\mx)$ such that $M=\im(\phi)$,
	$N=\im(\psi)$ and $\im\L_K^n(\psi)\cong C\cong\im\L_K^n(\phi)$. By Theorem \ref{t1}(iv) and Corollary \ref{c}, we have the exact sequences
	$0\to\Ext^n_R(C,K)\to X\to M\to0$ and $0\to\Ext^n_R(C,K)\to Y\to N\to0,$
	where $X,Y\in\mx\subseteq\mathcal{Y}$. As $\mathcal{Y}$ is thick, it follows from the above exact sequences that
$M\in\mathcal{Y}\Longleftrightarrow\Ext^n_R(C,K)\in\mathcal{Y}\Longleftrightarrow N\in\mathcal{Y}.$
\end{proof}
%%%%%%%%%%%%%%%%%%%%%%%%%%%%%%%%%%%%%%%%%%%%%%%%%
\begin{cor}\label{ccc1}
	Let $M$ and $N$ be $R$-modules in the same even liaison class with respect to $\mathcal{P}^n(R)$. Assume that $X$ is an $R$-module. Then the following statements hold true.
	\begin{enumerate}[\rm(i)]
		\item $M\in\ak(R)$ if and only if $N\in\ak(R)$.
		\item $\pd_R(M)<\infty$ if and only if $\pd_R(N)<\infty$.
		\item $\Ext^{i\gg0}_R(M,X)=0$ if and only if $\Ext^{i\gg0}_R(N,X)=0$.
		\item $\Tor_{i\gg0}^R(M,X)=0$ if and only if $\Tor_{i\gg0}^R(N,X)=0$.
		\item If $\gd_R(X)<\infty$, then $\Ext^{i\gg0}_R(X,M)=0$ if and only if $\Ext^{i\gg0}_R(X,N)=0.$
		\item $\gkkd_R(M)<\infty$ if and only if $\gkkd_R(N)<\infty$.
		\item $\cid_R(M)<\infty$ if and only if $\cid_R(N)<\infty$.
	\end{enumerate}
\end{cor}
\begin{proof}
	Note that every module of finite projective dimension satisfies in all of the above conditions (see \cite[Theorem 4.13]{AB} for the case (v) and \cite[Theorem 1.4]{AGP} for the case (vii)). Hence the assertions (i)-(vi) follow from Proposition \ref{p}(ii) and Example \ref{exaa}. For the proof of part (vii), we may assume that $M$ is linked to $N$ in two steps. Hence there exists an $R$-module $C$ such that $M\sim C\sim N$. As we have seen in the proof of Proposition \ref{p}, there exist the exact sequences $0\to\Ext^n_R(C,K)\to X\to M\to0$ and $0\to\Ext^n_R(C,K)\to Y\to N\to0$
	where $X,Y\in\mathcal{P}^n(R)$. Now the assertion follows from \cite[Lemma 3.6]{Wa}.
\end{proof}
%%%%%%%%%%%%%%%%%%%%%%%%%%%%%%%%%%%%%%%%%%%%%%%%%%%%%%%%%%%%%%%%%%%%%
%%%%%%%%%%%%%%%%%%%%%%%%%%%%%%%%%%%%%%%%%%%%%%%%%%%%%%%%
%%%%%%%%%%%%%%%%%%%%%%%%%%%%%%%%%%%%%%%%%%%%%%%%%%%%%%%%%%%%%%%%%%
An ideal $\fa$ is said to be {\it grade-unmixed} provided $R/\fa$ is a grade-unmixed $R$-module. In the following we consider the linkage for cyclic modules and compare it with the classical linkage.
%%%%%%%%%%%%%%%%%%%%%%%%%%%%%%%%%%%%%%%%%%%%%%%%%%%%%%%%%%%%%%%%%%%%
\begin{thm}\label{t}
	Let $I$ be a grade-unmixed ideal of finite $\gkkkd$-dimension and let $\fc\subsetneq I$ be a $\gkkkd$-perfect ideal with $\gr(\fc)=n=\gr(I)$. Set $\overline{K}=\Ext^n_R(R/\fc,K)$ and $\overline{I}=I/\fc$.	
	Then $R/I$ is directly linked to $\overline{K}/(0:_{\overline{K}}\overline{I})$ with respect to $\G^n_K(\mathcal{P})$.
\end{thm}
\begin{proof}	
	Applying the functor $\Hom_R(-,K)$ to the exact sequence $0\rightarrow I/\fc\rightarrow R/\fc\overset{\phi}{\rightarrow} R/I\rightarrow0$ implies 
	the exact sequence
	$0\longrightarrow\Ext^n_R(R/I,K)\longrightarrow\Ext^n_R(R/\fc,K)\overset{\L_K^n(\phi)}{\longrightarrow}\im\L_K^n(\phi)\longrightarrow0$.
	Note that, by Corollary \ref{c1}, $R/I$ is linked to $\im\L_K^n(\phi)$ by $\phi$. Now from Lemma \ref{G1}, we obtain the commutative diagram
	$$\begin{CD}
	&&&&&&&&\\
	\ \ &&&& 0@>>>\Ext^n_R(R/I,K) @>>> \Ext^n_R(R/\fc,K) @>>>\im \L_K^n(\phi) @>>>0&  \\
	&&&&&& @VV{\cong}V @VV{\cong}V \\
	\ \  &&&& 0@>>>\Hom_{R/\fc}(R/I,\overline{K}) @>>> \Hom_{R/\fc}(R/\fc,\overline{K}) @>f>>\Hom_{R/\fc}(\overline{I},\overline{K}). &\\
	\end{CD}$$\\
	It follows that $\im\L_K^n(\phi)\cong\im f$. Finally, using the commutative diagram
	$$\begin{CD}
	&&&&&&&&\\
	\ \ &&&& 0@>>>\Hom_{R/\fc}(R/I,\overline{K}) @>>> \Hom_{R/\fc}(R/\fc,\overline{K}) @>>>\im f@>>>0&  \\
	&&&&&& @VV{\cong}V @VV{\cong}V \\
	\ \  &&&& 0@>>>(0:_{\overline{K}}\overline{I}) @>>> \overline{K} @>>>\overline{K}/(0:_{\overline{K}}\overline{I}) @>>>0. &\\
	\end{CD}$$\\
	we get the isomorphism $\im f\cong\overline{K}/(0:_{\overline{K}}\overline{I})$.
\end{proof}
As an immediate consequence of Theorem \ref{t} we have the following result.
%%%%%%%%%%%%%%%%%%%%%%%%%%%%%%%%%%%%%%%%%%%%%%%%
\begin{cor}\label{c5}
	Let $R$	be a Cohen-Macaulay local ring with dualizing  module $\omega_R$ and let $I$ be an unmixed ideal. Assume that $\fc\subseteq I$ is a Cohen-Macaulay ideal with $\gr(\fc)=n=\gr(I)$. Set $S=R/\fc$, $\omega_S=\Ext^n_R(S,\omega_R)$ and $\overline{I}=I/\fc$.
	Then  $R/I$ is directly linked to $\omega_S/(0:_{\omega_S}\overline{I})$ with respect to $\cm^n(R)$. In particular, if $\fc$ is a Gorenstein ideal, then $R/I$ is directly linked to $R/(\fc:I)$ with respect to $\cm^n(R)$.
\end{cor}
Let $R$ be a Gorenstein local ring and let $I$ be an ideal of grade $n$.
It follows from Corollary \ref{c5} that $I$ is linked by a Gorenstein ideal if and only if $R/I$ is linked with respect to $\cm^n(R)$ in the sense of the new definition.
%%%%%%%%%%%%%%%%%%%%%%%%%%%%%%%%%%%%%%%%%%%%%%%%%%%%%%
The following is a generalization of \cite[Proposition 6]{MS} and \cite[Corollary 4.2]{Na}.
\begin{prop}\label{p2}
Let $\mx$ be an $n$-reflexive subcategory with respect to $K$ and let $M$, $N$ be $R$-modules that are directly linked with respect to $\mx$. The following statements hold true.
\begin{enumerate}[\rm(i)]
\item $\Ass_R(M)\cup\Ass_R(N)\subseteq\{\fp\in\Spec R \mid\depth R_\fp = n\}$.
\item If $R$ satisfies the Serre's conditions $(S_{n+1})$ and $\mx\subseteq\G^n_K(\mp)$, then
$$\Ass_R(M)\cup\Ass_R(N)=\{\fp\in\Supp_R(M)\cup\Supp_R(N)\mid\depth R_\fp = n\}.$$
\end{enumerate}
\end{prop}
\begin{proof} 
(i). Let $\fp\in\Ass_R(M)$ and let ${\bf x}$ be an ideal generated by a regular sequence of length $n$ contained in $\ann_R(M)$. Set $S=R/{\bf x}$, 
$\overline{K}=K\otimes_RS$ and $\overline{\fp}=\fp/{\bf x}$ so that $\overline{K}$ is a semidualizing $S$-module. It follows from Lemma
\ref{l5} and Corollary \ref{c} that 
$\Ext^1_S(\Tr_{\overline{K}}M,S)=0$. In other words, the natural homomorphism
$M\rightarrow\Hom_S(\Hom_S(M,\overline{K}),\overline{K})$ is injective (see  (\ref{d3}.3)). It follows from \cite[Proposition 9.A]{Ma1} that
$$\overline{\fp}\in\Ass_S(M)\subseteq\Ass_S(\Hom_S(\Hom_S(M,\overline{K}),\overline{K}))\subseteq \Ass_S(\overline{K})=\Ass S.$$
Therefore $\depth R_\fp-n=\depth S_{\overline{\fp}}=0$. By symmetry, $\Ass_R(N)\subseteq\{\fp\in\Spec R\mid\depth R_\fp=n\}$.

(ii). By Theorem \ref{t1}(iv)
and Corollary \ref{c}, one gets the exact sequence
\begin{equation}\tag{\ref{p2}.1}
0\to\Ext^n_R(M,K)\to X \to N\to 0,
\end{equation}
where $X\in\mx$. Let $\fp\in\Supp_R(M)\cup\Supp_R(N)$ such that $\depth R_\fp =n$. It follows from the exact sequence
(\ref{p2}.1) that $\fp\in\Supp_R(X)$. As $X\in\G^n_K(\mp)$, we have $n=\gkkpd_{R_\fp}(X_\fp)=\depth R_\fp-\depth_{R_\fp}(X_\fp)$ by Theorem \ref{G3}(iii). Therefore $\depth_{R_\fp}(X_\fp)=0$ and so $\fp\in\Ass_R(X)\subseteq\Ass_R(N)\cup\Ass_R(\Ext^n_R(M,K))$, where the inclusion follows from the exact sequence (\ref{p2}.1). If $\fp\in\Ass_R(N)$, then we have nothing to prove so let 
$\fp\in\Ass_R(\Ext^n_R(M,K))$. As $R$ satisfies $(S_{n+1})$ and $\gr_R(M)=n$, it is clear that $\fp\in\Min_R(M)\subseteq\Ass_R(M)$. Now the assertion follows from part (i). 
\end{proof}
%%%%%%%%%%%%%%%%%%%%%%%%%%%%%%%%%%%%%%%%%%%%%%%%%%%%%%%%%%%%%%%%%%%%%%%%%%%%%%%%%%%%%%%%%%%%%%%%%%%
\section{$\mathcal{G}_K$-perfect linkage}
%%%%%%%%%%%%%%%%%%%%%%%%%%%%%%%%%%%%%%%%%%%%
%%%%%%%%%%%%%%%%%%%%%%%%%%%%%%%%%%%%%%%%%%%%%%%%%%%%
In this section, we investigate the $\mathcal{G}_K$-perfect linkage which is a generalization of Gorenstein linkage. 
It is shown that several results known for Gorenstein linkage are still true in the more general case of $\mathcal{G}_K$-perfect linkage.
We prove our first main result, Theorem \ref{it}), in this section. Throughout this section, 
unless otherwise specified,  $\mx_n=\G_K^n(\mathcal{P})$ is the subcategory of $\md R$ consisting of all $\gkkkd$-perfect $R$-modules of grade $n$.
Note that if $K$ is a dualizing module, then $\mx_n=\cm^n(R)$, the category of Cohen-Macaulay $R$-modules of codimension $n$.
We first review some facts about the $\mathcal{G}_K$-perfect linkage.
\begin{eg}\label{e}
	The following statements hold true.
	\begin{enumerate}[(i)]
		\item {\it Every two $\gkkkd$-perfect modules are linked directly. In particular, $M$ is directly linked to $\overset{t}{\oplus}M$ for every $\gkkkd$-perfect module $M$ and an integer $t\geq0$.} 
		
		See Example \ref{ee}(i).
		\item {\it Let $M$ be a grade-unmixed $R$-module of finite $\gkkkd$-dimension. Assume that $I\subseteq\ann_R(M)$ is a $\gkkkd$-perfect ideal of $R$ such that $\gr(I)=n=\gr(M)$. If $\phi:F\twoheadrightarrow M$
		is a non-injective epimorphism where $F$ is an free $R/I$-module, then $M$ is linked by $\phi$ with respect to $\mx_n$.} 
	
	This is an immediate consequence of Corollary \ref{c1} and Example \ref{ex}(i).
		\item {\it Let $R$ be a Cohen-Macaulay local ring with dualizing module $\omega$ and let $M$ be an unmixed $R$-module. Assume that $I\subseteq\ann_R(M)$ is a Cohen-Macaulay ideal with the same grade $n$. If $\phi:F\twoheadrightarrow M$
		is a non-injective epimorphism where $F$ is a free $R/I$-module, then $M$ is linked by $\phi$ with respect to $\cm^n(R)$.} 
	
	This is a special case of part (ii).
		\item {\it Let $R$ be a Cohen-Macaulay local ring with dualizing module $\omega$ and let $M$ be an unmixed $R$-module. Assume that $0\to Y\to X\overset{\phi}{\to}M\to0$ is a maximal Cohen-Macaulay approximation of $M$, where $X$ is maximal Cohen-Macaulay $R$-module and $\id_R(Y)<\infty$ (see \cite{AuBu}). Let $I\subseteq\ann_R(M)$ be a complete intersection ideal with the same grade $n$ such that $\psi=\phi\otimes_RR/I$ is not injective. Then $M$ is linked by $\psi$ with respect to $\cm^n(R)$.}
		
		 This is an immediate consequence of Corollary \ref{c1} and Example \ref{ex}(iii).
	\end{enumerate}	
\end{eg}
%%%%%%%%%%%%%%%%%%%%%%%%%%%%%%%%%%%%%%%%%%%%%%%
\begin{lem}\label{ll22}
For every $\phi\in\Epi(\mx_n)$, $\coker(\eta_{K}^R(\im \phi))\cong\Ext^{n+1}_R(\im\L_K^n(\phi),K)$.
\end{lem}
\begin{proof}
First note that if $\phi$ is an isomorphism, then $\coker(\eta_{K}^R(\im \phi))=0=\Ext^{n+1}_R(\im\L_K^n(\phi),K)$. Hence we may assume that $\ker\phi\neq0$ . Consider the exact sequence $0\to\ker\phi\to X\overset{\phi}{\to}\im\phi\to0$, where $X\in\mx_n$.
Applying the functor $(-)^{\triangledown}=\Hom_R(-,K)$ implies the following exact sequence  (see \ref{d2}.1).
\begin{equation}\tag{\ref{ll22}.1}
0\to\Ext^n_R(\im\phi,K)\to\Ext^n_R(X,K)\overset{\L_K^n(\phi)}{\longrightarrow}\im\L_K^n(\phi)\to0.
\end{equation}
By Lemma \ref{ll1}, $\Ext^n_R(X,K)\in\mx_n$. In particular, by Theorem \ref{G3}(i), $\Ext^i_R(\Ext^n_R(X,K),K)=0$ for all $i>n$.
Again applying the functor $(-)^{\triangledown}$ to (\ref{ll22}.1) implies the commutative diagram
$$\begin{CD}
%&&&&&&&&\\
\\&&&&X@>\phi>> \im\phi@>>>0&  \\
&&&& @VV{\eta_K^R(X)}V @VV{\eta_K^R(\im\phi)}V \\
\ \  &&&& \Ext^n_R(\Ext^n_R(X,K),K)@>\phi'>>\Ext^n_R(\Ext^n_R(\im\phi,K),K) @>f>> \Ext^{n+1}_R(\im\L_K^n(\phi),K)@>>>0, &\\
\end{CD}$$\\
where $\phi'=\Ext^n_R(\Ext^n_R(\phi,K),K)$. As $X\in\mx_n$, $\eta_K^R(X)$ is an isomorphism. It follows from the commutativity of the above diagram that $\ker(f)=\im(\phi')=\im(\eta_K^R(\im\phi))$ and so
$\coker(\eta_{K}^R(\im \phi))\cong\Ext^{n+1}_R(\im\L_K^n(\phi),K)$.	
\end{proof}
%%%%%%%%%%%%%%%%%%%%%%%%%%%%%%%%%%%%%%%%%%%%%%
\begin{cor}\label{l1}
Let $M$ and $N$ be $R$-modules which are directly linked with respect to $\mx_n$. Then the following statements hold true.
	\begin{enumerate}[\rm(i)]
		\item There exists the exact sequence
		$0\to M\overset{\eta_K^R(M)}{\longrightarrow}\Ext^n_R(\Ext^n_R(M,K),K)\to\Ext^{n+1}_R(N,K)\to0.$
		\item $\Ext^{n+i}_R(N,K)\cong\Ext^{n+i-1}_R(\Ext^{n}_R(M,K),K)$ for all $i>1$.
	\end{enumerate}	
\end{cor}
\begin{proof}
	(i). By definition, there exists $\phi\in\Epi(\mx_n)$ such that $M=\im(\phi)$ and $N\cong\im\L_K^n(\phi)$. Now the assertion follows from Lemma \ref{ll22} and Corollary \ref{c}.

	(ii). By Theorem \ref{t1}(iv) and Corollary \ref{c}, we have the exact sequence
	\begin{equation}\tag{\ref{l1}.1}
	0\to\Ext^n_R(M,K)\to X\to N\to0,
	\end{equation}	
	where $X\in\mx_n=\G_K^n(\mathcal{P})$.
	Applying the functor $(-)^{\triangledown}=\Hom_R(-,K)$ to the exact sequence (\ref{l1}.1) implies the long exact sequence
   \begin{equation}\tag{\ref{l1}.2}
	\cdots\to\Ext^i_R(X,K)\to\Ext^i_R(\Ext^n_R(M,K),K)\to\Ext^{i+1}_R(N,K)\to\Ext^{i+1}_R(X,K)\to\cdots.
   \end{equation}
Note that, by Theorem \ref{G3}(i), $\Ext^i_R(X,R)=0$ for all $i>n$. Hence the assertion follows from the exact sequence (\ref{l1}.2).
\end{proof}
Recall that an $R$-module $M$ is called $\mathcal{C}$-{\it perfect} provided that $\gr_R(M)=\cid_R(M)$. Therefore, over complete intersection local rings, the category of $\mathcal{C}$-perfect modules coincides with the category of Cohen-Macaulay modules. Note that $M$ is $\mathcal{C}$-perfect if and only if it is $\mathcal{G}_R$-perfect and has finite complete intersection dimension (see \cite{AGP} for more details). We denote by $\C(\mathcal{P})$ the category of
$\mathcal{C}$-perfect modules.
%%%%%%%%%%%%%%%%%%%%%%%%%%%%%%%%%%%%%%%%%%%%%%%%
\begin{lem}\label{ll2}
	Let $M\in\C(\mathcal{P})$ be an $R$-module of grade $n$. Then
	$\Ext^n_R(M,R)\in\C(\mathcal{P})$.
\end{lem}
\begin{proof}
	By Lemma \ref{ll1} and \cite[Theorem 1.4]{AGP}, $\Ext^n_R(M,R)$	is a $\mathcal{G}_R$-perfect module. Hence it is enough to show that $\cid_R(\Ext^n_R(M,R))<\infty$.
	If $n=0$, then the assertion follows from \cite[Lemma 3.5]{BJ}. Hence we may assume that $n>0$. Applying the functor $(-)^*=\Hom_R(-,R)$ to a projective resolution ${\bf P}\to M\to0$ of $M$ implies the exact sequences
     \begin{equation}\tag{\ref{ll2}.1}
	0\to P^*_0\to P^*_1\to\cdots\to P^*_n\to\Tr\Omega^{n-1}M\to0,
	\end{equation}
	\begin{equation}\tag{\ref{ll2}.2}
	0\to\Ext^n_R(M,R)\to\Tr\Omega^{n-1}M\to X\to0,
	\end{equation}
	where $X$ is stably isomorphic to $\Omega\Tr\Omega^nM$. It follows from \cite[Lemma 1.9]{AGP} and \cite[Lemma 3.2]{CST} that  $\cid_R(X)=0$. By the exact sequence (\ref{ll2}.1), 
	$\pd_R(\Tr\Omega^{n-1}M)<\infty$. Hence, it follows from the exact sequence (\ref{ll2}.2) that $\cid_R(\Ext^n_R(M,R))<\infty$.
\end{proof}
%%%%%%%%%%%%%%%%%%%%%%%%%%%%%%%%%%%%%%%%%%%%%%%%%%%%%%%%%%%%%%%%%%
The following should be compared with \cite[Theorems 1,15]{MS} and \cite[Corollary 6.10]{Na}.
\begin{thm}\label{cc}
	Let $M$ and $N$ be $R$-modules which are in the same liaison class with respect to $\mx_n$. Then the following statements hold true.
	\begin{enumerate}[\rm(i)]
		\item $M$ is $\gkkkd$-perfect if and only if $N$ is $\gkkkd$-perfect .
		\item If $\mx_n=\mathcal{P}^n(R)$, then $M$ is perfect (resp. $\mathcal{C}$-perfect) if and only if $N$ is perfect (resp. $\mathcal{C}$-perfect). 
	\end{enumerate}	
\end{thm}
\begin{proof}
	Without loss of generality one may assume that $M$ is directly linked to $N$. Hence there exists $\phi\in\Epi(\mx_n)$ such that $N=\im(\phi)$ and $M\cong\im\L_K^n(\phi)$. By Theorem \ref{t1}(iv) and Corollary \ref{c}, it follows the exact sequence
	\begin{equation}\tag{\ref{cc}.1}
	0\to\Ext^n_R(M,K)\to X\to N\to0,
	\end{equation}	
	where $X\in\mx_n$. Suppose that $M$ is $\gkkkd$-perfect. Hence, by Lemma \ref{ll1}, $\Ext^n_R(M,K)$	is $\gkkkd$-perfect module of $\gkkkd$-dimension $n$. It follows from the exact sequence (\ref{cc}.1) and Theorem \ref{G3} that $\gkkd_R(N)\leq n+1$. As $\eta_K^R(M)$ is an isomorphism, $\Ext^{n+1}_R(N,K)=0$ by Corollary \ref{l1}. Therefore, by Theorem \ref{G3}(i), $\gr_R(N)=n=\gkkd_R(N)$.
	Similarly, one can prove the second part by using Lemma \ref{ll2}.	
\end{proof}
%%%%%%%%%%%%%%%%%%%%%%%%%%%%%%%%%%%%%%%%%%%%%%%%%%%%%%%%
\begin{prop}\label{t3}
	Let $M$ and $N$ be $R$-modules that are directly linked with respect to $\mx_n$. Assume that $t>1$ is an integer such that $\gkdp_{R_\fp}(M_\fp)<\infty$ for all $\fp\in\X^{n+t-1}(R)$ (e.g. $\id_{R_\fp}(K_\fp)<\infty$ for all $\fp\in \X^{n+t-1}(R)$).
	Then the following statements are equivalent.
	\begin{enumerate}[\rm(i)]
		\item $\depth_{R_\fp}(M_\fp)\geq\min\{t,\depth R_\fp-n\}$ for all $\fp\in\Spec R$.
		\item $\Ext^i_R(N,K)=0$ for all $i$, $n+1\leq i\leq n+t-1.$
	\end{enumerate}
\end{prop}
\begin{proof}
	This is an immediate consequence of Corollary \ref{l4}(ii) and Corollary \ref{l1}.
\end{proof}
%%%%%%%%%%%%%%%%%%%%%%%%%%%%%%%%%%%%%%%%%%%%%%%%%%%%%%
Recall that an $R$-module $M$ satisfies Serre's condition $(S_t)$ provided  
$\depth_{R_\fp}(M_\fp)\geq\min\{t,\dim_{R_\fp}(M_\fp)\}$ for all $\fp\in\Spec R$.
The following is a generalization of \cite[Theorem 4.1]{Sc}.
\begin{cor}\label{c4}
Let $(R,\fm)$ be a Cohen-Macaulay local ring which is a homomorphic image of a Gorenstein local
ring. Assume that $M$ and $N$ are $R$-modules that are directly linked with respect to $\cm^n(R)$. The
following statements are equivalent.
\begin{enumerate}[\rm(i)]
	\item $M$ satisfies the Serre's condition $(S_t)$.
	 \item $\hh^i_\fm(N)=0$ for all $i$, $\dim_R(N)-t<i<\dim_R(N)$.	
\end{enumerate}
\end{cor} 
\begin{proof}
First note that, by Proposition \ref{p2} and \cite[Corollary 4.6]{AB}, $\gr_{R_\fp}(M_\fp)=n$ for all $\fp\in\Supp_R(M)$.	
As $R$ is Cohen-Macaulay, we have $\dim_{R_\fp}(M_\fp)=\dim R_\fp-\gr_{R_\fp}(M_\fp)=\depth R_\fp-n$ for all $\fp\in\Supp_R(M)$. Now the assertion follows from Proposition \ref{t3} and the local duality theorem.
\end{proof}
%%%%%%%%%%%%%%%%%%%%%%%%%%%%%%%%%%%%%%%%%%%%%%%%%%%
The property of being an $m$-syzygy and its relation to the vanishing of certain cohomology modules was studied by Bass \cite{B}. In the following we investigate the connection of being an $m$-$K$-syzygy on a linked module with the vanishing of certain cohomology modules of its linked module. 

An $R$-module $M$ of codimension zero (i.e. $\dim_R(M)=\dim R$) is called {\em maximal}. Clearly the grade of a maximal module is zero.
\begin{prop}\label{p1}
Let $M$ and $N$ be maximal $R$-modules of that are directly linked with respect to $\mx_0$. Assume that $m>1$ is an integer such that $\gkdp_{R_\fp}(M_\fp)<\infty$ for all $\fp\in\X^{m-2}(R)$ (e.g. $\id_{R_\fp}(K_\fp)<\infty$ for all $\fp\in\X^{m-2}(R)$). Then the following statements are equivalent.
\begin{enumerate}[\rm(i)]
	\item $M$ is an $m$-$K$-syzygy module.
	\item $\Ext^i_R(N,K)=0$ for all $i$, $1\leq i\leq m-1.$
\end{enumerate}
\end{prop}
\begin{proof}
By definition, there exists homomorphism $\phi\in\Epi(\mx_0)$ such that $M\cong\im\phi$ and $N\cong\im\L_K^0(\phi)$. For $m=2$, the assertion follows from Theorem \ref{th}(b), Lemma \ref{l5}(i) and Corollary \ref{l1}(i). Suppose that $m>2$.
By definition, there exists the exact sequence
\begin{equation}\tag{\ref{p1}.1}
0\longrightarrow M^{\triangledown}\longrightarrow X\longrightarrow N\longrightarrow0,
\end{equation}
where $X\in\mx_0$ (see \ref{d2}). The exact sequences (\ref{p1}.1) and (\ref{d3}.1) induce the isomorphisms
\begin{equation}\tag{\ref{p1}.2}
\Ext^i_R(N,K)\cong\Ext^{i-1}_R(M^{\triangledown},K)\cong\Ext^{i+1}_R(\Tr_KM,K),  i>1.
\end{equation}
 Now the assertion follows from (\ref{p1}.2) and Theorem \ref{th}(b).
\end{proof}
%%%%%%%%%%%%%%%%%%%%%%%%%%%%%%%%%%%%%%%%%%%%%%%%%%%%%%%%%%%%%%%%%
The following Lemmas play important roles in the proof of our main results.
\begin{lem}\label{lll1}
	Let $M$ and $N$ be maximal $R$-modules of that are directly linked with respect to $\mx_0$. Then
	there exists an exact sequence $0\to N\to\lambda^{\pi}_KM\to Y\to 0$ where $\pi$ is a projective presentation of $M$ (see (\ref{d3}.2)) and $Y$
	is an $R$-module contained in $\mx_0$.
\end{lem}
\begin{proof}
	By definition, there exists an epimorphism $\phi:X\twoheadrightarrow M$ in $\Epi(\mx_0)$ such that $M=\im\phi$ and $N\cong\im\L^0_K(\phi)$. Take an epimorphism $f:P\twoheadrightarrow M$ where $P$ is projective. Hence there exists a homomorphism $g:P\to X$ such that  $f=\phi\circ g$. Let $h:Q\twoheadrightarrow\coker(g)$ be an epimorphism where $Q$ is projective. It induces a homomorphism $\upsilon:Q\to X$ such that $h=j\circ\upsilon$ where $j:X\twoheadrightarrow\coker(g)$ is the natural epimorphism. Set $P_0=P\oplus Q$ and define homomorphisms $\alpha=\left(f\ \ \phi\circ\upsilon \right):P_0\to M$ and $\psi=\left(g \ \ \upsilon \right):P_0\to X$. It is easy to see that $\alpha$ and $\psi$ are epimorphisms and $\phi\circ\psi=\alpha$. Hence we obtain the commutative diagram
	\[\xymatrix{
		0 \ar[rr] &&\ker\alpha\ar@{-->}[d]^{\iota}\ar[rr] &&P_0\ar@{->>}[d]^{\psi}\ar[rr]^{\alpha} &&M\ar@{=}[d]\ar[rr]&&0	\\
		0 \ar[rr] &&\ker\phi\ar[rr] &&X\ar[rr]^{\phi} &&M\ar[rr]&&0.	
	}\]
	Applying the functor $(-)^{\triangledown}$ to the diagram implies the commutative diagram (see (\ref{d3}.2) and (\ref{d2}.1))
	\[\xymatrix{
		0 \ar[rr] &&M^{\triangledown}\ar@{=}[d]\ar[rr] &&X^{\triangledown}\ar@{^{(}->}[d]^{\psi^{\triangledown}}\ar[rr]^{\L^0_K(\phi)} &&\im\L^0_K(\phi)\ar@{-->}[d]^{\iota'}\ar[rr]&&0	\\
		0 \ar[rr] &&M^{\triangledown}\ar[rr] &&P_0^{\triangledown}\ar[rr] &&\lambda^{\pi}_KM\ar[rr]&&0,	
	}\]
	where $\pi: P_1\to P_0\overset{\alpha}{\to}M\to0$ is a projective presentation of $M$, induced by $\alpha$. As $X$ is totally $K$-reflexive, so is $\ker(\psi)$. Hence $\coker(\psi^{\triangledown})\cong(\ker\psi)^\triangledown$ is totally $K$-reflexive. It follows from the snake lemma that $\iota'$ is injective
	and $\coker(\psi^{\triangledown})\cong\coker(\iota')$. Note that $N\cong\im\L^0_K(\phi)$. Therefore we obtain an exact sequence of the form $0\to N\to\lambda^\pi_KM\to Y\to 0$, where $Y$ is a totally $K$-reflexive module as desired.
\end{proof}
%%%%%%%%%%%%%%%%%%%%%%%%%%%%%%%%%%%%%%%%%%%%%%%%%%%%%%%%%%%%%%%%%
\begin{lem}\label{l6}
	Let $M$ and $N$ be $R$-modules which are directly linked with respect to $\mx_n$. Then there exists an ideal ${\bf x}$ of $R$, generated by a regular sequence of length $n$ contained in $\ann_R(M)\cap\ann_R(N)$, such that $M$ and $N$ are directly linked with respect to the category of totally $C$-reflexive $R/{\bf x}$-modules, where $C=\Ext^n_R(R/{\bf x},K)$.
\end{lem}
\begin{proof}
	By definition, there exists an epimorphism $\phi:X\twoheadrightarrow M$ in $\Epi(\mx_n)$ such that $\im\L^n_K(\phi)\cong N$. 
	Choose an ideal ${\bf x}\subseteq\ann_R(X)$, generated by a regular sequence of length $n$.
	Set $S=R/{\bf x}$ and $C=\Ext^n_R(S,K)$. Note that by Theorem \ref{G2}, $X$ is a totally $C$-reflexive $S$-module. It follows from Lemma \ref{l5} and Corollary \ref{c} that $M$ is linked as an $S$-module by $\phi$ with respect to the category of totally $C$-reflexive $S$-modules. By the exact sequence (\ref{d2}.1) and Lemma \ref{G1}, we obtain the commutative diagram
	\[\xymatrix{
		0 \ar[rr] &&\Ext^n_R(M,K)\ar[d]^{\cong}\ar[rr] &&\Ext^n_R(X,K)\ar[d]^{\cong}\ar[rr]^{\L_K^n(\phi)} &&\im\L^n_K(\phi)\ar@{-->}[d]\ar[rr]&&0	\\
		0 \ar[rr] &&\Hom_S(M,C)\ar[rr] &&\Hom_S(X,C)\ar[rr]^{\L_C^0(\phi)} &&\im\L_C^0(\phi)\ar[rr]&&0	
	}\]
	with exact rows. It follows that $N\cong\im\L_K^n(\phi)\cong\im\L_C^0(\phi)$.
\end{proof}	
%%%%%%%%%%%%%%%%%%%%%%%%%%%%%%%%%%%%%%%%%%%%%%%%%%%%
The theory of secondary representation has been introduced by Macdonald \cite{IGM}. Let $M$ be a nonzero Artinian $R$--module. Recall that $M$ is said to be \emph{secondary} if the multiplication map by $x$ on $M$ is either surjective or nilpotent for every $x\in R$. In this case, $\fp:=\sqrt{\ann_R(M)}$ is a
prime ideal of $R$ and $M$ is called a  $\fp$-secondary module. Note that every Artinian $R$-module $M$ has a minimal secondary representation $M=M_1+\cdots+M_n$, where $M_i$ is $\fp_i$-secondary, each $M_i$ is not redundant and $\fp_i\neq\fp_j$ for all $i\neq j $. In this situation, the set $\{\fp_1,\cdots,\fp_n\}$ is independent of the choice of a minimal secondary representation of $M$. This set is called the set of \emph{attached primes} of $M$ and denoted by $\Att_R(M)$. 

Recall that a ring $R$ is called \emph{equidimensional} if $\dim R/\fp=\dim R$ for every minimal prime ideal $\fp$ of $R$. A ring $R$ is said to be  \emph{formally equidimensional} if its completion $\widehat{R}$ is equidimensional. Clearly, every Cohen-Macaulay local ring is formally equidimensional. Also, every homomorphic image of a regular local ring that is
equidimensional is formally equidimensional (see\cite[Section 31]{Ma}). 

A subset $\X$ of $\Spec R$ is called \emph{stable under generalization} provided that if $\fp\in\X$ and
$\fq\in\Spec R$ with $\fq\subseteq\fp$, then $\fq\in\X$. Note that every open subset of $\Spec R$ is stable under generalization.

The following is a generalization of \cite[Theorem 3.3]{Sa1}.
\begin{thm}\label{t5}
	Let $(R,\fm)$ be a formally equidimensional local ring which is a homomorphic image of a Cohen-Macaulay local ring. Assume that the $R$-modules $M$, $N$ are linked with respect to $\mx_n$ in an
	odd number of steps and that $t$ is a positive integer. Assume further that $\X$ is a subset of $\Spec R$ which is stable under generalization and that $R_{\fp}$ is Cohen-Macaulay and $\gkkpd_{R_\fp}(M_\fp)<\infty$ for all $\fp\in X$ (e.g. $\id_{R_\fp}(K_\fp)<\infty$ for all $\fp\in\X$).
	Then the following are equivalent.
	\begin{enumerate}[\rm(i)]
		\item{$\X\subseteq \Sn_t(M)$.}
		\item{$\Att_R(\hh^i_{\fm}(N))\subseteq\Spec R\setminus\X$ for all $i$, $\dim_R(N)-t<i<\dim_R(N)$.}
	\end{enumerate}
\end{thm}
\begin{proof}
	We may assume, without loss of generality, that $M$ is directly linked to $N$. 
	
	(i)$\Rightarrow$(ii). Assume,  contrarily,  that  $\fp\in\Att_R(\hh^i_\fm(N))\cap X$ for some $i$ with $\dim_R( N)-t<i<\dim_R(N)$. Therefore,
	$\fp R_\fp\in\Att_{R_\fp}(\hh^{i-\dim R/\fp}_{\fp R_\fp}( N_\fp))$ by \cite[Theorem 1.1]{NQ}.
	As $\fp\in X\subseteq\Sn_t(M)$, we have $\hh^{j}_{\fp R_\fp}(N_\fp)=0$ for all $j$, $\dim_{R_\fp}(N_\fp)-t<j<\dim_{R_\fp}(N_\fp)$ by Corollary \ref{c4}.
	Note that, by \cite[Theorem 31.5]{Ma}, $R$ is equidimensional and catenary. Hence
	$\hei(\fp_2/\fp_1)=\hei(\fp_2)-\hei(\fp_1)$ 
	for all $\fp_1, \fp_2\in\Spec R$ with $\fp_1\subseteq\fp_2$ (see \cite[Lemma 2 on page 250]{Ma}). There exists a prime ideal $\fq\in\Min_R(N)$ such that $\dim_{R_\fp}(N_\fp)=\hei(\fp/\fq)$. Therefore we have the following equalities
	\begin{equation}\tag{\ref{t5}.1}
	\dim_{R_\fp}(N_\fp)+\dim R/\fp=\hei(\fp/\fq)+\hei(\fm/\fp)=\dim R/\fq=\dim_R(N),
	\end{equation}
	where the last equality follows from Proposition \ref{p2}.
	It follows from (\ref{t5}.1) that $\dim_{R_\fp}(N_\fp)-t<i-\dim R/\fp<\dim_{R_\fp}(N_\fp)$.
	In particular, $\hh^{i-\dim R/\fp}_{\fp R_\fp}(N_\fp)=0$ and so $\Att_{R_\fp}(\hh^{i-\dim R/\fp}_{\fp R_\fp}( N_\fp))=\emptyset$ which is a contradiction.
	
	(ii)$\Rightarrow$(i). By Lemma \ref{l6}, there exists an ideal ${\bf x}\subseteq\ann_R(M)\cap\ann_R(N)$, generated by a regular sequence of length $n$ such that $M$ and $N$ are directly linked with respect to totally $C$-reflexive $S$-modules, where $S=R/{\bf x}$ and $C=\Ext^n_R(S,K)$. Consider the natural epimorphism $f:R\twoheadrightarrow S$ and set $\overline{\X}=f(\X)=\{\fp\in\Spec S\mid \fp=P/{\bf x} \text{ for some } P\in\X \}$. It is clear that $\overline{\X}$ is stable under generalization and $S_\fp$ is Cohen-Macaulay for all $\fp\in\overline{\X}$. Also, by Theorem \ref{G2}, $\gcpd_{S_\fp}(M_\fp)<\infty$ for all $\fp\in\overline{\X}$. On the other hand, by the independence Theorem \cite[Theorem 4.2.1]{BS}, $\hh^i_{\fm}(M)\cong\hh^i_{\fm/{\bf x}}(M)$ for all $i$. Also, by \cite[Proposition 4.1]{Ml}, $\Att_R(\hh^i_{\fm}(M))=\{\fp\cap R\mid\fp\in\Att_S(\hh^i_{\fm}(M))\}$. Hence, without loss of generality, we may assume that $M$ and $N$ are maximal $R$-modules which are directly linked with respect to $\mx_0$. 
	Assume contrarily that $\X\nsubseteq\Sn_t(M)$. Set $\Y=\{\fp\in\X\mid M_\fp \text{ does not satisfy } (\Sn_t) \}$. Let $\fp_0$ be a minimal element of $\Y$ with respect to the inclusion relation. Here is our first claim.
	
	\textbf{Claim (I)}. $\Ext^i_R(\Tr_KM,K)_{\fp_0}\neq0$ for some $i$, $1\leq i\leq t$.\\
	{\it Proof of Claim (I)}. Assume contrarily that $\Ext^i_R(\Tr_KM,K)_{\fp_0}=0$ for all $i$, $1\leq i\leq t$. By Theorem \ref{th} 
	\begin{equation}\tag{\ref{t5}.2}
	\depth_{R_{\fq}}(M_{\fq})\geq\min\{t,\depth R_{\fq}\} \text{ for all prime ideal } \fq\subseteq\fp_0.
	\end{equation}	
	Note that, for all prime ideal $\fq$ with $\fq\subseteq\fp_0$, the local ring $R_\fq$ is Cohen-Macaulay and so 
	\begin{equation}\tag{\ref{t5}.3}
	\dim M_\fq=\dim R_\fq-\gr_{R_\fq}(M_\fq)=\depth R_\fq,
	\end{equation}
	where the second equality follows from the fact that $M$ is maximal. It follows from (\ref{t5}.2) and (\ref{t5}.3) that $\fp_0\in\Sn_t(M)$ which is a contradiction. Thus, the proof of the claim (I) is complete.\\
	
	Set $s=\min\{j\mid\Ext^j_R(\Tr_K M,K)_{\fp_0}\neq0\}.$ It follows from Corollary \ref{c} and Claim (I) that $1<s\leq t$. Note that $\fq\in\Sn_t(M)$ for all prime ideal $\fq\in\Spec R$ with $\fq\subsetneq\fp_0$, because $\fp_0$ is a minimal element of $\Y$ and $\X$ is stable under generalization. It follows from (\ref{t5}.2) and Theorem \ref{th} that
	\begin{equation}\tag{\ref{t5}.4}
	\Ext^s_R(\Tr_K M,K)_{\fq}=0 \text{ for all prime ideal } \fq\text{ with } \fq\subsetneq\fp_0.
	\end{equation}
	Therefore $\fp_0\in\Min_R(\Ext^s_R(\Tr_K M,K))$. Note that by \cite[Proposition 9.A]{Ma1} 
	$$\Ass_R(\Ext^s_R(\Tr_K M,K))=\{P\cap R\mid P\in\Ass_{\widehat{R}}(\Ext^s_{\widehat{R}}(\widehat{\Tr_K M},\widehat{K}))\},$$
	where $\widehat{R}$ is the completion of $R$ in the $\fm$-adic topology.
	Choose $P_0\in\Min_{\widehat{R}}(\Ext^s_{\widehat{R}}(\widehat{\Tr_K M},\widehat{K}))$ such that $P_0\cap R=\fp_0$.
	Note that
	\begin{equation}\tag{\ref{t5}.5}
	\Ext^i_{\widehat{R}}(\widehat{\Tr_K M},\widehat{K})_{P_0}=0 \text{ for all } i,  0<i<s.
	\end{equation}
	As $\gkkp0d_{R_{\fp_0}}(M_{\fp_0})<\infty$ and the local homomorphism $R_{\fp_0}\rightarrow \widehat{R}_{P_0}$ is flat,
	$\gkkP0d_{\widehat{R}_{P_0}}(\widehat{M}_{P_0})<\infty$. As $R$ is the homomorphic image of a Cohen-Macaulay ring, all formal fibres of $R$ are Cohen-Macaulay.
	Therefore, $\widehat{R}_{P_0}/{\fp_0\widehat{R}_{P_0}}\cong \left((R_{\fp_0}/{\fp_0 R_{\fp_0}}\right)\otimes_R\widehat{R})_{P_0}$ is Cohen-Macaulay and so $\widehat{R}_{P_0}$ is Cohen-Macaulay (see \cite[Page 181]{Ma}). Now we assert our second claim.
	
	\textbf{Claim (II)}. $P_0\widehat{R}_{P_0}\in\Ass(\Ext^s_{\widehat{R}_{P_0}}(\Tr_{\widehat{R}_{P_0}}\widehat{M}_{P_0},\omega_{\widehat{R}_{P_0}})),$ where $\omega_{\widehat{R}_{P_0}}$ is the dualizing module of $\widehat{R}_{P_0}$.\\
	{\it Proof of Claim (II)}. By using the fact that 
	$P_0\in\Min_{\widehat{R}}(\Ext^s_{\widehat{R}}(\widehat{\Tr_K M},\widehat{K}))$, \cite[Theorem 5.2]{Sa} and (\ref{t5}.5) we have  $\Ext^s_{\widehat{R}_{P_0}}(\Tr_{\widehat{K}_{P_0}}\widehat{M}_{P_0},\omega_{\widehat{R}_{P_0}})\neq0$ and also
	 $\Ext^s_{\widehat{R}_{Q}}(\Tr_{\widehat{K}_{Q}}\widehat{M}_{Q},\omega_{\widehat{R}_{Q}})=0$ for all prime ideal $Q\subsetneq P_0$. In other words, $P_0\widehat{R}_{P_0}\in\Min_{\widehat{R}_{P_0}}(\Ext^s_{\widehat{R}_{P_0}}(\Tr_{\widehat{K}_{P_0}}\widehat{M}_{P_0},\omega_{\widehat{R}_{P_0}}))$ which completes the proof of the claim (II).\\

	By Lemma \ref{lll1} and (\ref{d3}.2), there exist exact sequences $$0\to N\to\lambda^{\pi}_KM\to Y\to 0  \text{ \ \ and \ \ } 0\to\lambda_K^{\pi}M\to F_1^{\triangledown}\to\Tr^{\pi}_KM\to0,$$
	where $F_1$ is a free $R$-module, $\pi$ is a projective presentation of $M$ and $Y$	is an $R$-module contained in $\mx_0$. The above exact sequences induce the following exact sequences
	$0\to\widehat{N}_{P_0}\to\widehat{\lambda^{\pi}_KM}_{P_0}\to \widehat{Y}_{P_0}\to 0$ and 
	$0\to\widehat{\lambda_K^{\pi}M}_{P_0}\to \widehat{F_1^{\triangledown}}_{P_0}\to\widehat{\Tr^{\pi}_KM}_{P_0}\to0$.
	Note that $\widehat{Y}_{P_0}$ is a maximal Cohen-Macaulay $\widehat{R}_{P_0}$-module. 
	As $s>1$, applying the functor $\Hom_{\widehat{R}_{P_0}}(-,\omega_{\widehat{R}_{P_0}})$ to the above exact sequences one obtains the isomorphisms
	\begin{equation}\tag{\ref{t5}.6}
	\Ext^s_{\widehat{R}_{P_0}}(\Tr_{\widehat{K}_{P_0}}\widehat{M}_{P_0},\omega_{\widehat{R}_{P_0}})\cong\Ext^{s-1}_{\widehat{R}_{P_0}}(\widehat{\lambda_K^{\pi} M}_{P_0},\omega_{\widehat{R}_{P_0}})\cong\Ext^{s-1}_{\widehat{R}_{P_0}}(\widehat{N}_{P_0},\omega_{\widehat{R}_{P_0}}).
	\end{equation}
	
	Therefore, by (\ref{t5}.6), Claim (II), \cite[3.4]{Sh} and the local duality theorem,  $P_0\widehat{R}_{P_0}\in\Att_{\widehat{R}_{P_0}}(\hh^{m-s+1}_{P_0\widehat{R}_{P_0}}(\widehat{N}_{P_0}))$ where $m=\dim \widehat{R}_{P_0}$. As $\widehat{R}$ is equidimensional and catenary, $m+\dim(\widehat{R}/{P_0})=\dim \widehat{R}$. 
	Set $d=\dim R=\dim\widehat{R}$. 
	Hence, by \cite[Theorem 1.1]{NQ}, $P_0\in\Att_{\widehat{R}}(\hh^{d-s+1}_{\widehat{\fm}}(\widehat{N}))$ and so
	$\fp_0\in\Att_R(\hh^{d-s+1}_{m}(N))$ by \cite[Proposition 3.2]{Ml}, which is a contradiction because $d-t<d-s+1<d$ and $\fp_0\in\X$.
\end{proof}
%%%%%%%%%%%%%%%%%%%%%%%%%%%%%%%%%%%%%%%%%%%%%%%%%%%%%%
Note that $M\neq0$ if and if $\Att_R(M)\neq\emptyset.$
Hence, Theorem \ref{t5} can be seen as a generalization of Schenzel's result \cite[Theorem 4.1]{Sc}. As an immediate consequence of Theorem \ref{t5}, we state the following results.
%%%%%%%%%%%%%%%%%%%%%%%%
\begin{cor}
	Let $(R,\fm)$ be a Cohen-Macaulay local ring and let $t$ be a positive integer such that $\id_\fp(K_\fp)<\infty$ for all $\fp\in\X^{t-1}(R)$. Assume that $M$ and $N$
	are $R$-modules of dimension $d$ which are linked with respect to $\mx_n$ in an odd number of steps. The following are equivalent.
	\begin{enumerate}[\rm(i)]
		\item{$M_\fp$ is Cohen-Macaulay for all $\fp\in\X^{t-1}(R)$.}
		\item{$\depth R_\fp\geq t$ for all $\fp\in\underset{\tiny{0<i< d}}{\bigcup}\Att_R(\hh^i_{\fm}(N))$.}
	\end{enumerate}
\end{cor}
\begin{proof}
	The condition (i) is equivalent to say that $\X^{t-1}(R)\subseteq S_d(M)$. Now the assertion follows from Theorem \ref{t5}.
\end{proof}
%%%%%%%%%%%%%%%%%%%%%%%%%%%%%%%%%%%%%%%%%%%%%%%%%%%%%%%%%%%%%%%%%%%%%%%%%%%%
%%%%%%%%%%%%%%%%%%%%%%%%%%%%%%%%%%%%%%%%%%%
\begin{cor}
	Let $(R,\fm)$ be a Cohen-Macaulay local ring and let  $\id_\fp(K_\fp)<\infty$ for all $\fp\in\Spec R\setminus\{\fm\}$. Assume that $M$ and $N$
	are $R$-modules which are linked with respect to $\mx_n$ in an
	odd number of steps. The following are equivalent.
	\begin{enumerate}[\rm(i)]
		\item $M_\fp$ satisfies $(S_t)$ for all $\fp\in\Spec R\setminus\{\fm\}$.
		\item $\ell(\hh^i_\fm(N)))<\infty$ for all $i$, $\dim_R(N)-t<i<\dim_R(N)$, where $\ell(-)$ denotes the length function.
	\end{enumerate}
\end{cor}
\begin{proof}
	This is an immediate consequence of Theorem \ref{t5} and \cite[Corollary 7.2.12]{BS}.
\end{proof}
%%%%%%%%%%%%%%%%%%%%%%%%%%%%%%%%%%%%%%%%%%%%%%%%%%%%%%%%%%%%%%
The following theorem is a generalization of \cite[Theorem 3.7]{Sa1}.
%%%%%%%%%%%%%%
\begin{thm}\label{tc1}
	Let $(R,\fm)$ be a formally equidimensional local ring which is a homomorphic image of a Cohen-Macaulay local ring. Let $\X$ be a subset of $\Spec R$ which is stable under generalization such that $\id_{R_\fp}(K_\fp)<\infty$ for all $\fp\in\X$. Assume that $M$ and $N$ are $R$-modules of dimension $d$ which are linked with respect to $\mx_n$ in an odd number of steps and that $m$ is an integer with $0<m\leq d$. 
	If $\Att_R(\hh^j_\fm(N))\subseteq\Spec R\setminus\X$ for all $j$, $d-m<j<d$, then $\Att_R(\hh^i_\fm(M))\subseteq\Spec R\setminus \X$ for all $i$, $0<i<m$.
\end{thm}
\begin{proof}
	Assume contrarily that $\fp\in\Att_R(\hh^i_\fm(M))\cap\X$ for some $i$, $0<i<m$.
	By \cite[Theorem 1.1]{NQ}, $\fp R_\fp\in\Att_{R_\fp}(\hh^{i-\dim R/\fp}_{\fp R_\fp}(M_\fp))$.
	In particular,
	\begin{equation}\tag{\ref{tc1}.1}
	\hh^{i-\dim R/\fp}_{\fp R_\fp}(M_\fp)\neq0.
	\end{equation}
	On the other hand, by Theorem \ref{t5}, $M_\fp$ satisfies the Serre's condition $(S_m)$. In other words, 
	\begin{equation}\tag{\ref{tc1}.2}
	\depth_{R_\fp}(M_\fp)\geq\min\{m, \dim_{R_\fp}(M_\fp)\}.
	\end{equation}
	Note that, by \cite[Theorem 31.5]{Ma}, $R$ is equidimensional and catenary. Now by using the fact that $R_\fp$ is Cohen-Macaulay and Proposition \ref{p2}, it is easy to see that
	\begin{equation}\tag{\ref{tc1}.3}
	\dim_{R_\fp}(M_\fp)+\dim R/\fp=d.
	\end{equation}
	If $m\geq\dim_{R_\fp}(M_\fp)$, then $M_\fp$ is Cohen-Macaulay by (\ref{tc1}.2) and so $\hh^j_{\fp R_\fp}(M_\fp)=0$ for all $j\neq\dim_{R_\fp}(M_\fp)$.
	Therefore, by (\ref{tc1}.1), $i-\dim R/\fp=\dim_{R_\fp}(M_\fp)$ and so $i=d$ by (\ref{tc1}.3) which is a contradiction.
	On the other hand, if $m<\dim_{R_\fp}(M_\fp)$, then $\depth_{R_\fp}(M_\fp)\geq m$ by (\ref{tc1}.2) and so $\hh^j_{\fp R_\fp}(M_\fp)=0$ for all $j<m$.
	In particular, $\hh^{i-\dim R/\fp}_{\fp R_\fp}(M_\fp)=0$ which is a contradiction by (\ref{tc1}.1).
\end{proof}
%%%%%%%%%%%%%%%%%%%%%%%%%%%%%%%%%%%%%%%%%%%%%%%%%%%%%%%%%%%
Recall that a subset $\X$ of $\Spec R$ is called \emph{stable under specialization} if every prime ideal of $R$ which contains
an element of X also belongs to $\X$. Clearly every closed subset of $\Spec R$ is stable under specialization. The following is an immediate consequence of Theorem \ref{tc1}.
\begin{cor}\label{rc}
	Let $(R,\fm)$ be a formally equidimensional local ring which is a homomorphic image of a Cohen-Macaulay local ring. Let $\X$ be a subset of $\Spec R$ which is stable under specialization such that $\id_{R_\fp}(K_\fp)<\infty$ for all $\fp\in\Spec R\setminus\X$. Assume $M$ and $N$ are $R$-modules which are linked with respect to $\mx_n$ in an
	odd number of steps. The following are equivalent.
	\begin{enumerate}[\rm(i)]
		\item{$\Att_R(\hh^i_\fm(M))\subseteq\X$ for all $i$, $0<i<\dim M$.}
		\item{$\Att_R(\hh^i_\fm(N))\subseteq\X$ for all $i$, $0<i<\dim N$.}
	\end{enumerate}
\end{cor}
%%%%%%%%%%%%%%%%%%%%%%%%%%%%%%%%%%%%%%%%%%%%%%%%%%%
An $R$-module $M$ is called \emph{self-linked with respect to} $\mx_n$ provided that there exists $\phi\in\Epi(\mx_n)$ such that $M=\im\phi\cong\im\L_K^n(\phi)$.
Recall that the \emph{Cohen-Macaulay locus} of an $R$-module $M$, denoted by $\cm_R(M)$, is defined as $\cm_R(M)=\{\fp\in\Spec R\mid M_\fp \text{ is a Cohen-Macaulay } R_\fp\text{-module}\}. $

%%%%%%%%%%%%%%%%%%%%%%%%%%%%%%%%%%%%%%%%%%%%%%%%%%%%%
The following is an immediate consequence of Theorems \ref{t5} and \ref{tc1},
which is nothing but Corollary \ref{self}, mentioned in the introduction.
\begin{cor}
Let $(R,\fm)$ be a formally equidimensional local ring which is a homomorphic image of a Cohen-Macaulay local ring. Assume that $\X$ is a subset of $\Spec R$ which is stable under specialization such that $\id_{R_\fp}(K_\fp)<\infty$ for all $\fp\in\Spec R\setminus\X$. Let $M$ be an $R$-module of dimension $d$ which is self-linked with respect to $\mx_n$. The following are equivalent.
\begin{enumerate}[\rm(i)]
	\item $\Att_R(\hh^i_{\fm}(M))\subseteq\X$ for all $i$, $\lfloor d/2\rfloor\leq i<d$.
	\item $\Att_R(\hh^i_{\fm}(M))\subseteq\X$ for all $i$, $0<i<d$.
	\item $\Spec R\setminus\X\subseteq\cm_R(M)$.
\end{enumerate}
\end{cor}
Noted that the above result is new, even if we are restricted to classical linkage theory.
%%%%%%%%%%%%%%%%%%%%%%%%%%%%%%%%%%%%%%%%%%%%%%%%%%
\begin{cor}\label{cs}
	Let $(R,\fm)$ be a Gorenstein local ring and let $\X$ be an open subset of $\Spec R$. Assume that $I$ is an ideal  of dimension $d$ which is self-linked. The following are equivalent.
	\begin{enumerate}[\rm(i)]
		\item $\Att_R(\hh^i_{\fm}(R/I))\subseteq\Spec R\setminus\X$ for all $i$, $\lfloor d/2\rfloor\leq i<d$.
		\item $\Att_R(\hh^i_{\fm}(R/I))\subseteq\Spec R\setminus\X$ for all $i$, $0<i<d$.
		\item $\X\subseteq\cm_R(R/I)$.
	\end{enumerate}
\end{cor}
%%%%%%%%%%%%%%%%%%%%%%%%%%%%%%%%%%%%%%%%%%%%%%%%%%%%%%%%%%%%%%
Recall that an $R$-module $M$ of dimension $d\geq1$ is called 
\emph{ generalized Cohen-Macaulay} provided that
$\ell(\hh^i_\fm(M))<\infty$ for all $i$, $0\leq i\leq d-1$, where
$\ell(-)$ denotes the length function. There is an interesting duality between local cohomology modules of generalized Cohen-Macaulay ideals which are linked by a Gorenstein ideal, due to Schenzel. Let $R$ be a Gorenstein local ring and let $\fa$, $\fb$ be ideals of $R$ linked by a Gorenstein ideal
$\fc$. Assume that $R/\fa$ is generalized Cohen-Macaulay. Schenzel proved that $$\hh^i_{\fm}(R/\fa)\cong\Hom_R(\hh^{d-i}_{\fm}(R/\fb),\E_R(R/\fm)),$$
for $0<i<d$, where $d=\dim R/\fa=\dim R/\fb$  \cite[Corollary 3.3]{Sc}. Martsinkovsky and Strooker extended the above result to generalized Cohen-Macaulay modules which are linked by a Gorenstein ideal \cite[Theorem 11]{MS}. The Schenzel's result is generalized by Nagel for locally Cohen-Macaulay modules which are linked by a quasi-Gorenstein module \cite[Corollary 6.1(b)]{Na}. In the following, we extend the above result to modules which satisfy Serre's condition $(S_t)$ on the punctured spectrum of $R$.

Recall that an $R$-module $M$ is called $t$-torsionfree with respect to $K$ provided that $\Ext^i_R(\Tr_{K}M,K)=0$ for all $i$, $1\leq i\leq t$.
For an ideal $\fa$ of $R$, we denote by $\V(\fa)$ the set of all prime ideals of $R$ containing $\fa$.
The following is a generalization of \cite[Theorem 5.1]{Sa1}.
\begin{lem}\label{l7}
Let $\fa$ be an ideal of $R$ and let $M$ be an $R$-module. Assume that $t$ is an integer such that $0<t\leq\gr(\fa)$ and that $M_\fp$ is an $t$-torsionfree $R_\fp$-module with respect to $K_\fp$ for all  $\fp\in\Spec R\setminus\V(\fa)$. Then $\hh^i_{\fa}(M)\cong\Ext^{i+1}_R(\Tr_KM,K)$ for all $i$, $0\leq i\leq t-1$. In particular, $\hh^i_{\fa}(M)$ is finitely generated for all $i$,  $0\leq i\leq t-1$.
\end{lem}
\begin{proof}
First note that, by our assumption,	
$\underset{1\leq i\leq t}{\bigcup}\Supp_R(\Ext^i_R(\Tr_K M,K))\subseteq\V(\fa)$and so $\Ext^i_R(\Tr_K M,K)$ is
$\fa$-torsion for $1\leq i\leq t$. For each $i>0$, set $X_i=\Omega^{i-1}\Tr_{K}M$. By definition, there exist exact sequences
\begin{equation}\tag{\ref{l7}.1}
0\to\Ext^i_R(X_1,K)\to\Tr_{K}(X_i)\to Y_i\to0,
\end{equation} 
\begin{equation}\tag{\ref{l7}.2}
0\to Y_i\to P_i\otimes_{R}K\to\Tr_{K}(X_{i+1})\to0,
\end{equation}
where $P_i$ is a projective module. Applying the functor $\Gamma_{\fa}(-)$ to the exact sequences (\ref{l7}.1) and (\ref{l7}.2) and using the fact that $\Ext^i_R(\Tr_K M,K)$ is $\fa$-torsion for $1\leq i\leq t$ and that $\hh^j_{\fa}(P_i\otimes_{R}K)=0$ for all $j<\gr(\fa)$ and $i>0$, we get the isomorphisms
\begin{equation}\tag{\ref{l7}.3}
\hh^j_{\fa}(\Tr_K(X_i))\cong\hh^j_{\fa}(Y_i)\cong\hh^{j-1}_{\fa}(\Tr_{K}(X_{i+1})) \text{ for all }\  0<j<t,\ 1\leq i\leq t,
\end{equation}
\begin{equation}\tag{\ref{l7}.4}
\Ext^i_R(X_1,K)=\Gamma_{\fa}(\Ext^i_R(X_1,K))\cong\Gamma_{\fa} (\Tr_K(X_i)) \text{ for all }\ 1\leq i\leq t.
\end{equation}
By \cite[Lemma 2.12]{Sa}, there exists the exact sequence $0\to M\to\Tr_{K}(\Tr_{K}M)\to Z\to0$, where $Z$ is a totally $K$-reflexive module. Note that every $R$-regular sequence is also $Z$-regular sequence (see for example \cite[Corollary 32]{Ma}). Hence $\hh^i_{\fa}(Z)=0$ for all $i<t$. Applying the functor $\Gamma_{\fa}(-)$ to the above exact sequence implies the isomorphism
\begin{equation}\tag{\ref{l7}.5}
\hh^i_{\fa}(M)\cong\hh^i_{\fa}(\Tr_{K}(X_1)) \text{ for all } \ i,\ 0\leq i\leq t-1.
\end{equation}
Now the assertion follows from (\ref{t4}.3), (\ref{t4}.4) and (\ref{t4}.5). 
\end{proof}
%%%%%%%%%%%%%%%%%%%%%%%%%%%%%%%%%%%%%%%%%%%%%%%%%%%%%%%%%%%%%%%%%%%%%
\begin{prop}\label{l8}
	Let $\fa$ be an ideal of $R$ and let $M$, $N$ be $R$-modules that are directly linked with respect to $\mx_n$. Assume that $t$ is an integer with  $n<t\leq\gr(\fa)$ such that $\Supp_R(\Ext^i_R(N,K))\subseteq\V(\fa)$ for all $i$, $n< i\leq n+t$. Then $\hh^i_{\fa}(M)\cong\Ext^{i+n}_R(N,K)$ for all $i$,  $0<i<t$. In particular, $\hh^i_{\fa}(M)$ is finitely generated for all $i$, $0< i< t$.
\end{prop}
\begin{proof}
	By Lemma \ref{l6}, there exists an ideal ${\bf x}\subseteq\ann_R(M)\cap\ann_R(N)$, generated by a regular sequence of length $n$, such that $M$ and $N$ are directly linked with respect to the category of totally $C$-reflexive $R/{\bf x}$-modules, where $C=\Ext^n_R(R/{\bf x},K)$. Hence by Lemma \ref{G1} and the independence Theorem (see \cite[Theorem 4.2.1]{BS}) we may assume that $n=0$ and $M$, $N$ are maximal $R$-modules which are directly linked with respect to $\mx_0$. By Lemma \ref{lll1}, there exists an exact sequence $0\to N\to\lambda_K^{\pi}M\to Y\to0$ where $\pi$ is a projective presentation of $M$ and $Y\in\mx_0$. The above exact sequence and (\ref{d3}.2) induce the isomorphisms
	\begin{equation}\tag{\ref{l8}.1}
	\Ext^i_R(N,K)\cong\Ext^i_R(\lambda_K^{\pi}M,K)\cong\Ext^{i+1}_R(\Tr_{K}M,K), \text{ for all } i>0.
	\end{equation}
	Note that, by Corollary \ref{c}, $\Ext^1_R(\Tr_{K}M,K)=0$.
	Hence, it follows from our assumption and (\ref{l8}.1) that 
	$M_\fp$ is an $t$-torsionfree $R_\fp$-module with respect to $K_\fp$ for all $\fp\in\Spec(R)\setminus\V(\fa)$. Now the assertion follows from Lemma \ref{l7} and (\ref{l8}.1).
\end{proof}
%%%%%%%%%%%%%%%%%%%%%%%%%%%%%%%%%%%%%%%%%%%%%%%%%%%%%%%
Let $R$ be a Cohen-Macaulay local ring. An $R$-module $M$ is generalized Cohen-Macaulay if and only if $M_\fp$ is a Cohen-Macaulay $R_\fp$-module for all $\fp\in\Spec(R)\setminus\{\fm\}$ (see \cite[Lemmas 1.2, 1.4]{T}). 
Therefore the following result may be seen as a generalization of the Schenzel's result \cite[Corollary 3.3]{Sc} for Gorenstein liaison of ideals as well as the results of Martsinkovsky and Strooker \cite[Theorem 11]{MS}, Nagel \cite[Corollary 6.1(b)]{Na} and Sadeghi \cite[Corollary 5.4]{Sa1} for their smaller module liaison classes.
%%%%%%%%%%%%%%%%%%%%%%%%%%%%%%%%%%%%%%%%%%%%%%%%%%%%%%%%%%%
\begin{cor}\label{c6}
	Let $(R,\fm,k)$ be a Cohen-Macaulay local ring which is a homomorphic image of a Gorenstein local ring. Assume that $M$, $N$ are $R$-modules of dimension $d$ which are linked with respect to $\cm^n(R)$ in an odd number of steps. Let $t$ be a positive integer such that $M_\fp$ satisfies $(S_t)$ for all
	$\fp\in\Spec R\setminus\{\fm\}$. Then
	$\hh^i_{\fm}(M)\cong\Hom_R(\hh^{d-i}_{\fm}(N),\E_R(k))
	\text{ for all } i,\ 0<i<t$.
\end{cor}
\begin{proof}
	Without loss of generality we may assume that $M$ and $N$ are directly linked. As $R$ is Cohen-Macaulay, $\dim M_\fp=\dim R_\fp-\gr_{R_\fp}(M_\fp)=\depth R_\fp-n$ for all $\fp\in\Spec R$, where the last equality follows from Proposition \ref{p2}. Hence, by our assumption and Proposition \ref{t3}, we have $\Ext^i_R(N,\omega)$ has finite length for all $i$, $n+1\leq i\leq n+t-1$, where $\omega$ is the dualizing module of $R$. Now the assertion follows from Proposition \ref{l8} and the local duality theorem.
\end{proof}
%%%%%%%%%%%%%%%%%%%%%%%%%%%%%%%%%%%%%%%%%%%%%%%%%%%%
The following shows that homological dimensions are preserved in even module liaison classes which can be seen as a generalization of
\cite[Proposition 16]{MS}, \cite[Corollary 6.9]{Na} and \cite[Corollary 5.11(i)]{DS}.
\begin{thm}\label{ttt}
	Let $M$ and $N$ be $R$-modules. The following statements hold true.
	\begin{enumerate}[\rm(i)]
		\item If $M$ and $N$ are in the same even liaison class with respect to $\mx_n$, then $\Ext^i_R(M,K)\cong\Ext^i_R(N,K)$ for all $i>n$. Moreover, $\gkkd_R(M)=\gkkd_R(N)$.
		\item If $M$ and $N$ are in the same even liaison class with respect to $\mathcal{P}^n(R)$, then $\pd_R(M)=\pd_R(N)$. Moreover, $\cid_R(M)=\cid_R(N)$.	
	\end{enumerate}
\end{thm}
\begin{proof}
	Without loss of generality, we may assume that $M$ is linked to $N$ in two steps: $M\sim L \sim N$, for some $R$-module $L$. We only prove part (i). The proof of part (ii) is similar. By Corollary \ref{l1}, we have
	\begin{equation}\tag{\ref{ttt}.1}
	\Ext^{n+1}_R(M,K)\cong\coker\eta_K^R(L)\cong\Ext^{n+1}_R(N,K),
	\end{equation}
	\begin{equation}\tag{\ref{ttt}.2}
	\Ext^{n+i}_R(M,K)\cong\Ext^{n+i-1}_R(\Ext^n_R(L,K),K)\cong\Ext^{n+i}_R(N,K) \text{ for all } i>1.
	\end{equation}
	Consider the exact sequences 
	$0\to\Ext^n_R(L,K)\to X\to M\to0 \text{  and   }  0\to\Ext^n_R(L,K)\to X'\to N\to0,$ where $X, X'\in\mx_n=\G^n_K(\mathcal{P})$ (see Theorem \ref{t1}(iv)). It follows from the above exact sequences that
	\begin{equation}\tag{\ref{ttt}.3}
	\gkkd_R(M)<\infty \Leftrightarrow\gkkd_R(\Ext^n_R(L,K))<\infty\Leftrightarrow\gkkd_R(N)<\infty.
	\end{equation}	
	If $M$ is $\mathcal{G}_K$-perfect, then the assertion follows from Theorem \ref{cc}. Hence we may assume that $n<\gkkd_R(M)<\infty$.
	Now the assertion follows from (\ref{ttt}.1), (\ref{ttt}.2), (\ref{ttt}.3) and Theorem \ref{G3}(i).
\end{proof}
%%%%%%%%%%%%%%%%%%%%%%%%%%%%%%%%%%%%%%%%%%%%%%%%%%%%%%%
%%%%%%%%%%%%%%%%%%%%%%%%%%%%%%%%%%%%%%%%%%%%%%%%%%%%%%%%
Now we can state our first main result, which is nothing but Theorem \ref{it},  mentioned in the introduction.
\begin{cor}\label{c3}
	Let $(R,\fm,k)$ be a Cohen-Macaulay local ring of dimension $d$ with dualizing module and let $M$, $N$ be $R$-modules which are in the same liaison class with respect to $\cm^n(R)$.
	\begin{enumerate}[\rm(i)]
		\item $M\in\cm^n(R)$ if and only if $N\in\cm^n(R)$.
		\item Assume that $\X$ is an open subset of $\Spec R$. If $M$ and $N$ are linked in an odd number of steps, then
		$$\X\subseteq\Sn_t(M) \text{ if and only if } \Att_R(\hh^j_\fm(N))\subseteq\Spec R\setminus\X \text{ for } \dim N-t<j<\dim N.$$
		\item If $M$ and $N$ are in the same even liaison class, then $\hh^i_{\fm}(M)\cong\hh^i_{\fm}(N)$ for $0<i<d-n$.
		\item If $M$ and $N$ are linked in an odd number of steps and $M$ is generalized Cohen-Macaulay, then $$\hh^i_{\fm}(M)\cong\Hom_R(\hh^{d-n-i}_{\fm}(N),\E_R(k)) \text{ for } 0<i<d-n.$$
		In particular, $N$ is generalized Cohen-Macaulay.
	\end{enumerate}
\end{cor}
\begin{proof}
	This is an immediate consequence of Theorems \ref{cc}(i), \ref{ttt}, Proposition \ref{t3}, Corollary \ref{c6} and the local duality theorem.
\end{proof}
%%%%%%%%%%%%%%%%%%%%%%%%%%%%%%%%%%%%%%%%%%%%%%%%%%%%%%%%
An $R$-module $M$ of finite $\gkkkd$-dimension is called \emph{reduced $\gkkkd$-perfect} provided that $\Ext^i_R(M,K)=0$ for all $i\neq\gr_R(M), \gkkd_R(M)$. The following is a generalization of \cite[Theorem 3.3]{DS}.
\begin{thm}\label{t4}
	Let $R$ be a Cohen-Macaulay local ring of dimension $d$ and let $M$ and $N$ be $R$-modules which are linked in an odd number of steps with respect to $\mx_n$. If $M$ is reduced $\gkkkd$-perfect, then
	$$\depth_R(M)+\depth_R(N)=\dim_R(M) +\depth_R(\Ext^{\tiny{\gkkd_R(M)}}_R(M,K)).$$	
\end{thm}
\begin{proof}
	Without loss of generality we may assume that $M$ and $N$ are directly linked. Set $t:=\gkkd_R(M)$. If $M$ is $\gkkkd$-perfect, then by using the fact that $R$ is Cohen-Macaulay, we have $\dim M=\dim R-\gr_R(M)=\depth R-t$. Now the assertion follows from Theorems \ref{G3}(iii), \ref{cc} and Lemma \ref{ll1}(i). Hence we may assume that $\gr_R(M)<\gkkd_R(M)$.
	As $M$ is linked to $N$, by Theorem \ref{t1}(iv) and Corollary \ref{c}, one obtains the exact sequence
	\begin{equation}\tag{\ref{t4}.1}
	0\to\Ext^n_R(N,K)\to X\to M\to0,\
	\text{for some}\ X\in\mx_n.\end{equation}
	Hence $\gkkd_R(\Ext^n_R(N,K))<\infty$.
	Note that, by Theorem \ref{G3}(iii) and Proposition \ref{p2},
	\begin{equation}\tag{\ref{t4}.2}
	\dim_R(M)-\depth_R(M)=d-n-\depth_R(M)=t-n.
	\end{equation}
	We argue by induction on $s=t-n$. If $s=1$, then $t=n+1$. By Corollary \ref{l1} and Theorem \ref{G3}(i), $\Ext^i(\Ext^n_R(N,K),K)=0$ for all $i>n$ and one has the exact sequence
	\begin{equation}\tag{\ref{t4}.3}
	0\to N\to\Ext^n_R(\Ext^n_R(N,K),K)\to\Ext^t_R(M,K)\to0.
	\end{equation}
	Therefore, by Lemma \ref{l} and Theorem \ref{G3}(i), $\Ext^n_R(N,K)$
	is $\gkkkd$-perfect of grade $n$ and so is $\Ext^n_R(\Ext^n_R(N,K),K)$ by Lemma \ref{ll1}. Note that $\gr_R(\Ext^t_R(M,K))\geq t$ (see \cite[Corollary 30]{Ma}). Hence 
    \begin{equation}\tag{\ref{t4}.4}
	\begin{array}{rllll}	\depth_R(\Ext^t_R(M,K))&\leq\dim_R(\Ext^t_R(M,K))\\ &\leq d-t<d-n\\
	&=\depth_R(\Ext^n_R(\Ext^n_R(N,K),K)),
	\end{array}
	\end{equation}
	where the last equality follows from the fact that $\Ext^n_R(\Ext^n_R(N,K),K)$ is $\gkkkd$-perfect of grade $n$ and Theorem \ref{G3}(iii). Hence, by (\ref{t4}.2), (\ref{t4}.4) and the exact sequence (\ref{t4}.3), we find that
	\begin{align*}
	\depth_R(N)&=\depth_R(\Ext^t_R(M,K))+1=\depth_R(\Ext^t_R(M,K))+(t-n)\\ &=\depth_R(\Ext^t_R(M,K))+\dim_R(M)-\depth_R(M),
	\end{align*}
	as desired.
	Let $s>1$ and set $Z=\Ext^n_R(N,K)$.
	By Lemma \ref{l} and Example \ref{e}(ii), $Z$ is linked by some $\phi\in\Epi(\mx_n)$. Set $\L^n_K(Z):=\im\L^n_K(\phi)$ and consider the following exact sequence 
	\begin{equation}\tag{\ref{t4}.5}
	0\to\Ext^n_R(Z,K)\to Y\to\L_K^n(Z)\to0,\ \text{ where}\ Y\in\mx_n\  \text{(see \ref{d2}.1)}.
	\end{equation}
	Next, we prove that $\L_K^n(Z)\notin\mx_n$. Assume contrarily that $\L_K^n(Z)\in\mx_n$. Hence $Z\in\mx_n$ by Theorem \ref{cc}(i).
	As $s>1$ and $M$ is reduced $\gkkkd$-perfect, we have $\Ext^{n+1}_R(M,K)=0$. Therefore, it follows from Corollary \ref{l1}(i) that 
	\begin{equation}\tag{\ref{t4}.6}
	N\cong\Ext^n_R(\Ext^n_R(N,K),K)=\Ext^n_R(Z,K).
	\end{equation}
	Thus, by Lemma \ref{ll1}(i) and (\ref{t4}.6), $N\in\mx_n$ and so $M\in\mx_n$ by Theorem \ref{cc} which is a contradiction. Therefore $\L_K^n(Z)\notin\mx_n$ and so $\depth_R(\L_K^n(Z))<d-n=\depth_R(Y)$ by Theorem \ref{G3}(iii). It follows from the exact sequence (\ref{t4}.5) and the isomorphism (\ref{t4}.6) that
	\begin{equation}\tag{\ref{t4}.7}
	\depth_R(N)=\depth_R(\Ext^n_R(Z,K))=\depth_R(\L_K^n(Z))+1.
	\end{equation}
	By the exact sequence (\ref{t4}.1) we have
	\begin{equation}\tag{\ref{t4}.8}
	\depth_R(M)=\depth_R(Z)-1.
	\end{equation}
	Note that by Lemma \ref{l}, Theorem \ref{G3}(iii) and (\ref{t4}.8), $Z$ is a module of $\gkkkd$-dimension $t-1$ and of grade $n$. As $M$ is reduced $\gkkkd$-perfect and $s>1$, by Corollary \ref{l1}(ii), $$\Ext^{i}_R(Z,K)=\Ext^i_R(\Ext^n_R(N,K),K)\cong\Ext^{i+1}_R(M,K)=0 \text{ for all } i,\ n<i<t-1.$$ In other words, $Z$ is a reduced $\gkkkd$-perfect module.	Hence by induction hypothesis we have the equality
	\begin{equation}\tag{\ref{t4}.9}
	\depth_R(Z)+\depth_R(\L_K^n(Z))=\dim_R(Z)+\depth_R(\Ext^{t-1}_R(Z,K)).
	\end{equation}
	Note that Corollary \ref{l1}(ii) implies the isomorphism 
	\begin{equation}\tag{\ref{t4}.10}
	\Ext^t_R(M,K)\cong\Ext^{t-1}_R(Z,K).
	\end{equation}
	It follows from Lemma \ref{l} and Proposition \ref{p2} that $\dim_R(M)=\dim_R(Z)$.
	Now the assertion is clear by (\ref{t4}.7), (\ref{t4}.8), (\ref{t4}.9) and (\ref{t4}.10).
\end{proof}
%%%%%%%%%%%%%%%%%%%%%%%%%%%%%%%%%%%%%%%%%%%%%%%%%
Let $R$ be a Gorenstein local ring. Following \cite{H}, an $R$-module $M$ is said to be an \emph{Eilenberg-Maclane} module, if $\hh^i_\fm(M)=0$ for all $i\neq\depth_R(M), \dim_R(M)$. Hence
reduced $\gkkkd$-perfect modules can be viewed as a generalization of Eilenberg-Maclane modules.
\begin{cor}
	Let $R$ be a Cohen-Macaulay local ring with dualizing module $K$ and let $M$ and $N$ be $R$-modules which are linked in an odd number of steps with respect to $\cm^n(R)$. If $M$ is an Eilenberg-Maclane module, then $\depth_R(M)+\depth_R(N)=\dim_R(M)+\depth_R(\Ext^{\tiny{\gkkd_R(M)}}_R(M,K)).$
\end{cor}
\begin{proof}
	This is an immediate consequence of Theorem \ref{t4} and the local duality theorem.
\end{proof}
%%%%%%%%%%%%%%%%%%%%%%%%%%%%%%%%%%%%%%%%%%%%%%%%%
\section{Colinkage of modules}
%%%%%%%%%%%%%%%%%%%%%%%%%%%%%%%%%%%%%%%%%%%%%%%%%%%%%%%%%%%%%%%%%%%%%%%%%%%%%%%%%%%%%%%%%%%%%%%%%%%%%%%%%%%%%%%%%%%%%%%%%%%%%%%%%%%%%%
In this section we introduce and study the notion of colinkage of modules. This notion can be seen as the dual of the notion of linkage and enables us to study the theory of linkage for modules in the Bass class with respect to $K$. It is shown that every grade-unmixed module in the Bass class with respect to $K$ can be colinked with respect to the category of $\mp_K$-perfect modules. An adjoint equivalence between the linked modules with respect to the category of perfect modules and the colinked modules with respect to the category of $\mp_K$-perfect modules is established. We start by recalling some definitions, notations and
results which will be used in this section.

%%%%%%%%%%%%%%%%%%%%%%%%%%%%%%%%%%%%%%%%%%%%%%%%%%%%%%%%%%%%%%%%%%%%%%%%%%%%%
\begin{chunk}\textbf{Gorenstein injective dimension.}
The notion of Gorenstein injective dimension of a module has been introduced by Enochs and Jenda,  as a dual version of the notion of Gorenstein projective dimension \cite{EJ}.
An $R$-module $M$ is said to be \emph{Gorenstein injective} if there exists an  exact sequence
$I_{\bullet}:\ \ \cdots \to I_{1} \stackrel{\rm \partial_{1}}{\longrightarrow} I_{0} \stackrel{\rm \partial_{0}}{\longrightarrow}  I_{-1} \to \cdots$ of injective $R$-modules such that $M\cong \ker(\partial_{0})$ and $\Hom_{R}(E,I_{\bullet})$ is exact for any injective $R$-module $E$. The Gorenstein injective dimension of $M$, $\gid_R(M)$, is defined as the infimum of $n$ for which there exists an exact sequence $0 \to M \to J_{0} \to \cdots \to J_{-n}\to0$, where  each $J_{i}$ is Gorenstein injective.

The Gorenstein injective dimension is a refinement of the classical injective dimension, i.e. $\gid(M)\leq \id(M)$, with equality if $\id(M)<\infty$. Note that every module over a Gorenstein ring has finite Gorenstein injective dimension.
\end{chunk}
%%%%%%%%%%%%%%%%%%%%%%%%%%%%%%%%%%%%%%%%%%%%%%%%%%%%
\begin{chunk}\textbf{Homological dimensions with respect to a semidualizing module} (\cite{TW}, \cite{WSW}).\\
The class of $K$-projective modules is defined as
$\mp_K(R)=\{P\otimes_RK\mid P\in\mp(R) \}.$ 
The $\mp_K$-dimension of $M$, denoted
$\pkd_R(M)$, is less than or equal to $t$ if and
only if there is an exact sequence
$$0\rightarrow X_t\to\cdots\rightarrow X_0\rightarrow M\rightarrow0,$$
where $X_i\in\mp_K(R)$ for each $i$ \cite[Corollary 2.10]{TW}. Note that if $R$ is local and $K$ is a dualizing module, then $\pkd_R(M)<\infty$ if and only if $\id_R(M)<\infty$. As every $K$-projective module is a totally $K$-reflexive module, we have $\gkkd_R(M)\leq\pkd_R(M)$ with equality when the right hand side is finite.

An exact complex in $\mp_K(R)$ is called \emph{totally $\mp_K$-acyclic} if it is $\Hom_R(\mp_K(R),-)$-exact
and $\Hom_R(-,\mp_K(R))$-exact. We denote by $\G(\mp_K)$ the subcategory of $\md R$ with objects of the form
$M\cong\coker(\partial^X_1)$ for some totally $\mp_K$-acyclic complex $X$. The $\G(\mp_K)$-dimension of $M$ is defined as
$$\gpkd_R(M)=\inf\{\sup\{ i\geq0\mid X_i\neq0\}\mid X \text{ is an } \G(\mp_K)\text{-resolution of } M\}.$$
%%%%%%%%%%%%%%%%%%%%%%%%%%%%%%%%%%%%%%%%%%%%%%%%%%%%%
\end{chunk}
\begin{defnot}
	An $R$-module $M$ is called $\mp_K$-\emph{perfect} ( resp. $\G(\mp_K)$-\emph{perfect}) provided that $\gr_R(M)=\pkd_R(M)$ ( resp. $\gr_R(M)=\gpkd_R(M)$).	We denote by $\mp_K^n(R)$ ( resp. $\G^n_K(\mp_K)$ ) the subcategory of $\md R$, consisting of all $\mp_K$-perfect ( resp. $\G(\mp_K)$-perfect ) $R$-modules of grade $n$.
	Note that, in the trivial case, i.e. $K=R$, we omit the subscript and recover the subcategory of perfect 
	( resp. $\mathcal{G}$-perfect) modules of grade $n$, $\mp^n(R)$ ( resp. $\G^n_R(\mp)$).
\end{defnot}
In the following we collect some basic properties and examples of modules in the Auslander class $\ak(R)$ and the Bass class $\bk(R)$ (see Definition \ref{def}) which will be used in the rest of this paper. For an $R$-module $M$, we denote $M^{\blacktriangledown}=\Hom_R(K,M)$.

\begin{thm}\label{example1} Assume that M is an R-module. The following statements hold true.
	\begin{enumerate}[\rm(i)]
		\item If $R$ is local and $K$ is a dualizing module, then we have
		$M\in\mathcal{A}_{K}(R)$ (respectively, $M\in\mathcal{B}_{K}(R)$) if and only if $\gd_R(M)<\infty$ (respectively, $\gid_R(M)<\infty$)
		\cite[Theorem 1]{F1}.
		\item If $M$ has a finite
		$\mp_K$-dimension, then $M\in\mathcal{B}_K(R)$
		\cite[Corollary 2.9]{TW}.
		\item (Foxby equivalence) There are (horizontal) adjoint equivalences
		\[\xymatrix{
			\mp^n(R) \ar@{^(->}[dd]\ar@<1ex>[rrr]^{-\otimes_RK}_{\sim} & & & \pk^n(R)\ar@{^(->}[dd] \ar@<1ex>[lll]^{\Hom_R(K,-)}\\ \\
			\ak(R) \ar@<1ex>[rrr]^{-\otimes_RK}_{\sim} & & & \bk(R) \ar@<1ex>[lll]^{\Hom_R(K,-)}
		}\]
		of categories \cite[Theorem 2.12]{TW}.
		\item $\gpkd_R(M)<\infty$ if and only if $M\in\bk(R)$ and $\gkkd_R(M)<\infty$
		\cite[Lemma 2.9]{WSW}. In particular, if $K$ is dualizing, then $\gpkd_R(M)<\infty$ if and only if
		$\gid_R(M)<\infty$.	
		\item If $M\in\ak(R)$, then $\depth_R(M)=\depth_R(M\otimes_RK)$ and $\dim_R(M)=\dim_R(M\otimes_RK)$. In particular, $M\in\cm^n(R)$ if and only if $M\otimes_RK\in\cm^n(R)$ \cite[Lemma 2.11]{DS1}.
		\item If $M\in\bk(R)$, then $\depth_R(M)=\depth_R(M^{\blacktriangledown})$ and $\dim_R(M)=\dim_R(M^{\blacktriangledown})$. In particular, $M\in\cm^n(R)$ if and only if $M^{\blacktriangledown}\in\cm^n(R)$ \cite[Lemma 3.5]{DS2}.
		\item If $R$ is local and $K$ is a dualizing module, then there is an adjoint equivalence
			\[\xymatrix{
			\G^n_R(\mp)\ar@<1ex>[rrr]^{-\otimes_RK}_{\sim} & & & \G_K^n(\mp_K)\ar@<1ex>[lll]^{\Hom_R(K,-)}
		}\]of categories (which is an immediate consequence of the previous parts of the Theorem).
	\end{enumerate}
\end{thm}
%%%%%%%%%%%%%%%%%%%%%%%%%%%%%%%%%%%%%%%%%%%%%%%%%
Note that, by Theorem \ref{example1}, if $R$ is local and $K$ is a dualizing module, then we have
\begin{equation}\tag{\ref{example1}.1}
\G^n_{K}(\mp_K)=\{M\in\cm^n(R)\mid M\text{ has finite Gorenstein injective dimension}\}.
\end{equation}
%%%%%%%%%%%%%%%%%%%%%%%%%%%%%%%%%%%%%%%%%%%%%%%
In the following, we present another equivalence between the category of $\mathcal{G}$-perfect modules and the category of $\G(\mp_K)$-perfect modules.
\begin{lem}\label{cll}
For an integer $n\geq0$, the following statements hold true.
\begin{enumerate}[\rm(i)]
	\item There is an equivalence
	\[\xymatrix{
		\mp^n(R)\ar@<1ex>[rrr]^{\Ext^n_R(-,K)}_{\sim} & & & \mp_K^n(R)\ar@<1ex>[lll]^{\Ext^n_R(-,K)}.
	}\]
	\item If $R$ is local and $K$ is a dualizing module, then there is an equivalence
	\[\xymatrix{
		\G^n_R(\mp)\ar@<1ex>[rrr]^{\Ext^n_R(-,K)}_{\sim} & & & \G_K^n(\mp_K)\ar@<1ex>[lll]^{\Ext^n_R(-,K)}.
	}\]
\end{enumerate}	
\end{lem}
\begin{proof}
(i). First note that $\mp^n(R),  \mp_K^n(R)\subseteq\G_K^n(\mp)$. Hence, by Lemma \ref{ll1}, $M\cong\Ext^n_R(\Ext^n_R(M,K),K)$ for all $M\in \mp^n(R)\cup\mp_K^n(R)$. Let $M$ be a perfect module of grade $n$. By \cite[Lemma 2.44]{AB}, we have 
\begin{equation}\tag{\ref{cll}.1}
\Ext^n_R(M,K)\cong\Ext^n_R(M,R)\otimes_RK.
\end{equation}
Note that $\Ext^n_R(M,R)\in\mp^n(R)$ and so,
by (\ref{cll}.1) and Theorem \ref{example1}(iii), we have $\Ext^n_R(M,K)\in\mp_K^n(R)$.

Conversely, assume that $M\in\mp_K^n(R)$. It follows from Theorem \ref{example1}(ii) that $M\in\bk(R)$. Hence, by \cite[Theorem 4.1 , Corollary 4.2]{TW}, we have
\begin{equation}\tag{\ref{cll}.2}
\Ext^n_R(M,K)\cong\Ext^n_R(M^{\blacktriangledown},R).
\end{equation}
By Theorem \ref{example1}(iii), $M^{\blacktriangledown}\in\mp^n(R)$ and so $\Ext^n_R(M^{\blacktriangledown},R)\in\mp^n(R)$. Now the assertion follows from (\ref{cll}.2).

(ii). It follows from Theorem \ref{example1} that	$\G^n_R(\mp)=\ak(R)\cap\cm^n(R)$ and $\G_K^n(\mp_K)=\bk(R)\cap\cm^n(R)$. Note that $M\cong\Ext^n_R(\Ext^n_R(M,K),K)$ for all $M\in\cm^n(R)$ (see for example \cite[Theorem 3.3.10]{BH}). Assume that $M\in\G^n_R(\mp)$. We want to show that $\Ext^n_R(M,K)\in\G_K^n(\mp_K)$. By \cite[Proposition 3.3.3(b)]{BH}, $\Ext^n_R(M,K)\in\cm^n(R)$. Hence it is enough to show that $\Ext^n_R(M,K)\in\bk(R)$. As $\gd_R(M)=n$, it follows from Theorem \ref{G3}(ii) that $\Tr\Omega^nM$ is a totally reflexive module. In other words, by Theorem \ref{example1}(i), $\Tr\Omega^nM\in\ak(R)$. In particular, $\Tor_i^R(\Tr\Omega^nM,K)=0$ for $i>0$. It follows from \cite[Theorem 2.8]{AB} that $\Ext^n_R(M,K)\cong\Ext^n_R(M,R)\otimes_RK$. By Lemma \ref{ll1} and Theorem \ref{example1}(i), $\Ext^n_R(M,R)\in\ak(R)$. Hence $\Ext^n_R(M,K)\in\bk(R)$ by Theorem \ref{example1}(iii).

Conversely, assume that $M\in\G_K^n(\mp_K)$. As $M\in\cm^n(R)$, we have $\Ext^n_R(M,K)\in\cm^n(R)$ by \cite[Proposition 3.3.3(b)]{BH}. Hence we only need to show that $\Ext^n_R(M,K)\in\ak(R)$. By Theorem \ref{example1}(vii), $M^{\blacktriangledown}\in\G_R^n(\mp)$ and so by Lemma \ref{ll1} $\Ext^n_R(M^{\blacktriangledown},R)\in\G_R^n(\mp)$. In other words, by Theorem \ref{example1}(i), $\Ext^n_R(M^{\blacktriangledown},R)\in\ak(R)$.
As $M,K\in\bk(R)$, by \cite[Theorem 4.1 , Corollary 4.2]{TW}, we have
$\Ext^n_R(M,K)\cong\Ext^n_R(M^{\blacktriangledown},R)\in\ak(R)$.
\end{proof}
%%%%%%%%%%%%%%%%%%%%%%%%%%%%%%%%%%%%%%%%%%%%%%%%%%%%%%%%%%%%%%
For an integer $n\geq0$, set $\DnK(-):=\Ext^n_R(\Hom_R(K,-),K)$.
\begin{thm}\label{t2}
	Assume that $M$ is an $R$-module of grade $n$ and that the evaluation $R$-homomorphism $\nu_K^R(M):K\otimes_R\Hom_R(K,M)\to M$ is an isomorphism (e.g. $M\in\bk(R)$). Then there exists an $R$-homomorphism $\xi_K^R(M):M\to\DnK\DnK(M)$ with the following properties.
	\begin{enumerate}[\rm(i)]
		\item There is the following commutative diagram.
		$$\begin{CD}
		\ \ &&&& M@>\eta_K^R(M)>>\Ext^n_R(\Ext^n_R(M,K),K)&\\
		&&&& @VV{\|}V @VV{\cong}V \\
		\ \ &&&&M@>\xi_K^R(M)>>\DnK\DnK(M)&\\
		\end{CD}$$\\	
		In particular, $\ker\eta_K^R(M)\cong\ker\xi_K^R(M)$	and $\coker\eta_K^R(M)\cong\coker\xi_K^R(M)$.	
		\item If $M$ is $\P_K$-perfect, then $\xi_K^R(M)$ is an isomorphism. Moreover, $\DnK(M)$ is $\P_K$-perfect of grade $n$.
		\item If $R$ is Cohen-Macaulay local ring, $K$ is dualizing module and $M$ is Cohen-Macaulay $R$-module of finite Gorenstein injective dimension, then $\xi_K^R(M)$ is an isomorphism. Moreover, $\DnK(M)\in\cm^n(R)$ and has finite Gorenstein dimension.
	\end{enumerate}
\end{thm}
\begin{proof}
	(i). Let ${\bf x}\subseteq\ann_R(M)$ be an ideal, generated by a regular sequence of length $n$. Set $S=R/{\bf x}$ and $\overline{(-)}:=-\otimes_RS$.
	It follows from standard isomorphism and our assumption that $M\cong K\otimes_R\Hom_R(K,M)\cong\overline{K}\otimes_S\Hom_S(\overline{K},M)$.  Therefore we have the isomorphisms
	\[\begin{array}{rl}\tag{\ref{t2}.1}
	\Ext^n_R(\Ext^n_R(M,K),K)&\cong\Hom_S(\Hom_S(M,\overline{K}),\overline{K})\\
	&\cong \Hom_S(\Hom_S(\overline{K}\otimes_S\Hom_S(\overline{K},M),\overline{K}),\overline{K})\\
	&\cong \Hom_S(\Hom_S(\overline{K},\Hom_S(\Hom_S(\overline{K},M),\overline{K})),\overline{K})\\
	&\cong\Ext^n_R(\Hom_R(K,\Ext^n_R(\Hom_R(K,M),K),K)
	\\
	&=\DnK\DnK(M).\\
	\end{array}\]
	We denote by $\sigma:\Ext^n_R(\Ext^n_R(M,K),K)\to\DnK\DnK(M)$ the isomorphism (\ref{t2}.1) and define the homomorphism $\xi_K^R(M):M\to\DnK\DnK(M)$ to be the composition $\sigma\circ\eta_K^R(M)$. 
	Now the assertion is clear.
	
	(ii). First note that every $\P_K$-perfect module is also $\gkkkd$-perfect. It follows from Corollary \ref{l1} that $\eta_K^R(M)$ is isomorphism and so is $\xi_K^R(M)$ by part (i). Now we show that $\mathcal{D}^n_K(M)\in\mp^n_K(R)$.
	As $M$ is a $\P_K$-perfect module of grade $n$, we have $M^{\blacktriangledown}$ is perfect module of grade $n$ by Theorem \ref{example1}(iii). It follows from Lemma \ref{cll}(i) that $\DnK(M)=\Ext^n_R(M^{\blacktriangledown},K)\in\mp_K^n(R)$.
	
	(iii). It follows from \cite[Theorem 3.3.10(c)]{BH} that $\eta_K^R(M)$ is an isomorphism and so is $\xi_K^R(M)$ by part (i). Note that by
	(\ref{example1}.1),  $M\in\G^n_K(\mp_K)$. Thus we have $M^{\blacktriangledown}\in\G^n_R(\mp)$ by Theorem \ref{example1}(vii). It follows from Lemma \ref{cll}(ii) that $\DnK(M)=\Ext^n_R(M^{\blacktriangledown},K)\in\G^n_K(\mp_K)$.
	In other words, by (\ref{example1}.1), $\DnK(M)$ is Cohen-Macaulay of grade $n$ and has finite Gorenstein injective dimension.	
\end{proof}
%%%%%%%%%%%%%%%%%%%%%%%%%%%%%%%%%%%%%%%%%%%%%%%%%%%%%%%%%%%%%
\begin{defnot}
We denote by $\corefnk$ the subcategory of $\md R$ consisting of all $R$-modules $M$ with grade $n$ such that $\xi_K^R(M):M\to\Dn_K(\Dn_K(M))$ is an isomorphism.

A subcategory $\mx\subseteq\corefnk$ is called $n$-{\em coreflexive subcategory with respect to $K$} if it is closed under $\Dn_K(-)$ (i.e. for any $R$-module $M$, if $M\in\mx$ then $\Dn_K(M)\in\mx$).	
\end{defnot}
Here are some examples of coreflexive subcategories with respect to $K$ (see Theorem \ref{t2}(ii), (iii)).
%%%%%%%%%%%%%%%%%%%%%%%%%%%%%%%%%%%%%%%%%%%%%%%%%%%%%%%%%%%%%%%
\begin{eg}\label{exa}
	The following subcategories of $\md R$ are $n$-coreflexive with respect to $K$.
	\begin{enumerate}[(i)]
		\item $\mathcal{P}^n_K(R)$, i.e. the subcategory of $\md R$ consisting of all $\mathcal{P}_K$-perfect modules of grade $n$. In particular, if $R$ is Cohen-Macaulay local with dualizing module $K$, then $$\my=\{M\in\cm^n(R)\mid M \text{ has finite injective dimension }\}$$ is an $n$-coreflexive subcategory with respect to $K$.
		\item Let $R$ be a Cohen-Macaulay local ring with dualizing module $K$. Then $\G^n_{K}(\mp_K)$ is an $n$-coreflexive subcategory with respect to $K$.
	\end{enumerate}	
\end{eg}
%%%%%%%%%%%%%%%%%%%%%%%%%%%%%%%%%%%%%%%%%%%%%%%%%%%%%%%%%%%%%%%%
Let $\my$ be an $n$-coreflexive subcategory with respect to $K$ and let $\phi\in\Epi(\my)$. Hence $\phi:Y\twoheadrightarrow M$ is an epimorphism for some $M\in\md R, Y\in\my$ with $\gr_R(M)=\gr_R(Y)$. Such a homomorphism is called a \emph{coreflexive homomorphism}. Given $\phi\in\Epi(\my)$, we want to construct a new coreflexive homomorphism $\mathcal{L}^n_K(\phi)$.

\begin{dfn}\label{cd2}
	Let $\my$ be an $n$-coreflexive subcategory with respect to $K$ and let $\phi\in\Epi(\my)$. Consider the exact sequence $0\rightarrow\ker\phi\overset{i}{\rightarrow} Y\overset{\phi}{\rightarrow}\im\phi \rightarrow 0$, where $Y\in\my$.
	Applying the functor $(-)^{\blacktriangledown}=\Hom_{R}(K,-)$ gives the exact sequence
	$0\rightarrow(\ker\phi)^{\blacktriangledown}\overset{i^{\blacktriangledown}}{\rightarrow} Y^{\blacktriangledown}\rightarrow(\im\phi)^{\blacktriangledown}$.
	Note that $\gr_R(\ker\phi)=n=\gr_R((\ker\phi)^{\blacktriangledown})$. We denote by $\ml_K^n(\phi):\Ext^{n}_{R}(Y^{\blacktriangledown}, K)\twoheadrightarrow\im(\Ext^n_R(i^{\blacktriangledown},K))$ the epimorphism induced by $\Ext^n_R(i^{\blacktriangledown},K)$.
	Therefore we have the exact sequence
	\begin{equation}\tag{\ref{cd2}.1}
	0\rightarrow\ker\ml_K^n(\phi)\longrightarrow\DnK(Y)  \overset{\ml_K^n(\phi)}{\longrightarrow}\im\ml_K^n(\phi)\rightarrow0.
	\end{equation}
	Note that if $\Ext^1_R(K,\ker\phi)=0$, then $\ker\ml_K^n(\phi)\cong\DnK(\im\phi)$.
\end{dfn}
%%%%%%%%%%%%%%%%%%%%%%%%%%%%%%%%%%%%%%%%%%%%%%%%%%%%%%%%
%%%%%%%%%%%%%%%%%%%%%
\begin{lem}\label{cl}
	Let $\my\subseteq\bk(R)$ be an $n$-coreflexive subcategory with respect to $K$ and let $\phi\in\Epi(\my)$ be a homomorphism which is not injective. If $\nu_K^R(\im\phi):K\otimes_R\Hom_R(K,\im\phi)\to\im\phi$ is injective, then $\Ext^1_R(K,\ker\phi)=0$. In particular, $\ker\ml_K^n(\phi)\cong\DnK(\im\phi)$.
\end{lem}
\begin{proof}
	Consider the exact sequence $0\to\ker\phi\to Y\overset{\phi}{\to}\im\phi\to0$, where $Y\in\my$.
	As $Y\in\bk(R)$, we have $\Ext^1_R(K,Y)=0$. Hence, 
	applying the functor $(-)^{\blacktriangledown}=\Hom_R(K,-)$, implies the exact sequence $Y^{\blacktriangledown}\to(\im\phi)^{\blacktriangledown}\to\Ext^1_R(K,\ker\phi)\to0$, from which we obtain the induced
	commutative diagram 
	$$\begin{CD}
	\ \ &&&&  K\otimes_RY^{\blacktriangledown} @>\phi'>> K\otimes_R(\im\phi)^{\blacktriangledown}@>>> K\otimes_R\Ext^1_R(K,\ker\phi)@>>>0&  \\
	&&&& @VV{\nu_K^R(Y)}V @VV{\nu_K^R(\im\phi)}V \\
	\ \  &&&& Y @>>> \im\phi @>>>0&\\
	\end{CD}$$\\
	with exact rows. By using the fact that $\nu_K^R(\im\phi)$ is injective and that $\nu_K^R(Y)$ is an isomorphism, we see that $\phi'$, in the above diagram, is surjective which is equivalent to say that $K\otimes_R\Ext^1_R(K,\ker\phi)=0$ and so $\Ext^1_R(K,\ker\phi)=0$.	
\end{proof}
%%%%%%%%%%%%%%%%%%%%%%%%%%%%%%%%%%%%%%%%%%%%%%%%%%%%%%
We denote by $\Delta_K(R):=\{M\in\md R\mid \nu_K^R(M):K\otimes_R M^{\blacktriangledown}\to M \text{ is an isomorphism }\}$. Note that
$M\in\Delta_K(R)$ if and only if it has a $\mp_K$-presentation, i.e. there exists an exact sequence
$X_1\to X_0\to M\to0$, where $X_i\in\mp_K(R)$ for $i=1,2$ (see for example \cite[Proposition 3.6]{XZ}). Clearly $\bk(R)\subseteq\Delta_K(R)$.

%%%%%%%%%%%%%%%%%%%%%%%%%%%%%%%%%%%%%%%%%%%%%%%%%%%%
Here are some basic properties of $\ml_K^n(\phi)$.
\begin{thm}\label{ct1}
	Let $\my$ be an $n$-coreflexive subcategory with respect to $K$ and let $\phi\in\Epi(\my)$ be a homomorphism
	which is not injective. Then the following statements hold true.
	\begin{enumerate}[\rm(i)]
		\item $\ml_K^n(\phi)\in\Epi(\my)$.
		\item $\im\ml_K^n(\phi)$ is a grade-unmixed $R$-module.
		\item Assume that $\im\phi\in\Delta_K(R)$. If $\Ext^1_R(K,\ker\phi)=0$ (e.g. $\my\subseteq \bk(R)$, see Lemma \ref{cl}), then
		the image of map $\xi_K^R(\im\phi)$ is isomorphic to
		$\im\ml_K^n(\ml_K^n(\phi))$.
	\end{enumerate}	
\end{thm}
\begin{proof}
	(i) and (ii). First note that $\gr_R(\ker\phi)=n=\gr_R((\ker\phi)^{\blacktriangledown})$. Hence, by Lemma \ref{l}, $\DnK(\ker\phi)=\Ext^n_R((\ker\phi)^{\blacktriangledown},K)$ is grade-unmixed of grade $n$. It follows from the monomorphism $\im\ml_K^n(M)\hookrightarrow\DnK(\ker\phi)$ that $\im\ml_K^n(M)$ is a grade-unmixed $R$-module of grade $n$. As $\my$ is $n$-coreflexive with respect to $K$, we have $\DnK(Y)\in\my$ for all $Y\in\my$ and so $\ml_K^n(\phi)\in\Epi(\my)$.	
	
	(iii). Consider the exact sequence $0\to\ker\phi\to Y\overset{\phi}{\to}\im\phi\to0$, where  $Y\in\my$.
	Applying the functor $\DnK(-)$ and using the fact that $\Ext^1(K,\ker\phi)=0$, we get the exact sequence
	\begin{equation}\tag{\ref{ct1}.1}
	0\to\DnK(\im\phi)\longrightarrow\DnK(Y)\overset{\ml_K^n(\phi)}{\longrightarrow}\im\ml_K^n(\phi)\to 0, (\text{ see } \ref{cd2}.1)
	\end{equation}
	which induces the commutative diagram
	$$\begin{CD}
	%&&&&&&&&\\
	\ \ &&&& Y @>\phi>>\im\phi &  \\
	&&&& @VV{\xi_K^R(Y)}V @VV{\xi_K^R(\im\phi)}V \\
	\ \  &&&& \DnK\DnK(Y)@>\phi''>>\DnK\DnK(\im\phi),&\\
	\end{CD}$$\\
	where $\phi''=\DnK\DnK(\phi)$. As $Y\in\my$, we have $\xi_K^R(Y)$ is an isomorphism. It follows from the above commutative diagram that $\im(\xi_K^R(\im\phi))=\im(\phi'')=\im(\ml_K^n(\ml_K^n(\phi)))$.
\end{proof}
%%%%%%%%%%%%%%%%%%%%%%%%%%%%%%%%%%%%%%%%%%%%%%%%%%%%%%%%
%%%%%%%%%%%%%%%%%%%%%%%%%%%%%%%%%%%%%%%%%%%%%%%%%%%%%%%
\begin{lem}\label{clll}
	Let $\my$ be an $n$-coreflexive subcategory and let $\phi, \psi\in\Epi(\my)$. The following statements hold true (see Definition \ref{def2}).
	\begin{enumerate}[\rm(i)]
		\item{If $\phi\equiv\psi$, then
			$\ml_K^n(\phi)\equiv\ml_K^n(\psi)$. In particular, $\im\ml_K^n(\phi)\cong\im\ml_K^n(\psi)$.}
		\item{$\phi\equiv\ml_K^n(\ml_K^n(\phi))$ if and only if there exists 
			$\mu\in\Epi(\my)$ such that $\phi\equiv\ml_K^n(\mu)$ and $\mu\equiv\ml_K^n(\phi)$.}
	\end{enumerate}
\end{lem}
\begin{proof}
	(i). Let $\phi:X\twoheadrightarrow M$ and $\psi:Y\twoheadrightarrow N$, where $X,Y\in\my$ and $M, N\in\md R$. There exist isomorphisms $\alpha:M\overset{\cong}{\to} N$ and $\beta:X\overset{\cong}{\to} Y$ such that $\psi\circ\beta=\alpha\circ\phi$. Hence we obtain the commutative diagram
	$$\begin{CD}
	&&&&&&&&\\
	\ \ &&&&0@>>>\ker\phi@>>> X @>\phi>> M@>>>0 &  \\
	&&&&&&@VV{\gamma}V @VV{\beta}V @VV{\alpha}V \\
	\ \  &&&&0@>>>\ker\psi@>>> Y@>\psi>>N@>>>0,&\\
	\end{CD}$$\\
which, by the snake lemma, $\gamma$ is an isomorphism. Applying the functor $(-)^{\blacktriangledown}=\Hom_R(K,-)$ to the above diagram implies the commutative diagram
	$$\begin{CD}
	&&&&&&&&\\
	\ \ &&&&0@>>>(\ker\phi)^{\blacktriangledown}@>>>X^{\blacktriangledown} @>>>C@>>>0 &  \\
	&&&&&&@VV{\gamma^{\blacktriangledown}}V @VV{\beta^{\blacktriangledown}}V @VV{\alpha'}V \\
	\ \  &&&&0@>>>(\ker\psi)^{\blacktriangledown}@>>>Y^{\blacktriangledown}@>>>C'@>>>0&\\
	\end{CD}$$\\
	with exact rows, where $C=\im(\phi^{\blacktriangledown})$ and $C'=\im(\psi^{\blacktriangledown})$.
	As $\gamma^{\blacktriangledown}$, $\beta^{\blacktriangledown}$ are isomorphisms, so is $\alpha'$. The above diagram induces the 
	commutative diagram
	$$\begin{CD}
	&&&&&&&&\\
	\ \ &&&&0@>>>\Ext^n_R(C',K)@>>>\DnK(Y) @>\ml_K^n(\psi)>>\im\ml_K^n(\psi)@>>>0 &  \\
	&&&&&&@VV{\alpha''}V @VV{\DnK(\beta)}V @VV{\pi}V \\
	\ \  &&&&0@>>>\Ext^n_R(C,K)@>>>\DnK(X)@>\ml_K^n(\phi)>>\im\ml_K^n(\phi)@>>>0.&\\
	\end{CD}$$\\
	As $\alpha''=\Ext^n_R(\alpha',K)$ and $\DnK(\beta)$ are isomorphisms, so is $\pi$.
	Hence $\ml_K^n(\phi)\equiv\ml_K^n(\psi)$.
	
	(ii). follows from part (i).
\end{proof}
%%%%%%%%%%%%%%%%%%%%%%%%%%%%%%%%%%%%%%%%%%%%%%%%%%%%%%%
\begin{dfn}
	Let $\my$ be an $n$-coreflexive subcategory with respect to $K$ and let $\phi,\psi\in\Epi(\my)$. We say $R$-modules
	$M$ and $N$ are \emph{colinked with respect to} $\my$, in one step (directly), by the pair $(\phi,\psi)$ provided that the following conditions hold.
	\begin{enumerate}[\rm(i)]
		\item $M=\im\phi$ and $N=\im\psi$.
		\item $\phi\equiv\ml_K^n(\psi)$ and $\psi\equiv\ml_K^n(\phi)$.
	\end{enumerate}
    In this situation we write $M\overset{c}{\underset{(\phi,\psi)}{\sim}}N$,  or simply $M\overset{c}{\sim}N$.
	Equivalently, an $R$-module $M$ is said to be \emph{colinked} by $\phi$, if $M=\im\phi$ and $M\cong\im\ml_K^n(\ml_K^n(\phi))$ (see Lemma \ref{clll}).
\end{dfn}
%%%%%%%%%%%%%%%%%%%%%%%%%%%%%%%%%%%%%%%%%%%%%%%%%%%%%%%%%%%%
%%%%%%%%%%%%%%%%%%%%%%%%%%%%%%%%%%%%%%%%%%%%%%%%%%%
Here is a characterization of a colinked module in terms of the homomorphism  $\xi_K^R(-).$
%%%%
\begin{cor}\label{cc1}
	Let $\my\subseteq\bk(R)$ be an $n$-coreflexive subcategory with respect to $K$ and let $\phi$ be a non-injective homomorphism in $\Epi(\my)$.
	Assume that $M$ is an $R$-module such that $\im(\phi)=M$. If $M\in\Delta_K(R)$, then
	the following statements are equivalent.
	\begin{enumerate}[\rm(i)]
		\item $M$ is colinked by $\phi$.
		\item $\xi_K^R(M)$ is injective.
		\item $\eta_K^R(M)$ is injective.
	\end{enumerate}
\end{cor}
\begin{proof}
	This is an immediate consequence of Theorem \ref{t2}(i), Lemma \ref{cl}, Theorem \ref{ct1}.(iii) and \cite[Theorem 2.4]{Ma}.
\end{proof}
%%%%%%%%%%%%%%%%%%%%%%%%%%%%%
\begin{cor}\label{c11}
	Let $M\in\bk(R)$ be an $R$-module of grade $n$. Assume that $\gkkpd_{R_\fp}(M_\fp)<\infty$ for all $\fp\in\X^n(R)$ (e.g. $\id_{R_\fp}(K_\fp)<\infty$ for all $\fp\in\X^n(R)$). Then the following statements are equivalent.
	\begin{enumerate}[\rm(i)]
		\item $M$ is colinked with respect to $\mp^n_K(R)$.
		\item $M$ is grade-unmixed.	
		\item $\depth_{R_\fp}(M_\fp)\geq\min\{1,\depth R_\fp-n\}$ for all $\fp\in\Spec(R)$.
	\end{enumerate}	
\end{cor}
\begin{proof}
	Let ${\bf x}\subseteq\ann_R(M)$ be an ideal generated by a regular sequence of length $n$ and set $S=R/{\bf x}$.
	Choose a non-injective epimorphism $\phi:F\twoheadrightarrow M^{\blacktriangledown}$ where $F$ is a free $S$-module. As $M\in\bk(R)$, we have $\nu_{K}^R(M):K\otimes_{R}M^{\blacktriangledown}\to M$ is isomorphism. It is straightforward to see that $\nu_{K}^R(M)\circ(\phi\otimes_R K):F\otimes_{R}K\twoheadrightarrow M$ is a non-injective epimorphism in $\Epi(\mp^n_K(R))$. Now the assertion follows from Corollaries \ref{l4}, \ref{cc1}.
\end{proof}
%%%%%%%%%%%%%%%%%%%%%%%%%%%%%%%%%%%%%%
The following corollary is an immediate consequence of Corollaries \ref{c1}, \ref{c11} and Theorem \ref{example1}(iv).
\begin{cor}\label{cc3}
	Let $M$ be an $R$-module of finite $\G(\mp_K)$-dimension. The following are equivalent.
	\begin{enumerate}[\rm(i)]
		\item $M$ is linked with respect to $\mp^n(R)$.
		\item $M$ is colinked with respect to $\mp^n_K(R)$.
		\item $M$ is grade-unmixed of grade $n$.
	\end{enumerate}
\end{cor}
%%%%%%%%%%%%%%%%%%%%%%%%%%%%%%%%%%%%%%%%%%%%%%%%%%%%%%%%%%%%

%%%%%%%%%%%%%%%%%%%%%%%%%%%%%%%%%%%%%%%%%%%%%%%%%%%%%
%%%%%%%%%%%%%%%%%%%%%%%%%%%%%%%%%%%%%%%%%%%%%%%%%%%%%%%%
%%%%%%%%%%%%%%%%%%%%%%%%%%%%%%%%%%%%%%%%%%%%%%%%
%%%%%%%%%%%%%%%%%%%%%%%%%%%%%%%%%%%%%%%%%%%%%%%%%%%%%%%%
\begin{eg}\label{cee}
	
 Here we collect some examples of colinked modules.
		\begin{enumerate}[(i)]
			\item {\it Let $\my\subseteq\bk(R)$ be an $n$-coreflexive subcategory with respect to $K$ which is closed under direct sum. Every two modules in $\my$ are colinked directly. In particular, $M$ is directly colinked to $\overset{t}{\oplus}M$ for every $M\in\my$ and an integer $t$. Also, every $M\in\my$ is self-colinked.}
				
				This is an immediate consequence of the exact sequence
	$0\rightarrow\DnK(N) \rightarrow M\oplus\DnK(N) \rightarrow M\rightarrow0$ (see Theorem \ref{t2} and Corollary \ref{cc1}). 
	\item {\it Let $M$ be a grade-unmixed $R$-module of finite $G(\mp_K)$-dimension with grade $n$. Assume that $I\subseteq\ann_R(M)$ is a complete intersection ideal in $R$ of height $n$. If $\phi:F\twoheadrightarrow M^{\blacktriangledown}$ is a non-injective epimorphism where $F$ is a free $R/I$-module, then $M$ is colinked by $\psi$ with respect to $\mp^n_K(R)$ where $\psi=\nu_K^R(M)\circ(\phi\otimes K)$.}
	
	 This is an immediate consequence of Corollary \ref{c11} and Theorem \ref{example1}(iv).
	\item {\it Let $R$ be a Cohen-Macaulay local ring with dualizing module $K$. Let $M$ be an unmixed $R$-module of finite Gorenstein injective dimension with grade $n$. Then $M$ is colinked with respect to $\mp_{K}^n(R)$.}
	
	 This is a special case of part (ii).
	\item  {\it 	
	Let $R$ be a Cohen-Macaulay local ring with dualizing module $K$ and let $M$ be an unmixed $R$-module of finite Gorenstein injective dimension with grade $n$. Assume that $0\to Y\to X\overset{\phi}{\to}M\to0$ is a maximal Cohen-Macaulay approximation of $M$, where $\id_R(Y)<\infty$ and $X$ is a maximal Cohen-Macaulay module of finite Gorenstein injective dimension (see \cite{AuBu}). If $I\subseteq\ann_R(M)$ is a complete intersection ideal in $R$ of height $n$ such that $\phi\otimes_RR/I$ is non-injective,
	then $M$ is colinked by $\phi\otimes_RR/I$ with respect to $\G^n_K(\mp_{K})$. }

 This is an immediate consequence of Corollaries \ref{cc1}, \ref{l4}(i) and Example \ref{exa}(ii).
	\end{enumerate}	
\end{eg}
%%%%%%%%%%%%%%%%%%%%%%%%%%%%%%%%%%%%%%%%%%%%%%%%%%%%%%%%%%%%%%
%%%%%%%%%%%%%%%%%%%%%%%%%%
As a generalization of the notion of horizontal linkage due to Martsinkovsky and Strooker ( Remark \ref{r}(I)), an $R$-module $M$ is defined, by Dibaei-Sadeghi, to be \emph{horizontally linked with respect to} $K$ provided that $M\cong\lambda^2_R(K,M)(=\lambda_R(K,\lambda_R(K,M)))$, where  
$\lambda_R(K,-)=\Omega_K\Tr_{K}\Hom_R(K,-)$ \cite[Definition 3.1]{DS2}. 
Note that if $P_1\to P_0\overset{f}{\to} M^{\blacktriangledown}\to0$ is a minimal projective presentation of $M^{\blacktriangledown}$, then $\lambda_R(K,M)=\coker(\Hom(f,K))$ (see \cite[Remark 2.11]{DS2}).

Let $\fa$ be an ideal of $R$ and let $C$ be a semidualizing $R/\fa$-module, $M$ and $N$ are $R$-module. Then  $M$ is said
to be linked to  $N$ by the ideal $\fa$ with respect to $C$ if $\fa\subseteq\ann_R(M)\cap\ann_R(N)$ and $M$ and
$N$ are horizontally linked with respect to $C$ as $R/\fa$-modules \cite[Definition 3.2]{DS2}.
In the following, we show that the notion of linkage with respect to a semidualizing module is a special case of the notion of colinkage.

Recall that an ideal $\fa$ of finite Gorenstein dimension over a local ring $R$ is called \emph{quasi-Gorenstein} provided that there is equality between
Bass numbers $\mu^{i+\depth R}_R(R)=\mu^{i+\depth R/\fa}_{R/\fa}(R/\fa)$ for all $i\geq0$ ( see \cite{AvF} for more details).
%%%%%%%%%%%%%%%%%%%%%%%%%%%%%%%%%%%%%%%%%%%%%%%%%%%%%%%%%%%%%%%%
\begin{prop}\label{ccp}
Let $R$ be a Cohen-Macaulay local ring with dualizing module $K$ and let $\fc$ be a Cohen-Macaulay quasi-Gorenstein ideal of grade $n$. Assume that $M$ is a Cohen-Macaulay $R$-module of finite Gorenstein injective dimension. If $M$ is linked by $\fc$ with respect to $K$, then it is colinked with respect to $\G_K^n(\mp_{K})$. Moreover, there exists $\phi\in\Epi(\G_K^n(\mp_{K}))$ such that $\im\phi=M$ and $\lambda_{\overline{R}}(\overline{K},M)\cong\im\ml_K^n(\phi)$ where $\overline{R}=R/\fc$ and $\overline{K}=\Ext^n_R(\overline{R},K)$.
\end{prop}
\begin{proof}
By definition, $\fc\subseteq\ann_R(M)$ and $M$ is horizontally linked $\overline{R}$-module with respect to $\overline{K}$. Set $(-)^{\dag}=\Hom_{\overline{R}}(\overline{K},-)$. It follows from \cite[Theorem 3.13]{DS2} that $M^{\dag}$ is linked by the ideal $\fc$. In other words, $M^{\dag}$ is a horizontally linked $\overline{R}$-module and so it is stable as an $\overline{R}$-module by \cite[Proposition 3]{MS}. By Theorem \ref{example1}(i) and \cite[Corollary 7.9]{AvF}, $M\in\mathcal{B}_{\overline{K}}(\overline{R})$. In particular, $\nu_{\overline{K}}^{\overline{R}}(M) :\overline{K}\otimes_{\overline{R}} M^{\dag}\to M$ is an isomorphism. 
Next we show that, if $P_1\to P_0\overset{f}{\to}M^{\dag}\to0$ is a minimal $\overline{R}$-projective presentation of $M^{\dag}$, then $M$ is colinked by  $\phi=\nu_{\overline{K}}^{\overline{R}}(M)\circ(f\otimes_{\overline{R}}\overline{K})$. First we show that $\phi\in\Epi(\G_K^n(\mp_{K}))$. 
As $\fc$ is a Cohen-Macaulay ideal of finite Gorenstein dimension, it is $\mathcal{G}$-perfect. It follows from \cite[Lemma 3.16]{Sa} that $\gr_R(M^{\dag})=\gr_R(\fc)=\gr_R(P_0\otimes_RK)$.
As $R/\fc\in\G^n_R(\mp)$, it follows from Lemma \ref{cll}(ii) that $\overline{K}=\Ext^n_R(R/\fc,K)\in\G^n_K(\mp_K)$. Therefore $\phi\in\Epi(\G_K^n(\mp_{K}))$. Note that $\phi$ is not injective.
Otherwise, $\phi$ is an isomorphism and so is $f\otimes_{\overline{R}}\overline{K}$. The commutative diagram
$$\begin{CD}
&&&&P_0@>f>>M^{\dag} &  \\
&&&&@VV{\mu_{\overline{K}}^{\overline{R}}(P_0)}V @VV{\mu_{\overline{K}}^{\overline{R}}(M^\dag)}V  \\
\ \  &&&&\Hom_{\overline{R}}(\overline{K},\overline{K}\otimes_{\overline{R}}P_0)@>g>>\Hom_{\overline{R}}(\overline{K},\overline{K}\otimes_{\overline{R}}M^{\dag}),&\\
\end{CD}$$\\
where $g=\Hom_{\overline{R}}(\overline{K},f\otimes_{\overline{R}}\overline{K})$ is an isomorphism, and the facts that $\mu_{\overline{K}}^{\overline{R}}(P_0)$ and $\mu_{\overline{K}}^{\overline{R}}(M^\dag)$ are isomorphisms, imply that $f$ is an isomorphism. This is a contradiction,  because $M^\dag$ is a stable $\overline{R}$-module. As $K$ is a dualizing module and $M$ is Cohen-Macaulay, $M\in\G^n_K(\mp)$. In particular, by Lemma \ref{ll1}, $\eta_{K}^R(M)$ is an isomorphism.
It follows from Corollary \ref{cc1} that $M$ is colinked by $\phi$ with respect to $\G^n_K(\mp_K)$.

Next we want to show that $\lambda_{\overline{R}}(\overline{K},M)\cong\im\ml_K^n(\phi)$. By \cite[Theorem 7.8]{AvF}, $\overline{K}\cong K\otimes_{R}\overline{R}$. It follows that $\Hom_R(K,X)\cong\Hom_{\overline{R}}(\overline{K},X)$ and $X\otimes_{\overline{R}}\overline{K}\cong X\otimes_RK$ for all $\overline{R}$-module $X$. Hence, by Lemma \ref{G1}, we get the isomorphism 
$\DnK(X)=\Ext^n_R(\Hom_R(K,X),K)\cong\Hom_{\overline{R}}(\Hom_{\overline{R}}(\overline{K},X),\overline{K})$ for all $\overline{R}$-module $X$. Therefore we obtain the following commutative diagram
$$\begin{CD}
&&&&&&&&\\
\ \ &&&&0@>>>\Hom_{\overline{R}}(M^{\dag},\overline{K})@>>> \Hom_{\overline{R}}(P_0,\overline{K}) @>>>\lambda_{\overline{R}}(\overline{K},M) @>>>0 &  \\
&&&&&&@VV{\cong}V @VV{\cong}V \\
\ \  &&&&0@>>>\Ext^n_R(\Hom_R(K,M),K)@>>>\Ext^n_R(\Hom_R(K,P_0\otimes_{\overline{R}}\overline{K}),K)@>>>\im\ml_K^n(\phi)@>>>0.&\\
\end{CD}$$\\
It follows from the above diagram that $\lambda_{\overline{R}}(\overline{K},M)\cong\im\ml_K^n(\phi)$.
\end{proof}
%%%%%%%%%%%%%%%%%%%%%%%%%%%%%%%%%%%%%%%%%%%%%%%%%%%%%%%%%%%%%%%%
\begin{lem}\label{cle}
	Let $\my\subseteq\bk(R)$ be an $n$-coreflexive subcategory with respect to $K$ and let $\phi\in\Epi(\my)$. Assume that $M$ is an $R$-module which is colinked by $\phi$ with respect to $\my$. If $M\in\Delta_K(R)$, then there exists the exact sequence $0\to\DnK(M)\to Y\to\im\ml_K^n(\phi)\to0,$
	for some $Y\in\my$.
\end{lem}
\begin{proof}
This is an immediate consequence of the exact sequence (\ref{cd2}.1) and Lemma \ref{cl}.	
\end{proof}
%%%%%%%%%%%%%%%%%%%%%%%%%%%%%%%%%%%%%%%%%%%%%%
\begin{dfn} Let $\my$ be an $n$-coreflexive subcategory with respect to $K$ and let $m>0$ be an integer. We say that $R$-modules $M$ and $N$ are {\it colinked in $m$ steps with respect to} $\my$ if there are modules $N_0 = M, N_1, \cdots, N_{m-1},N_m=N$ such that $N_i$ and $N_{i+1}$ are directly colinked for all $i=0,\cdots,m-1$. If $m$ is even, then $M$ and $N$ are said to be {\it evenly colinked}. {\it Module coliaison} is the equivalence relation generated by directly colinkage. Its equivalence classes are called coliaison classes. Even colinkage also generates an equivalence relation. Its equivalence classes are called \emph{even coliaison classes}.
\end{dfn}
%%%%%%%%%%%%%%%%%%%%%%%%%%%%%%%%%%%%%%%%%%%%%%%%%%%%%%%%
The proof of the following result is analogous to the proof of Proposition \ref{p}.
\begin{thm}\label{cp}
	Let $\my\subseteq\bk(R)$ be an $n$-coreflexive subcategory with respect to $K$ and let $M, N\in\Delta_K(R)$ be $R$-modules. Then the following statements hold true.
	\begin{enumerate}[\rm(i)]
		\item Let $\my$ be a thick subcategory, $M$ and  $N$ in the same coliaison class. Then $ M\in \my$ if and only if $N\in\my.$
		\item Assume that $\mathcal{Z}$ is a thick subcategory of $\md R$ containing $\my$ and that $M$, $N$ are in the same even coliaison class. Then $M\in \mathcal{Z}$ if and only if $N\in\mathcal{Z}$.
	\end{enumerate}
\end{thm}
\begin{proof}
This follows from Lemma \ref{cle}.	
\end{proof}
%%%%%%%%%%%%%%%%%%%%%%%%%%%%%%%%%%%%%%%%%%%%%%%%%
\begin{cor}\label{cccc1}
	Let $M$, $N\in\Delta_K(R)$ be $R$-modules in the same even coliaison class with respect to $\mathcal{P}^n_K(R)$. Assume that $X$ is an $R$-module. Then the following statements hold true.
	\begin{enumerate}[\rm(i)]
    	\item $M\in\bk(R)$ if and only if  $N\in\bk(R)$.
		\item $\pkd_R(M)<\infty$ if and only if $\pkd_R(N)<\infty$. In particular, if $K$ is a dualizing module, then $\id_R(M)<\infty$ if and only if $\id_R(N)<\infty$.
		\item $\Ext^{i\gg0}_{\mp_K}(M,X)=0$ if and only if $\Ext^{i\gg0}_{\mp_K}(N,X)=0$ (see \cite{TW}).
		\item $\gpkd_R(M)<\infty$  if and only if $\gpkd_R(M)<\infty$.
	\end{enumerate}
\end{cor}
\begin{proof}
	As every module of finite $\mp_K$-projective dimension satisfies in all of the above conditions (see Theorem \ref{example1}), the assertion follows from Theorem \ref{cp} and Example \ref{exaa}.
\end{proof}
%%%%%%%%%%%%%%%%%%%%%%%%%%%%%%%%%]
The following is the dual version of Theorem \ref{cc}.
\begin{thm}\label{ct}
	Let $M$ and $N$ be $R$-modules in the same coliaison class
	with respect to $\mp^n_K(R)$. The following statements hold true.
	\begin{enumerate}[\rm(i)]
		\item $M$ is $\mp_K$-perfect if and only if $N$ is $\mp_K$-perfect.
		\item If $K$ is dualizing, then $M$ is $\G(\mp_K)$-perfect if and only if $N$ is $\G(\mp_K)$-perfect.
	\end{enumerate}	
\end{thm}
\begin{proof}
	(i). Without loss of generality, we may assume that $M$ is directly colinked to $N$. Hence there exists $\phi\in\Epi(\mp^n_K(R))$ such that
	$M=\im\phi$ and $N\cong\im\ml_K^n(\phi)$. Assume that $M$ is $\mp_K$-perfect of grade $n$ and consider the exact sequence
	\begin{equation}\tag{\ref{ct}.1}
	0\to Z\to Y\overset{\phi}{\to} M\to0,\ \text{where}\ Y\in\mp^n_K(R). 
	\end{equation}
	 It follows that $Z$ is $\mp_K$-perfect of grade $n$ and so $Z\in\bk(R)$ by Theorem \ref{example1}(ii). In particular, $\Ext^1_R(K,Z)=0$. Applying the functor $(-)^{\blacktriangledown}=\Hom_R(K,-)$ to (\ref{ct}.1) implies the exact sequence $0\to Z^{\blacktriangledown}\to Y^{\blacktriangledown}\overset{\phi^{\blacktriangledown}}{\to} M^{\blacktriangledown}\to0$ which induces the long exact sequence
	\begin{equation}\tag{\ref{ct}.2}
	 0\to\Ext^n_R(M^{\blacktriangledown},K)\to\Ext^n_R(Y^{\blacktriangledown},K)\rightarrow\Ext^n_R(Z^{\blacktriangledown},K)\to\Ext^{n+1}_R(M^{\blacktriangledown},K)\to\cdots.
	\end{equation}
	As $\pd_R(M^{\blacktriangledown})=\pkd_R(M)=n$ by \cite[Theorem 2.11]{TW},  $\Ext^{n+1}_R(M^{\blacktriangledown},K)=0$. It follows from  the exact sequence (\ref{ct}.2) and Theorem \ref{t2}(ii) that $$N\cong\im\ml_K^n(\phi)\cong\Ext^n_R(Z^{\blacktriangledown},K)=\mathcal{D}^n_K(Z)\in\mp^n_K(R).$$ The converse follows from the symmetry. The second part can be proved analogously.
\end{proof}
%%%%%%%%%%%%%%%%%%%%%%%%%%%%%%%%%%%%%%%%%%%%%%%%%%%%%%%%
For a subcategory $\mathcal{Z}$ of $\md R$, we denote by $\mathcal{Z}^{\blacktriangledown}=\{M^{\blacktriangledown}\mid M\in\mathcal{Z}\}$ and $\mathcal{Z}^{\otimes_K}=\{M\otimes_RK\mid M\in\mathcal{Z}\}$.

\begin{dfn}
Let $\mx, \my$ be subcategories of $\md R$. The ordered pair $(\mx,\my)$ is called an $n$-\emph{adjoint pair of subcategories with respect to} $K$, provided that the following conditions hold.
\begin{enumerate}[\rm(i)]
\item $\mx$ is an $n$-reflexive subcategory with respect to $R$.
\item $\my$ is an $n$-coreflexive subcategory with respect to $K$.
\item $\mx^{\otimes_K}\subseteq\my\subseteq\bk(R)$ and $\my^{\blacktriangledown}\subseteq\mx\subseteq\ak(R)$.
\end{enumerate}
\end{dfn}
Here are some examples of $n$-adjoint pairs of subcategories with respect to $K$.
\begin{ex}\label{ex1}
The following statements hold true.
\begin{enumerate}[\rm(i)]
	\item The ordered pair $(\mp^n(R),\mp_K^n(R))$ is an $n$-adjoint pair of subcategories with respect to $K$.
	\item If $R$ is local and $K$ is a dualizing module, then $(\G^n_R(\mp),\G^n_K(\mp_K))$ is an $n$-adjoint pair of subcategories with respect to $K$.
\end{enumerate}
\end{ex}
\begin{proof}
Part (i) follows from Theorem \ref{example1}(ii), (iii) and Example \ref{exa}(i). Part (ii) follows from Theorem \ref{example1}(i), (iv), (vii), Example \ref{ex}(i) and  Example \ref{exa}(ii). 	
\end{proof}
%%%%%%%%%%%%%%%%%%%%%%%%%%%%%%%%%%%%%%%%%%%%%%%%%%%%%%%%%%%%%%%%%%%
We denote by $\nabla_K(R):=\{M\in\md R\mid \mu_K^R(M):M\to\Hom_R(K, M\otimes_RK) \text{ is an isomorphism }\}$. Note that $\ak(R)\subseteq\nabla_K(R)$.
The following theorem is the main result of this section which is a  generalization of \cite[Theorem B]{DS2}.
%%%%%%
\begin{thm}\label{ctt} Let $(\mx,\my)$ be an $n$-adjoint pair of subcategories with respect to $K$.
	There is an adjoint equivalence of categories
	\begin{center}
		$\left\{\begin{array}{llll}
		M\in\nabla_K(R) &{\bigg |} \begin{array}{lll} M\mathrm{\ is\ linked\ with } \\
		\mathrm{\ \ \  respect\ to } \ \mx \\
		\end{array} \end{array}
		\right\}
		\begin{array}{ll}\overset{-\otimes_{R}K}{\xrightarrow{\hspace*{2cm}}}\\ \underset{\Hom_{R}(K,-)}{\xleftarrow{\hspace*{2cm}}}\\ \end{array}
		\left\{\begin{array}{lll}
		N\in\Delta_K(R) &{\bigg |} \begin{array}{lll} N\mathrm{\ is\ colinked\ with }\\
		\mathrm{\ \ \ respect\ to }	\ \my\\ \end{array}\end{array}
		\right\}.$
	\end{center}	
\end{thm}
\begin{proof}
	Let $M\in\nabla_K(R)$ be an $R$-module which is linked by a reflexive homomorphism $\phi\in\Epi(\mx)$.
	Consider the exact sequence $0\to C\to X\overset{\phi}{\to} M\to0$, where $X\in\mx$. It induces the exact sequence
	$X\otimes_{R}K\overset{\phi'}{\to} M\otimes_RK\to0$,
	where $\phi'=\phi\otimes_RK$. As $(\mx,\my)$ is an $n$-adjoint pair with respect to $K$, we have $X\otimes_{R}K\in\my$. It follows that $\phi'\in\Epi(\my)$. Note that neither $\phi$ nor $\phi'$ is injective.
	Set $N=M\otimes_RK$. As $M\in\nabla_K(R)$, we have $N\in\Delta_K(R)$. Next we prove that $N$ is colinked by $\phi'$ with respect to $\my$. By Corollary \ref{cc1}, it is enough to show that $\eta_{K}^R(N)$ is injective. Let ${\bf x}\subseteq\ann_R(M)$ be an ideal generated by a regular sequence of length $n$. Set $S=R/{\bf x}$ and $\overline{K}=K\otimes_RS$. As $M$ is linked by $\phi$, we have $\eta_{R}^R(M)$ is injective by Corollary \ref{c}. It follows from Lemma \ref{l5}(ii) that $\Ext^1_S(\Tr_S M,S)=0$. In other words, $M$ is a first syzygy $S$-module, i.e. there exists an exact sequence $0\to M\to P\to Z\to0$, where $P$ is a projective $S$-module.
	Applying the functor $-\otimes_S\overline{K}$ gives the following exact sequence
    \begin{equation}\tag{\ref{ctt}.1}
	0\to\Tor_1^S(Z,\overline{K})\to M\otimes_S\overline{K}\to P\otimes_S\overline{K}\to Z\otimes_RK\to0.
	\end{equation}
 Applying the functor $\Hom_S(\overline{K},-)$ to (\ref{ctt}.1) implies the commutative diagram
		$$\begin{CD}
	&&&&&&&&\\
	\ \ &&&&0@>>>\Hom_S(\overline{K},\Tor_1^S(Z,\overline{K})) @>>> \Hom_S(\overline{K},M\otimes_S\overline{K})@>>>\Hom_S(\overline{K},P\otimes_S\overline{K})&  \\
	&&&&&&&&@VV{\mu_K^S(M)}V @VV{\mu_K^S(P)}V \\
	\ \  &&&&&&0@>>> M@>>>P &\\
	\end{CD}$$\\
	with exact rows. As $M\in\nabla_K(R)$, we have $\mu_K^R(M)$ is an isomorphism and so is $\mu_{\overline{K}}^S(M)$. Also $P\in\mathcal{A}_{\overline{K}}(S)$ and so $\mu_{\overline{K}}^S(P)$ is an isomorphism. It follows from the above commutative diagram that $\Hom_S(\overline{K},\Tor_1^S(Z,\overline{K}))=0$. Therefore $\Tor_1^S(Z,\overline{K})=0$ and, by the exact sequence (\ref{ctt}.1),
	 we get the exact sequence $0\to M\otimes_S\overline{K}\to P\otimes_S\overline{K}$.
	Note that $N\cong M\otimes_S\overline{K}$. Set $Y=P\otimes_S\overline{K}$. As $\overline{K}$ is a semidualizing $S$-module, we have $\eta_{\overline{K}}^S(Y)$ is an isomorphism. Hence, it follows from the commutative diagram
	\[\xymatrix{
	 &N\ar@{->}[d]^{\eta_{\overline{K}}^S(N)}\ar@{^{(}->}[r] &Y\ar@{->}[d]^{\eta_{\overline{K}}^S(Y)} 	\\
	 &\Hom_S(\Hom_S(N,\overline{K}),\overline{K})\ar[r] &\Hom_S(\Hom_S(Y,\overline{K}),\overline{K}),	
	}\]
	that $\eta_{\overline{K}}^S(N)$ is injective and so is
    $\eta_{K}^R(N)$ (see Lemma \ref{l5}).

	Conversely, assume that $N\in\Delta_K(R)$ is a module which is linked by a coreflexive homomorphism $\psi\in\Epi(\my)$. Consider the exact sequence $0\to D\to Y\overset{\psi}{\to} N\to0$, where $Y\in\my\subseteq\bk(R)$. It follows that $\Ext^1_R(K,Y)=0$. Hence, 
	applying the functor $(-)^{\blacktriangledown}=\Hom_R(K,-)$ gives the exact sequence
	\begin{equation}\tag{\ref{ctt}.2}
	0\to D^{\blacktriangledown}\to Y^{\blacktriangledown}\overset{\psi^{\blacktriangledown}}{\to} N^{\blacktriangledown}\to\Ext^1_R(K,D)\to0.
	\end{equation}
	Applying the functor $K\otimes_R-$ to the exact sequence (\ref{ctt}.2) implies the commutative diagram
	$$\begin{CD}
	&&&&&&&&\\
	\ \ &&&&K\otimes_RY^{\blacktriangledown} @>>> K\otimes_RN^{\blacktriangledown}@>>>K\otimes_R\Ext^1_R(K,D)@>>>0&  \\
	&&&&@VV{\nu_K^R(Y)}V @VV{\nu_{K}^R(N)}V \\
	\ \  &&&&Y@>>>N @>>>0,&\\
	\end{CD}$$\\
	with exact rows. As $N\in\Delta_K(R)$, we have $\nu_K^R(N)$ is an isomorphism. Also, since $Y\in\bk(R)$, we have $\nu_K^R(Y)$ is an isomorphism. It follows from the above commutative diagram that $K\otimes_R\Ext^1_R(K,D)=0$ and so $\Ext^1_R(K,D)=0$. Hence, by
	the exact sequence (\ref{ctt}.2) we obtain the following exact sequence $Y^{\blacktriangledown}\overset{\psi^{\blacktriangledown}}{\to}N^{\blacktriangledown}\to0$.
	As the pair $(\mx,\my)$ is an $n$-adjoint pair of subcategories with respect to $K$, we have $Y^{\blacktriangledown}\in\mx$.
	It follows that $\psi^{\blacktriangledown}\in\Epi(\mx)$. Since $\psi$ is not injective neither is $\psi^{\blacktriangledown}$.
	Set $M=N^{\blacktriangledown}$. Note that
	$M\in\nabla_K(R)$.
	Next we prove that $M$ is linked by $\psi^{\blacktriangledown}$ with respect to $\mx$. By Corollary \ref{c}, it is enough to show that $\eta_{R}^R(M)$ is injective. Let ${\bf y}\subseteq\ann_R(M)$ be an ideal generated by a regular sequence of length $n$. Set $T=R/{\bf y}$ and $K'=K\otimes_RT$. As $N$ is colinked by $\psi$, we have $\eta_{K}^R(N)$ is injective by Corollary \ref{cc1}. It follows from Lemma \ref{l5}(ii) that $\Ext^1_T(\Tr_{K'} N,K')=0$. In other words, $N$ is a first $K'$-syzygy $T$-module, i.e. there exists an exact sequence $0\to N\to Y\to Z\to0$, where $Y=P\otimes_TK'$ for a projective $T$-module $P$.  Applying the functor $\Hom_T(K',-)$ to the above exact sequence gives the exact sequence $0\to \Hom_T(K',N)\to\Hom_T(K',Y)$.
	Note that $M=\Hom_R(K,N)\cong\Hom_T(K',N)$
	by the adjoint isomorphism. Also, since $P\in\mathcal{A}_{K'}(T)$, we have $P\cong\Hom_T(K',K'\otimes_TP)$.
	Hence we obtain the exact sequence $0\to M\to P$. 
	It follows that $\eta_{T}^T(M)$ is injective and so is $\eta_R^R(M)$ by Lemma \ref{l5}.	
\end{proof}
%%%%%%%%%%%%%%%%%%%%%%%%%%%%%%%%%%%%%%%%%%%%%%%%%%%%%%%%%%
The following result is an immediate consequence of
Theorem \ref{ctt} and Example \ref{ex1}.
\begin{cor}\label{cctt}
	The following statements hold true.
	\begin{enumerate}[\rm(i)]
		\item There is an adjoint equivalence of categories
		\begin{center}		
			$\left\{\begin{array}{llll}
			M\in\ak(R){\bigg |} \begin{array}{lll} M\mathrm{\ is\ linked\ with } \\
			\mathrm{\ \ respect\ to } \ \mp^n(R) \\
			\end{array} \end{array}
			\right\}
			\begin{array}{ll}\overset{-\otimes_{R}K}{\xrightarrow{\hspace*{2cm}}}\\ \underset{\Hom_{R}(K,-)}{\xleftarrow{\hspace*{2cm}}}\\ \end{array}
			\left\{\begin{array}{lll}
			N\in\bk(R){\bigg |} \begin{array}{lll} N\mathrm{\ is\ colinked\ with }\\
			\mathrm{\ \ respect\ to }	\ \mp_K^n(R)\\ \end{array}\end{array}
			\right\}.$
		\end{center}	
		\item If $K$ is a dualizing module, then there is an adjoint equivalence of categories
		\begin{center}		
			$\left\{\begin{array}{llll}
			M\in\ak(R){\bigg |} \begin{array}{lll} M\mathrm{\ is\ linked\ with } \\
			\mathrm{\ \ respect\ to } \ \G^n_R(\mp) \\
			\end{array} \end{array}
			\right\}
			\begin{array}{ll}\overset{-\otimes_{R}K}{\xrightarrow{\hspace*{2cm}}}\\ \underset{\Hom_{R}(K,-)}{\xleftarrow{\hspace*{2cm}}}\\ \end{array}
			\left\{\begin{array}{lll}
			N\in\bk(R){\bigg |} \begin{array}{lll} N\mathrm{\ is\ colinked\ with }\\
			\mathrm{\ \ respect\ to }	\ \G^n_K(\mp_K)\\ \end{array}\end{array}
			\right\}.$
		\end{center}	
	\end{enumerate}
\end{cor}
%%%%%%%%%%%%%%%%%%%%%%%%%%%%%%%%%%%%%%%%%%%%%%%%%%%%%%%%
%%%%%%%%%%%%%%%%%%%%%%%
%%%%%%%%%%%%%%%%%%%%%%%%%%%%%%%%%%%%%%%%%%%%%%%%%%%%%%%%
%%%%%%%%%%%%%%%%%%%%%%%%%%%%%%%%%%%%%%%%%%%%%%%%%%%%%%%%%%%
%%%%%%%%%%%%%%%%%%%%%%%%%%%%%%%%%%%%%%%%%%%%%%%
%%%%%%%%%%%%%%%%%%%%%%%%%%%%%%%%%%%%%%%%%%%%%%%%%%%%%%%%%%%%%
The following is a dual version of Proposition \ref{t3}.
\begin{thm}\label{ct2}
Let $(\mx,\my)$ be an $n$-adjoint pair of subcategories with respect to $K$ such that $\mx\subseteq\G^n_K(\mp)$.
Assume that $t>1$ is an integer such that $\id_{R_\fp} (K_\fp)<\infty$ for all $\fp\in X^{n+t-1}(R)$ and that  $M$, $N$ are $R$-modules which are directly colinked with respect to $\my$. If $M\in\bk(R)$, then the following statements are equivalent.
	\begin{enumerate}[\rm(i)]
		\item $\depth_{R_\fp}(M_\fp)\geq\min\{t,\depth R_\fp-n\}$ for all $\fp\in\Spec R$.
		\item $\Ext^i_R(N,K)=0$ for all $i$, $n+1\leq i\leq n+t-1.$
	\end{enumerate}	
\end{thm}
\begin{proof}
	By definition, there exists $\phi\in\Epi(\my)$ such that $M=\im\phi$ and $N\cong\im\ml_K^n(\phi)$.
	It follows from Theorem \ref{ctt} that $M^{\blacktriangledown}$ is linked by $\phi^{\blacktriangledown}$ with respect to $\mx$.
	As $\mx\subseteq\G^n_K(\mp)$, it follows from Corollary \ref{c1} that $M^{\blacktriangledown}$ is also linked by $\phi^{\blacktriangledown}$ with respect to $\G^n_K(\mp)$. By Proposition \ref{t3}, 
	\[\begin{array}{rl}\tag{\ref{ct2}.1}
	\depth_{R_\fp}(M^{\blacktriangledown}_\fp)\geq\min\{t,\depth R_\fp-n\} \text{ for all } \fp\in\Spec R\\
	\Longleftrightarrow\Ext^i_R(\im\L_K^n(\phi^\blacktriangledown),K)=0 \text{ for all } i,\ n<i<n+t.
	\end{array}\]
	By using the fact that $M\in\bk(R)$ and Definitions \ref{d2} and \ref{cd2}, it is easy to see that $\ml_K^n(\phi)=\L_K^n(\phi^{\blacktriangledown})$.
	Hence we obtain the following isomorphisms
	\begin{equation}\tag{\ref{ct2}.2}
	\Ext^i_R(N,K)\cong\Ext^i_R(\im\ml_K^n(\phi),K)\cong\Ext^i_R(\im\L_K^n(\phi^{\blacktriangledown}),K) \text{ for all } i.
	\end{equation}
	Now the assertion follows from (\ref{ct2}.1), (\ref{ct2}.2) and Theorem \ref{example1}(vi).
\end{proof}
%%%%%%%%%%%%%%%%%%%%%%%%%%%%%%%%%%%%%%%%%%%%%%%%%%%%%%%%
\begin{cor}
Assume that $t>1$ is an integer such that $\id_{R_\fp} (K_\fp)<\infty$ for all $\fp\in \X^{n+t-1}(R)$ and that  $M$, $N$ are $R$-modules which are directly colinked with respect to $\mp_K^n(R)$. If $M\in\bk(R)$, then the following statements are equivalent.
\begin{enumerate}[\rm(i)]
	\item $\depth_{R_\fp}(M_\fp)\geq\min\{t,\depth R_\fp-n\}$ for all $\fp\in\Spec R$.
	\item $\Ext^i_R(N,K)=0$ for all $i$, $n+1\leq i\leq n+t-1.$
\end{enumerate}		
\end{cor}
\begin{proof}
First note that every perfect module is also $\gkkkd$-perfect. In particular, $\mp^n(R)\subseteq\G^n_K(\mp)$. Now the assertion follows from Theorem \ref{ct2} and Example \ref{ex1}(i).	
\end{proof}
%%%%%%%%%%%%%%%%%%%%%%%%%%%%%%%%%%%%%%%%%%%%%%%%%
The following is a generalization of \cite[Theorems A, C]{DS2}.
\begin{cor}\label{ccc2}
Let $(R,\fm)$ be a Cohen-Macaulay local ring of dimension $d$ and let $K$ be a dualizing module. Assume that $M$ and $N$ are $R$-modules of finite Gorenstein injective dimension which are in the same coliaison class with respect to $\G^n_K(\mp_K)$. Then the following statements hold true.
\begin{enumerate}[\rm(i)]
	\item $M$ satisfies the Serre's condition $(S_t)$ if and only if
	$\hh^i_{\fm}(N)=0$ for all $i$, $d-n-t<i<d-n.$
	\item $M$ is Cohen-Macaulay if and only if $N$ is so.
	\item $M$ is generalized Cohen-Macaulay if and only if $N$ is so.
\end{enumerate}
 \end{cor}
\begin{proof}
(i). By Corollary \ref{cctt}(ii), $M^{\blacktriangledown}$ is linked with respect to $\G^n_R(\mp)$. It follows from Proposition \ref{p2} and Theorem \ref{example1}(vi) that $\dim_{R_\fp}(M_\fp)=\depth R_\fp-n$ for all $\fp\in\Supp_R(M)$. Now the assertion follows from Theorem \ref{ct2}, Example \ref{ex1}(ii) and the local duality theorem. Part (ii) is a special case of part (i).

(iii). By \cite[Lemmas 1.2, 1.4]{T}, $M$ is generalized Cohen-Macaulay if and only if $M_\fp$ is a maximal Cohen-Macaulay $R_\fp$-module for all $\fp\in\Spec(R)\setminus\{\fm\}.$ Now the assertion is clear by part (ii).
\end{proof}
%%%%%%%%%%%%%%%%%%%%%%%%%%%%%%%%%%%%
%%%%%%%%%%%%%%%%%%%%%%%%%%%%%%%%%%%%%%%%%%%%%%%%%
\begin{prop}\label{cttt}
	Let $(\mx,\my)$ be an $n$-adjoint pair of subcategories with respect to $K$ such that $\mx\subseteq\G^n_K(\mp)$ and $\my\subseteq\G^n_K(\mp_K)$.
	Assume that $M, N\in\bk(R)$ are $R$-modules in the same even coliaison class with respect to $\my$. Then 
     $\mathcal{D}^i_K(M)\cong\mathcal{D}^i_K(N)$ 
	for all $i>n$.
    \end{prop}
\begin{proof}
	Without loss of generality, we may assume that $M$ is colinked to $N$ in two steps: $M\overset{c}{\sim}L \overset{c}{\sim}N$, for some $R$-module $L$. Hence there exist non-injective epimorphisms $\phi:X\twoheadrightarrow M$ and $\psi:Y\twoheadrightarrow N$ in $\Epi(\my)$ 
    such that $\im\ml_K^n(\phi)\cong L\cong \im\ml_K^n(\psi)$.
    By Theorem \ref{ctt}, $M^{\blacktriangledown}$ is linked by $\phi^{\blacktriangledown}$ with respect to $\mx$.
    As $\mx\subseteq\G^n_K(\mp)$, it follows from Corollary \ref{c1} that $M^{\blacktriangledown}$ is also linked by $\phi^{\blacktriangledown}$ with respect to $\G^n_K(\mp)$.
	Applying the functor $(-)^{\blacktriangledown}=\Hom_R(K,-)$ to the exact sequence $0\to\ker\phi\to X\overset{\phi}{\to} M\to0$ and using the fact that $\ker\phi\in\bk(R)$ give the following exact sequence 
	$0\to(\ker\phi)^{\blacktriangledown}\to X^{\blacktriangledown}\overset{\phi^\blacktriangledown}{\to} M^{\blacktriangledown}\to0$. Now it is easy to see that  $\L^n_K(\phi^{\blacktriangledown})=\ml_K^n(\phi)$ (see Definitions \ref{d2} and \ref{cd2}). Therefore  $\im\L^n_K(\phi^\blacktriangledown)=\im\ml^n_K(\phi)\cong L$. Similarly, one can prove that $N^{\blacktriangledown}$ is linked by $\psi^{\blacktriangledown}$ with respect to $\G^n_K(\mp)$ and $\im\L^n_K(\psi^\blacktriangledown)=\im\ml^n_K(\psi)\cong L$. Thus we have 
	$M^{\blacktriangledown}\sim L\sim N^{\blacktriangledown}$ and so by Theorem \ref{ttt}(i) we have the following isomorphism
	$$\mathcal{D}^i_K(M)=\Ext^i_R(M^{\blacktriangledown},K)\cong\Ext^i_R(N^{\blacktriangledown},K)=\mathcal{D}^i_K(N) \text{ for all } i>n.$$
	\end{proof}
%%%%%%%%%%%%%%%%%%%%%%%%%%%%%%%%%%%%%%%%%%%%%%%%%%%%%%
We end this section by proving the fact that homological dimensions with respect to $K$ are preserved in even module coliaison classes.
%%%%%%%%%%%%%%%%%%%%%%%%%%%%%%%%%%%%%%%%%%%%%%%%%%%%%%%
\begin{cor}\label{ccc3} Let $M$ and $N$ be $R$-modules contained in $\bk(R)$ which are in the same even coliaison class with respect to $\mp_K^n(R)$. Then the following statements hold true.
	\begin{enumerate}[\rm(i)]
		\item $\pkd_R(M)=\pkd_R(N)$.	
		\item If $R$ is local and $K$ is a dualizing module, then $\gpkd_R(M)=\gpkd_R(N)$.
	\end{enumerate}
\end{cor}
\begin{proof}
(i). By Corollary \ref{cccc1}(ii), we may assume that $M$ and $N$ have finite $\mp_K$-dimensions. Hence, by \cite[Theorem B]{TW} and Proposition \ref{cttt} we obtain the equalities
\[\begin{array}{rllll}
\pkd_R(M)&=\pd_R(M^{\blacktriangledown})\\
&=\sup\{i\mid\Ext^i_R(M^\blacktriangledown,K)\neq0\}\\&=\sup\{i\mid\Ext^i_R(N^{\blacktriangledown},K)\neq0\}\\
&=\pd_R(N^{\blacktriangledown})\\
&=\pkd_R(N).
\end{array}\]
(ii). By Corollary \ref{cccc1}(iv), we may assume that $M$ and $N$ have finite $\G(\mp_K)$-dimension. Set $d=\dim R$. It follows from  Proposition \ref{cttt} and \cite[Corollary 3.5.11]{BH} that
\[\begin{array}{rlll}
d-\depth_R(M^\blacktriangledown)&=\sup\{i\mid\Ext^i_R(M^\blacktriangledown,K)\neq0\}\\
&=\sup\{i\mid\Ext^i_R(N^{\blacktriangledown},K)\neq0\}\\&=d-\depth_R(N^\blacktriangledown).
\end{array}\]
Now the assertion follows from Theorems \ref{G3}(iii) and \ref{example1}(iv), (vi) and the above equality.
\end{proof}

%%%%%%%%%%%%%%%%%%%%%%%%%%%%%%%%%%%%%%%%%%%
\section*{Acknowledgment}
The authors are grateful to the anonymous referee for his/her careful reading and helpful comments.
%%%%%%%%%%%%%%%%%%%%%%%%%%%%%%%%%
\bibliographystyle{amsplain}

\end{document}